\documentclass[a4paper]{article}

\usepackage{amsmath,amssymb,amsthm}
\usepackage{graphicx}
\usepackage[english]{babel}
\usepackage[utf8]{inputenc}
\usepackage[T1]{fontenc}
\usepackage{mathtools}
\usepackage[margin=1cm]{caption}
\usepackage{bbm}
\usepackage[margin=3cm]{geometry}
\usepackage{xcolor}
\usepackage{enumitem}
\usepackage{hyperref}
\usepackage{tikz}
\usetikzlibrary{patterns}
\usetikzlibrary{intersections}
\usepackage{curves}
\usetikzlibrary{decorations.pathmorphing}
\allowdisplaybreaks

\newcommand{\N}{\mathbb{N}}
\newcommand{\Z}{\mathbb{Z}}

\newcommand{\R}{\mathbb{R}}

\newcommand{\E}{\mathbb{E}}
\newcommand{\ud}{\,\mathrm{d}}
\newcommand{\dist}{\mathsf{d}}
\newcommand{\e}{{\mathrm{e}}}

\newcommand{\pp}{\mathbb{P}}
\newcommand{\PP}{\mathbf{P}}
\newcommand{\A}{{\mathbf{a}}}
\newcommand{\oA}{{\overline{\A}}}
\newcommand{\V}{\mathcal{V}}
\newcommand{\W}{\mathcal{W}}
\newcommand{\CC}{\mathcal{C}}
\newcommand{\F}{\mathcal{F}}
\newcommand{\ZZ}{\mathcal{Z}}

\newcommand{\Ev}{\mathcal{E}}

\newcommand{\I}{\mathbbm{1}}

\newcommand{\T}{\mathsf{T}}
\newcommand{\Rone}{\mathsf{R}^{(1)}}
\newcommand{\Rtwo}{\mathsf{R}^{(2)}}
\newcommand{\Rscale}{\mathsf{R}}
\newcommand{\ep}{\varepsilon}

\DeclareMathOperator{\Cov}{Cov}
\DeclareMathOperator{\Var}{Var}
\DeclareMathOperator{\argmax}{argmax}

\numberwithin{equation}{section}
\numberwithin{figure}{section}

\renewcommand{\notin}{\not\in}

\newtheorem{theorem}{Theorem}[section]

\newtheorem{lemma}[theorem]{Lemma}
\newtheorem{assumption}[theorem]{Assumption}
\newtheorem{question}[theorem]{Question}
\newtheorem{remark}[theorem]{Remark}
\theoremstyle{definition}

\newcommand{\corO}{\textcolor{purple}}

\begin{document}

\title{The maximum of log-correlated Gaussian fields in random environment}
\author{Florian Schweiger
\thanks{
Department of Mathematics, 
Weizmann Institute of Science,
Rehovot 76100, Israel.\newline
Email \texttt{florian.schweiger@weizmann.ac.il}}
\and 
Ofer Zeitouni
\thanks{
Department of Mathematics, 
Weizmann Institute of Science,
Rehovot 76100, Israel,\newline
and Courant Institute,
New York University,
251 Mercer Street, New York, NY 10012, USA.\newline
Email \texttt{ofer.zeitouni@weizmann.ac.il}}
}
\date{\today}
\maketitle
\begin{abstract}
We study the distribution of the maximum of 
a large class of Gaussian fields indexed by a box $V_N\subset\Z^d$ 
and possessing
logarithmic correlations up to local defects that are sufficiently rare. 
Under appropriate assumptions that generalize those in Ding, Roy and Zeitouni 
\cite{DRZ17}, we show that asymptotically,
the centered maximum of the field has a randomly-shifted Gumbel distribution. We prove that
the two dimensional Gaussian free field on a super-critical
bond percolation cluster with $p$
close enough to $1$, as well as the Gaussian free field in i.i.d. bounded conductances, fall under the assumptions of our general theorem.
\end{abstract}

\section{Introduction and statement of results}
\subsection{Overview}
The study of logarithmically correlated Gaussian fields (LCGFs) has its roots in 
two-dimensional quantum field theory and the study of Gaussian multiplicative
chaos \cite{K85, KPZ88}. The last decade saw rapid development,
driven both by these motivations \cite{S07,DS11,RV14}
as well as the fact that those Gaussian field
arise as natural limits in a variety of contexts including
random matrix theory
\cite{BMP21,CN18,CFLW19}, planar random walks and Brownian motion \cite{DPRZ04,J20} and number theory \cite{FHK12, ABR20}.

Within the theory of LCGFs,
the study of extremes 
plays an important role and 
highlights the link with the study of Branching
Random Walks (BRW); this link is already made in the seminal \cite{K85},
and played an important role in
the study of the maxima of LCGFs,
see \cite{BZ12, BDZ16b, BL18}
and the lecture notes \cite{B20,Z16}.

A general framework for the study of the maximum of (discretely indexed)
LCGFs is put forward in \cite{DRZ17}.
Building on the approach of \cite{BDZ16}, it provided an axiomatic framework that
ensures that the maximum of such a field, properly centered, converges 
in distribution as the size of the box on which it is defined increases.
Checking the hypotheses for that framework is often a non-trivial task, 
and a recent success is the evaluation of the limit for the 4D membrane model,
see \cite{S20}.

In all of the above mentioned examples, 
the covariance of the LCGF is uniformly
controlled by the logarithm. We introduce notation to explain this point. 
Consider the lattice $\Z^d$ with nearest neighbor edges,
and let $E(\Z^d)$ denote the collection of edges. 
For $w\in\Z^d$ let $V_N(w)=(w+[0,N-1]^d)\cap\Z^d$ be the lattice cube of sidelength $N$ with lower left corner $w$, and 
for $\delta>0$ set 
\begin{equation}
\label{eq:VNd}
V^\delta_N(w):=\left\{v\in  V_N(w)\colon\dist(v,\partial^+ V_N(w)\ge\delta N\right\},
\end{equation}
where $\partial^+ V_N(w)$ is the outer boundary of $V_N(w)$, that is those points $v\in \Z^d\setminus V_N(w)$ at  distance $1$ from $V_N(w)$.

For $N\in\N$ and $w\in\Z^d$, we consider a centered Gaussian field $\varphi^{N,w}=(\varphi^{N,w}_v)_{v\in V_N(w)}$ on $V_N(w)$ with variance and
covariance 
\begin{equation}
\label{eq-cov}
\Cov^{N,w}(v,u)=\E\varphi^{N,w}_v\varphi^{N,w}_u, \quad v,u\in V_N, 
\quad \Var^{N,w}(u)=\Cov^{N,w}(u,u),
\end{equation}
and
law $\pp^{N,w}$. Throughout,  when $w=0$ we
omit it from the notation. We say that  $(\varphi^{N}_v)_{v\in V_N}$
is a LCGF if for any $\delta>0$ there exists $C_\delta<\infty$ so that
\begin{equation}
  \label{e:LCGF}\left|\Cov^N(v,u)
  %\E\varphi^{N}_v\varphi^{N}_u
  -\log N+\log_+|u-v|\right|\le C_\delta, \quad 
  u,v\in V_N^\delta,
\end{equation}
  where $\log_+ x:=\max(\log x,0)$.

  Our goal in this paper is to extend the theory of the extrema of LCGFs in the direction of Gaussian
  fields defined in a \textit{random environment}, for which 
  \eqref{e:LCGF} does not hold uniformly. We describe below in
  Theorem \ref{t:mainthm} a general result in that direction, but as the result is rather technical, we present first two main applications. The first application is the Gaussian free field with i.i.d uniformly elliptic conductances, i.e. the Gaussian field whose covariance corresponds to random walk in i.i.d. bounded random conductances. Throughout, we consider $\Z^2$ as a graph with nearest-neighbour edges, and write $E(G)$ for the edges of a graph $G$.
  \begin{theorem}
    \label{t:random_conductances}
Fix  $0<\Lambda^-\le\Lambda^+$ and  let $\PP$ be an i.i.d.
probability measure on $[\Lambda^-,\Lambda^+]^{E(\Z^2)}$.

For each sample $\A$ from $\PP$  construct the family of Gaussian fields $\varphi^{\A,N}$ on $V_N$ with law 
\begin{align*}
	&\pp^{\A,N}(\mathrm{d}\varphi^{\A,N})\\
	&=\frac{1}{Z_\A^{V_N}}\exp\left(-\frac12\sum_{\{u,v\}\in E(\Z^2)}\A(\{u,v\})(\varphi^{\A,N}_u-\varphi^{\A,N}_v)^2\right)\prod_{v\in V_N}\ud\varphi^{\A,N}_v\prod_{v\notin V_N}\ud\delta_0(\mathrm{d}\varphi^{\A,N}_v).
\end{align*}
Then there is a deterministic constant $\oA\in[\Lambda^-,\Lambda^+]$ such that 
for $\PP$-almost every $\A$, the sequence of random variables
\[\max_{v\in V_N}\varphi^{\A,N}_v-\sqrt{\frac{2}{\pi\oA}}\left(\log N-\frac{3}{8}\log\log N\right)\]
converges in law as $N\to\infty$ to a deterministic limit.
\end{theorem}
Theorem \ref{t:random_conductances} resolves a version of \cite[Problem 1.23]{B11}, where the conductances $\kappa$ in the notation there are i.i.d.

Our second application is concerned with the Gaussian free field on the infinite cluster of supercritical Bernoulli percolation with parameter sufficiently close to 1. Recall, see \cite{K80},
that for every 
$p> 1/2$, under the the law $\PP_p$ of Bernoulli bond percolation on $\Z^2$,
there is almost surely
a unique infinite cluster $\CC_\infty$ of open edges.
\begin{theorem}\label{t:perc_unit_cond}
There is $1/2\le p_*<1$ with the following property. For $p> p_*$,
construct a family of Gaussian fields $\varphi^{\CC_\infty,N}$ on $V_N\cap \CC_\infty$ with law 
\begin{align*}
	&\pp^{\CC_\infty,N}(\mathrm{d}\varphi^{\CC_\infty,N})\\
	&=\frac{1}{Z_{\CC_\infty}^{V_N}}\exp\left(-\frac12\sum_{\{u,v\}\in E(\CC_\infty)}(\varphi^{\CC_\infty,N}_u-\varphi^{\CC_\infty,N}_v)^2\right)\prod_{v\in V_N(w)\cap\CC_\infty}\ud\varphi^{\CC_\infty,N}_v\prod_{v\notin V_N(w)\cap\CC_\infty}\ud\delta_0(\mathrm{d}\varphi^{\CC_\infty,N}_v).
\end{align*}
Then there is $ a_p $ depending on $p$ only such that for $\PP_p$-almost every realization of $\CC_\infty$, the sequence of random variables
\[\max_{v\in V_N}\varphi^{\CC_\infty,N}_v-\sqrt{\frac{2}{\pi a_p }}\left(\log N-\frac{3}{8}\log\log N\right)\]
converges in law as $N\to\infty$ to a deterministic limit.
\end{theorem}

As in the case of LCGF's, we also have a description of the limits in Theorem \ref{t:perc_unit_cond} and \ref{t:perc_unit_cond} as a randomly shifted Gumbel distribution.

%Our general result can also be applied to a variety of other examples, but before discussing these, let us state the result itself precisely.

Before proceeding, we explain the emergence of the constants $\oA$ and $ a_p $ in Theorems \ref{t:random_conductances} and \ref{t:perc_unit_cond}, as
well as the intuition why should one expect these results. 
%We focus first on Theorem \ref{t:random_conductances}. 
Recall that the covariance of the Gaussian fields $\varphi^{\A,N}$ in the theorems
is given by the Green function of a random walk on the random-conductance 
graph given by $\A$, killed when hitting $\partial^+ V_N$. It is natural then to
ask for the scaling limit (on the infinite lattice) of the random
walk, and homogenization theory gives that scaling limit. More 
precisely, building on the Kipnis-Varadhan theory, it is 
proved in \cite{DMFGW88} that the random walk converges (under the
so called \textit{averaged} law, that is under the law which is the
semidirect product of $\PP$ (or $\PP_p$) 
with the law of the random walk), to a Brownian
motion of diffusivity $\oA$ (or $ a_p $). 
Convergence to Brownian motion in the random conductance model of Theorem
\ref{t:random_conductances}
under the
\textit{quenched} law (that is, for $\PP$ almost every $\A$) appears 
(with some restrictions) in \cite{Ko85}, and in full generality 
in \cite{SS04}, where the percolation case is also handled, but only
in dimension $d\geq 4$; the analogous result for percolation clusters in all
dimensions $d\geq 2$
was settled in \cite{BB07, MP07}, building on earlier heat kernel estimates
in \cite{B04} and \cite{MR04}. See \cite{B11} for a review. 

With the invariance principle (and related heat kernel estimates) in place, 
it is then natural to expect the scaled Green function for the random walk to
be related to that for Brownian motion. Indeed, 
the (quenched) invariance principle is enough for proving that macroscopic averages of the Green function behave like macroscopic
averages of the covariance of a LCGF (cf. \cite{ADS20}); in particular, this makes the statements
in Theorems 
\ref{t:random_conductances} and \ref{t:perc_unit_cond} plausible.
However, for our needs much more \textit{quantitative} estimates are needed. 
For these, we rely on recently developed versions of homogenization theory, such as developed in 
\cite{GNO15, AKM17, AKM19}.
 We note that convergence at the level
of the associated Gaussian fields is proved (in the continuous setup) in
\cite{CR22}. Further details are provided below in Section \ref{sec-4}.

Theorems 
\ref{t:random_conductances} and \ref{t:perc_unit_cond} both are consequence of
Theorem
\ref{t:percolation_cluster} below, which utilizes quantitative homogenization
in checking the general hypotheses of our main result,
Theorem \ref{t:mainthm}. We note that the latter is expected to be useful also
in other contexts, and we list some examples in Section \ref{sec-applications}.

\subsection{Maxima of Gaussian fields under non-uniform log-correlation assumptions}
We next turn to the general result alluded to above, and begin by introducing notation and assumptions. For $L\le N$ set
%\iffalse
%For $\delta>0$ we define 
%\[V^\delta_N(w):=\left\{v\in  V_N(w)\colon\dist(v,\partial V_N(w)\ge\delta N\right\}\]
%Let also $\log_+ x:=\max(\log x,0)$, and for $L\le N$ let
%\fi
\begin{equation}
  \label{eq:calW}
  \W_{N,L}(w)=\left\{w'\in\Z^d: w'-w\in L\Z^d,V_L(w')\subset V_N(w)\right\}.
\end{equation}

%Recall also the covariance notation $R^{w,N}(\cdot,\cdot)$, see \eqref{eq-cov}. 
For any Polish space $\mathcal{X}$, we let $\mathcal{M}_1(\mathcal{X})$ denote the space of probability measures on $\mathcal{X}$, equipped with the topology of weak convergence. For $w\in \Z^d$, we let $\tau_w$ denote the shift by $w$ on $\Z^d$ (i.e. $\tau_w(v)=v+w)$, and also the shift on $\R^{\Z^d}$ (i.e. $\tau_w F(\cdot)=F(\tau_w(\cdot))$).

For $N\in\N$ and $w\in\Z^d$, we consider a centered Gaussian field $\varphi^{N,w}=(\varphi^{N,w}_v)_{v\in V_N(w)}$ on $V_N$ with %covariance 
%\begin{equation}
%\label{eq-cov}
%R^{N,w}(v,u)=\E\varphi^{N,w}_v\varphi^{N,w}_u, \quad v,u\in V_N,%\end{equation}
%and
law $\pp^{N,w}$.

\begin{assumption}
\label{as:main}
There exist functions $\T,\Rone,\Rtwo\colon\Z^d\to[0,\infty)$ so that the variances and covariances $\Var^{N,w}(\cdot)$, $\Cov^{N,w}(\cdot,\cdot)$ of $\varphi^{N,w}$ satisfy the following.
\begin{enumerate}[label=(A.\arabic*)]
	\item\label{a:logupp} \textbf{(Logarithmic upper bounds on the covariances)} For all $N$, all $w\in\Z^d$ and for all $u,v\in V_N(w)$,
	%\[\Var\varphi^{N,w}_v\le\log N+\T_v\]
	\[\Var^{N,w}(v)\le\log N+\T_v\]
	and
	%\[\Var\varphi^{N,w}_v-\E\varphi^{N,w}_v\varphi^{N,w}_u\le \log_+|u-v|+\T_u+\T_v.\]
	\[\Var^{N,w}(v)-\Cov^{N,w}(u,v)\le \log_+|u-v|+\T_u+\T_v.\]
	\item\label{a:logbd} \textbf{(Logarithmic bounds on the covariances away from the boundary)} For every $\delta>0$ there is an increasing function $\alpha_{\delta}\colon[0,\infty)\to[0,\infty)$ such that for all $N$, all $w\in\Z^d$ and for all $u,v\in V^\delta_N(w)$,
	  \[\left|\Cov^{N,w}(u,v)-\log N+\log_+|u-v|\right|\le \alpha_{\delta}(\Rone_u)+\alpha_{\delta}(\Rone_v).\]
	%\[\left|\E\varphi^{N,w}_v\varphi^{N,w}_u-\log N+\log_+|u-v|\right|\le \alpha_{\delta}(\Rone_u)+\alpha_{\delta}(\Rone_v).\]
	\item\label{a:micro} \textbf{(Approximation on a microscopic scale near the diagonal)} There are a continuous function $f\colon(0,1)^d\to\R$ which is bounded from above and a function $g\colon\Z^d\times\Z^d\to\R$ such that the following holds. For all $L,\delta,\ep>0$ there is $N_{A.3}$ such that for $N\ge N_{A.3}$, $w\in\Z^d$ and for all $u,v\in V^\delta_N(w)$ with $|u-v|_\infty\le L$ and $\Rone_u\vee\Rone_v\le \delta N$, $\Rtwo_w\le N$ we have
	  \[\left|\Cov^{N,w}(u,v)-\log N-f\left(\frac{v-w}{N}\right)-g(u,v)\right|\le\ep.\]
	%\[\left|\E\varphi^{N,w}_v\varphi^{N,w}_u-\log N-f\left(\frac{v-w}{N}\right)-g(u,v)\right|\le\ep.\]
	\item\label{a:macro} \textbf{(Approximation on a macroscopic scale off the diagonal)} There is a continuous function $h\colon(0,1)^d\times(0,1)^d\setminus\{(x,x)\colon x\in(0,1)^d\}\to\R$ such that the following holds. For all $L,\delta,\ep>0$ there is $N_{A.4}$ such that for $N\ge N_{A.4}$, $w\in\Z^d$ and for all $u,v\in V^\delta_N(w)$ with $|u-v|_\infty\ge {N}/{L}$ and $\Rone_u\vee\Rone_v\le (\delta N)\wedge({N}/{L})$, $\Rtwo_w\le N$ we have
	  \[\left|\Cov^{N,w}(u,v)-h\left(\frac{u-w}{N},\frac{v-w}{N}\right)\right|\le\ep.\]
	%\[\left|\E\varphi^{N,w}_v\varphi^{N,w}_u-h\left(\frac{u-w}{N},\frac{v-w}{N}\right)\right|\le\ep.\]
\end{enumerate}
Additionally, the following holds for $w=0$.

\begin{enumerate}[label=(B.\arabic*)]%\setcounter{enumi}{4}
	\item\label{a:sparseT} \textbf{(Quantitative sparsity of large values of $\T$)} There are $\ep>0$, $C>0$ such that the following holds. For every $L,N\in\N$, for any $T\ge 1$ and for any $w'\in\W_{N,\left\lfloor\frac{N}{L}\right\rfloor}$ we have that
	\[\left|\left\{v\in V_{\left\lfloor\frac{N}{L}\right\rfloor}(w')\colon \T_v\ge T\right\}\right|\le C\left(\frac{N}{L}\right)^d\e^{-(d+\ep)T}.\]
	\item\label{a:sparseR} \textbf{(Qualitative sparsity of large values of $\Rone$ and $\Rtwo$)}  We have that
\begin{align*}
\limsup_{R\to\infty}\sup_{L\in\N}\limsup_{N\to\infty}\left(\frac{L}{N}\right)^d\max_{w'\in \W_{N,\left\lfloor \frac{N}{L}\right\rfloor}(w_N)}\left|\left\{v\in V_{\left\lfloor \frac{N}{L}\right\rfloor}(w'): \Rone_{v}> R\right\}\right|&=0,\\
\sup_{L\in\N}\limsup_{L'\to\infty}\limsup_{N\to\infty}\left(\frac{LL'}{N}\right)^d\max_{w'\in \W_{N,\left\lfloor \frac{N}{L}\right\rfloor}(w_N)}\left|\left\{w''\in \W_{\left\lfloor \frac{N}{L}\right\rfloor,L'}(w'): \Rtwo_{w''}> L'\right\}\right|&=0.
\end{align*}
	\item\label{a:lln} \textbf{(Law of large numbers for the local behaviour)} For every $L'\in\N$ there is a probability measure $\nu_{L'}\in \mathcal{M}_1(\mathcal{M}_1(\R^{V_{L'}(0)}))$ such that for any $L\in\N$, for every $a\in V_L$, setting $w_{a,N}'=a\left\lfloor \frac{N}{L}\right\rfloor$ and
	\[\mu_{N,a}=\frac{1}{\left|\W_{\left\lfloor \frac{N}{L}\right\rfloor,L'}\right|}\sum_{w''\in\W_{\left\lfloor \frac{N}{L}\right\rfloor,L'}(w_{N,a}')}\delta_{\pp^{L',w''}(\tau_{w''}(\cdot))}\]
	then the sequence $\mu_{N,a}$ converges weakly to $\nu_{L'}$.
\end{enumerate}
\end{assumption}
In Assumption \ref{as:main}, one thinks of the variables $\Rone_v$ as
the scale at which covariance estimates for $\Cov^{N,w}(v,\cdot)$ 
hold, $\Rtwo_w$ as the scale at which homogenization applies for random walk
started at $w$, and $\T_v$ as the local maximal 
error for estimates on the covariance $\Cov^{N,w}(v,\cdot)$.

\begin{theorem}\label{t:mainthm}
Let Assumption \ref{as:main} hold. Set 
%$M_{N}:=\max_{v\in V_N(w_N)}\varphi^{N,w_N}_v$, 
$M_{N}:=\max_{v\in V_N}\varphi^{N}_v$, 
and define the sequence \[m_N:=\sqrt{2d}\log N-\frac{3}{2\sqrt{2d}}\log\log N.\]
Then 
$M_{N}-m_N$ 
%$M_{N,w_N}-m_N$ 
converges in distribution to a limit that is given as a randomly shifted Gumbel distribution. That is, there exists a 
positive number $\beta^*$ and a random variable $Z$ so that,
for any $t\in \R$, 
%\[ \lim_{N\to\infty}\pp^{N}( M_{N,w_N}-m_N\leq t)= E( e^{- e^{-t+Z}}).\]
\[ \lim_{N\to\infty}\pp^{N}(M_{N}-m_N\leq t)= \E\left( \e^{- \beta^*Z \e^{-\sqrt{2d}t}}\right).\]
Moreover, $Z$ is the weak limit of the sequence $Z_N$ defined by
\[Z_N=\sum_{v\in V_N}(\sqrt{2d}\log N-\varphi^N_v)\e^{-2d\log N+\sqrt{2d}\varphi^N_v}.\]
\end{theorem}
As discussed in \cite{DRZ17}, $Z_N$ resembles the derivative martingale which occurs in the study of various branching processes. In our setting, $Z_N$ is not necessarily a martingale. Nonetheless we expect that many of the properties of $Z_N$ and its limit
%the derivative martingale 
in the case of the two-dimensional Gaussian free field carry over to our setting. In particular, in analogy to \cite{BL16}, there should actually exist a random measure $\ZZ$ on $[0,1]^d$ with $Z=\ZZ([0,1]^d)$ that encodes the locations of near-maximizers of the fields and has a number of other interesting properties. However, we do not pursue this description in the current paper.

We next explain in more detail the role of the different parts of
Assumption \ref{as:main}. First, one should compare our assumptions to those in \cite{DRZ17}, which are similar to those in part A of 
Assumption \ref{as:main} (we kept similar notation to allow for easy comparison). In fact, as we explain in Remark \ref{r:gen_of_DRZ} below, one can check that Theorem \ref{t:mainthm} is a strict generalization of the main results of \cite{DRZ17}.

The assumptions in part A are on $\varphi^{N,w}$ for any $w\in\Z^d$, i.e. on boxes of any size and position. One can think of them as providing an implicit definition of $\T_\cdot$, $\Rone_\cdot$ and $\Rtwo_\cdot$. The assumptions in part B are only on $\varphi^N=\varphi^{N,0}$, i.e. on our domain of interest $V_N=V_N(0)$. They encode the fact that on $V_N$ the random scales $\Rone_\cdot$ and $\Rtwo_\cdot$ and the error term $T_\cdot$ are well-behaved.

In Assumption \ref{a:logupp} we state upper bounds on the variances in terms of the random field $\T_\cdot$. It is clear that if there were many points where the variances are atypically large, then these could
influence the maximum of the field in a non-negligible way. So we accompany Assumption \ref{a:logupp} with Assumption \ref{a:sparseT}, which provides a quantitative tail bound on the number of points in $V_N$ where $T_\cdot$ is large. This assumption is close to optimal, as the following heuristic 
back-of-the-envelope calculation shows.
Suppose there are $\kappa_T N^d$ many points in $V_N$ with variance $\ge \log N+T$, where $1\ll T\ll\log N$, $1\ll |\log\kappa_T|\ll\log N$, and such that
the field at these points is still logarithmically correlated. 
Then the maximum of the field over those points should be at least of order
\begin{align*}
\sqrt{\frac{\log N+T}{\log(\kappa_T^{1/d}N)}}m_{\kappa_T^{1/d}N}&=\sqrt{2d}\sqrt{(\log N+T)\log(\kappa_T^{1/d}N)}-\frac{3}{2\sqrt{2d}}\sqrt{\frac{\log N+T}{\log(\kappa_T^{1/d}N)}}\log\log(\kappa_T^{1/d}N)\\
%&=\sqrt{2d}\left(\log N+\frac{T+\log\kappa_T^{1/d}}{2}+o(1)\right)-\frac{3}{2\sqrt{2d}}\left(\log\log N+o(1)\right)\\
&=m_N+\sqrt{\frac{d}{2}}(T+\log\kappa_T^{1/d})+o(1).
\end{align*}
Therefore, if the maximum of the entire field is supposed to be of order $m_N$, we need that $T+\log\kappa_T^{1/d}$ is bounded above for $T$ large enough. That is, we need $\kappa_T\le C\e^{-dT}$ for $T$ large enough. Our assumption \ref{a:sparseT} essentially implies that $\kappa_T\le C\e^{-(d+\ep)T}$ for some $\ep>0$, and so we see that this is close to optimal.

Assumption \ref{a:logbd} provides error bounds with error $O(1)$ away from the boundary, and Assumptions \ref{a:micro} and \ref{a:macro} provide even sharper estimates, with error $o(1)$, on microscopic and macroscopic scales. These three assumptions are all in terms of the random scales $\Rone_\cdot$ and $\Rtwo_\cdot$. For the conclusion of Theorem \ref{t:mainthm} we need to have these bounds on most points of $V_N$, and here the qualitative assumption \ref{a:sparseR} is sufficient.

While assumption \ref{a:micro} ensures that the covariances of the field are locally described by a function $g\colon\Z^d\times\Z^d$, it requires nothing of $g$. We still need to make sure that $g$ behaves roughly the same on the points of $V_N$. Our rather weak, qualitative way to do so is assumption \ref{a:lln}. It makes sure the law of the field on small boxes satisfies a law of large numbers.

\begin{remark}\label{r:variable_wN}
As mentioned above,
the assumptions in part A are on $\varphi^{N,w}$ for any $w\in\Z^d$, while the assumptions in part $B$ are only on $\varphi^N=\varphi^{N,0}$, i.e. on our domain of interest $V_N=V_N(0)$.
It is also possible to prove a variant of Theorem \ref{t:mainthm}, where instead of considering $V_N(0)$ one considers $V_N(w_N)$ for some sequence $(w_N)_{N\in\N}$ of points in $\Z^d$. In this case one needs to modify the assumptions in part B so that they hold for $\W_{N,\left\lfloor{N}/{L}\right\rfloor}(w_N)$ instead of $\W_{N,\left\lfloor{N}/{L}\right\rfloor}$, and one also needs to make the additional assumption that for $N$ large enough $\Rtwo_{w_N}\le N$ (if $w_N=0$ for all $N$, that is trivially true, so we did not list this in Assumption \ref{as:main}).
The proof of this slightly more general statement is exactly the same as the one we give below.
\end{remark}

\begin{remark}\label{r:gen_of_DRZ} As mentioned above, Theorem \ref{t:mainthm} is a generalization of the main result of \cite{DRZ17}. To see this, it suffices to check that any Gaussian field satisfying the assumptions of \cite{DRZ17} also satisfies Assumption \ref{as:main}
% the assumptions of Theorem \ref{t:mainthm} 
for a suitable choice of $\T,\Rone,\Rtwo$.

Indeed we can choose $\T_v=\alpha_0$ with the $\alpha_0$ from \cite[Assumption (A.0)]{DRZ17}, $\Rone_v=\Rtwo_v=1$, and $\alpha_\delta(\cdot)={\alpha^{(\delta)}}/{2}$ with the $\alpha^{(\delta)}$ from \cite[Assumption (A.1)]{DRZ17}.
With these choices $\T,\Rone,\Rtwo$ are all bounded, so that \ref{a:sparseT} and \ref{a:sparseR} hold trivially. 
%\sout{Assumption \ref{a:lln} follows from the fact that the $\pp^{L',w''}(\tau_{w''}(\cdot))$ are independent and identically distributed.}
Assumption \ref{a:lln} is immediate since the Gaussian laws $\pp^{L',w''}(\tau_{w''}(\cdot))$ do not depend, in the setup of \cite{DRZ17}, on $w''$.
Next, assumptions \ref{a:logupp}, \ref{a:logbd} and \ref{a:macro} all follow from straightforward calculations which we omit. For \ref{a:micro} things are less obvious, so let us give some details:

To avoid clashes in notation, rename the functions $f$ and $g$ from \cite[Assumption (A.2)]{DRZ17} to $\tilde f$, $\tilde g$. The function $\tilde g$ must be translation-invariant (to see this, use \cite[Assumption (A.2)]{DRZ17} with $x=\bar x:=\left(\frac12,\ldots,\frac12\right)^d$ and $(u,v)=(\bar u,\bar v)$ arbitrary, and then also with $x=\bar x+\frac{\bar u}{N}$ and $(u,v)=(0,\bar v-\bar u)$). Thus $\tilde g(u,v)$ is actually a function of $u-v$. The function $\tilde f$ must be bounded above, because otherwise Assumption (A.2) with $(u,v)=(0,0)$ would contradict \cite[Assumption (A.0)]{DRZ17}.
These considerations show that if we choose $f=\tilde f$ and $g=\tilde g$ then \ref{a:micro} is satisfied as well.
\end{remark}

\begin{remark}
In Assumption \ref{a:micro} the condition that $f$ is bounded from above can be relaxed. In fact, if we set
\[\Gamma_\delta:=\sup_{x\in[\delta,1-\delta]^d}f(x)\]
then the condition 
\[\Gamma_\delta\le\frac{1-\ep}{d}|\log \delta|\]
for some $\ep>0$ is sufficient. To show this slightly stronger result, one can no longer take the limit $\delta\to0$ before $T\to\infty$ in the results below, but needs to choose $\delta$ as a suitable function of $T$.
However, in all applications we are aware of, $f$ will be bounded from above, and so we chose to give the detailed proof only in that case.
\end{remark}

 \begin{remark}
In all examples that we have in mind, Assumption \ref{a:lln}
   is checked using an appropriate ergodic theorem, and provides information concerning averages of local functions. This forces us, in various places in the proof and in particular in applications to the percolation problem, to 
   approximate general translation invariant functions by local functions.
   One could replace assumption \ref{a:lln} by a stronger ergodicity assumption involving the law of the field jointly with $\T_\cdot$, $\Rone_\cdot$ and $\Rtwo_\cdot$. This stronger assumption would still be satisfied in our examples of interest, and some proofs (in particular step 4 in the proof of Lemma \ref{l:right_tail}) could be shortened slightly. However, we found it worthwhile to showcase that we only require a very weak ergodicity assumption like \ref{a:lln}, and so we chose the current formulation.
 \end{remark}

\subsection{Applications to log-correlated fields in random environment}\label{sec-applications}
After our discussion of Theorem \ref{t:mainthm}, we 
now explain how it can be used. We already mentioned that our main application is to establish Theorems \ref{t:random_conductances} and \ref{t:perc_unit_cond}. These theorems follow from Theorem \ref{t:mainthm} once we show that for almost every realization of the environment the corresponding Gaussian measures satisfy Assumptions \ref{as:main}. In fact, we prove that this is the case for a common generalization of the setting of both Theorem \ref{t:random_conductances} and Theorem \ref{t:perc_unit_cond}
where we allow conductances in $(\{0\}\cup[\lambda,\Lambda])$ under the assumption that $p:=\PP(\A(e)>0)>p_c=1/2$. We also allow boxes with locations varying with $N$, i.e. instead of considering the fields on $V_N(0)$ we consider them on $V_N(w_N)$ for some deterministic sequence $(w_N)_{N\in\N}$. 
This is the content of the following theorem, which 
is the second main result of the present paper.

\begin{theorem}\label{t:percolation_cluster}
Let $1/2<p\le1$ and $0<\Lambda^-\le\Lambda^+$. 
%We let $\F$ be the Borel $\sigma$-algebra on $(\{0\}\cup[\Lambda^-,\Lambda^+])^{E(\Z^2)}$, and we let 
Let $\PP$ be an i.i.d. Borel probability measure on 
$(\{0\}\cup[\Lambda^-,\Lambda^+])^{E(\Z^2)}$,
%,\F$ 
with $\PP(\A(e)>0)=p$. For $\PP$-almost every sample $\A$,
let $\CC_\infty$ denote the
(unique) infinite cluster of edges $e$ where $\A(e)>0$, see \cite{K80}.
%For $\PP$-almost every sample $\A$ (which we identify with its indicator function) there is a unique infinite cluster $\CC_\infty$ of edges $e$ where $\A(e)>0$. We can 
With $\tilde V_N(w)=V_N(w)\cap \CC_\infty$,
construct a family of Gaussian fields $\varphi^{\A,N,w}$ on $V_N(w)\cap C_\infty$ for $N\in\N$ with law 
\begin{align*}
	&\pp^{\A,N}(\mathrm{d}\varphi^{\A,N,w})\\
	&=\frac{1}{Z_\A^{V_N(w)}}\exp\Big(-\frac12\sum_{\{u,v\}\in E(\Z^2)}\A(\{u,v\})(\varphi^{\A,N,w}_u-\varphi^{\A,N,w}_v)^2\Big)\prod_{v\in \tilde V_N(w)}\ud\varphi^{\A,N,w}_v\!\!\!
	\prod_{v\notin \tilde V_N(w)}\delta_0(\mathrm{d}\varphi^{\A,N,w}_v)
\end{align*}
and extend this to a family of Gaussian fields $\varphi'^{\A,N,w}$ on 
$V_N(w)$ by setting $\varphi'^{\A,N,w}_v=\varphi^{\A,N,w}_{v^*}$, 
where $v^*\in\Z^2$ is the point in $\CC_\infty$ closest to $v$ 
(with ties broken by taking the lexicographically first point).% This ensures in particular
%\[\max_{v\in V_N(w_N)\cap\CC_\infty}\varphi^{\A,N,w}_v=\max_{v\in V_N(w_N)}\varphi'^{\A,N,w}_v\]

Then there is a deterministic constant $\oA$ with the following property. 
For $\PP$-almost every $\A$ the collection of Gaussian fields 
$\sqrt{2\pi\oA}\varphi'^{\A,N,w}$ on $V_N(w)$ satisfy 
assumptions \ref{a:logbd}, \ref{a:micro}, \ref{a:macro}, \ref{a:sparseR} and \ref{a:lln} of Theorem \ref{t:mainthm}. Furthermore, 
there is a constant $\frac12\le p_{\Lambda^+/\Lambda^-}<1$ 
which depends on ${\Lambda^+}/{\Lambda^-}$ only, 
such that if $p>p_{\Lambda^+/\Lambda^-}$ then the collection of Gaussian fields also satisfies, for $\PP$-almost every $\A$,
assumptions \ref{a:logupp} and \ref{a:sparseT}.
In particular, if $p>p_{\Lambda^+/\Lambda^-}$, then for $\PP$-almost every $\A$ the fields $\sqrt{2\pi\oA}\varphi'^{\A,N,w}$ satisfy all assumptions of Theorem \ref{t:mainthm}, and consequently 
\[\max_{v\in V_N}\varphi'^{\A,N}-\sqrt{\frac{1}{2\pi\oA}}m_N=\max_{v\in V_N\cap\CC_\infty}\varphi^{\A,N}_v-\sqrt{\frac{1}{2\pi\oA}}m_N\]
converges in distribution to a limit 
that is a randomly shifted Gumbel distribution.
\end{theorem}
In view of Remark \ref{r:variable_wN}, Theorem \ref{t:percolation_cluster} also holds with $V_N=V_N(0)$ replaced by $V_N(w_N)$ for a deterministic sequence $(w_N)_{N\in\N}$ (but note that then the $\PP$-null set of environments $\A$ for which the theorem fails will in general depend on $(w_N)_{N\in\N}$). In particular, our proof of Theorem \ref{t:percolation_cluster} does not use in any way that the boxes $V_N(0)$ are nested, but rather uses Borel-Cantelli-type arguments to show that almost surely the assumptions in part B of Assumption \ref{as:main} are satisfied.

Our results for percolation clusters are restricted to the highly supercritical regime where $p$ is close to 1. A natural question is what happens for other $p>1/2$. Of course, Theorem \ref{t:mainthm} implies that we can take the constant $p_*$ in Theorem \ref{t:perc_unit_cond} to be equal to the constant $p_1$ in Theorem \ref{t:percolation_cluster}. This raises the following question.
\begin{question}\label{q:nearcritical}
How does $p_{\Lambda^+/\Lambda^-}$ depend on ${\Lambda^+}/{\Lambda^-}$? 
In particular, can one take $p_{\Lambda^+/\Lambda^-}=1/2$, at least if 
${\Lambda^+}={\Lambda^-}$?
If that is not possible, can one nonetheless take $p_*=1/2$ in Theorem \ref{t:perc_unit_cond}?
\end{question}
We think that this question is very interesting, but possibly challenging. 
Our proof of Theorem \ref{t:percolation_cluster}
relies on the fact that $p_{\Lambda^+/\Lambda^-}$ is close to 1
(as explained in Remark \ref{r:nearcritical} below). 
So a positive answer to question \ref{q:nearcritical} would require some new ideas.

While in this work we focus on the application in Theorem \ref{t:percolation_cluster}, the general result in Theorem \ref{t:mainthm} should also be applicable in various other settings.
\begin{enumerate}
%  \item Our proof of Theorem \ref{t:percolation_cluster} heavily relies
%    on $p>p_0$ for some $p_0$ close enough to $1$. Indeed, there is a
%    competition between the behavior of the exponent $s$ 
%    appearing in the statements of 
%    the theorems and lemmas of Section \ref{sec-4}
%    and critical exponents for chemical distance as $p\searrow 1/2$,
%    which are not known but for which numerical-based conjectures exist.
%    \todo{Add ref}. We do not know whether the theorem remains true in
%    the whole super-critical range.
\item In Theorem \ref{t:percolation_cluster} we assumed that the conductances $\A$ are independent. It should be possible to relax this assumption and to replace it, for example, by a sufficiently strong mixing assumption. However, we certainly need a quantitative assumption, and mere ergodicity is not sufficient. In fact, in \cite{ADS20} under the assumption of ergodicity error bounds on the covariances with error $o(\log N)$ were shown, and this is best possible.

One example where such extension is
relevant might be when one replaces the percolation model by the two-dimensional
interlacement model of \cite{CPV16}.

Another interesting example with non-i.i.d. environment
are the gradient interface models introduced in \cite{BK07,BS11}. These are non-Gaussian, but can be represented as a mixture of Gaussian fields with ergodic random conductances. Our framework might be useful to study the maximum of the latter fields, see for example \cite[Problem 1.23]{B11}.

  \item As mentioned above, another important example of a LCGF is the membrane 
    model in the critical dimension $d=4$, see \cite{S20} for the convergence
    in law of its centered maximum. 
    One can define disordered versions of this model in various ways. 
    A natural one is to take i.i.d. positive random variables 
    $\mathbf{b}$ for each vertex of $\Z^4$, and then 
    consider the probability measure
    \begin{align*}
	&\pp^{\mathbf{b},N}(\mathrm{d}\varphi^{\mathbf{b},N})\\
	&=\frac{1}{Z_\mathbf{b}^{V_N}}\exp\left(-\frac12\sum_{v\in \Z^4}\mathbf{b}(v)(\Delta\varphi^{\mathbf{b},N}_v)^2\right)\prod_{v\in V_N}\ud\varphi^{\mathbf{b},N}_v\prod_{v\notin V_N}\ud\delta_0(\mathrm{d}\varphi^{\mathbf{b},N}_v).
\end{align*}
    Checking that Theorem \ref{t:mainthm} applies in that setting
    would require the development of a suitable quantitative homogenization
    theory for fourth-order elliptic equations, 
    and bypassing various comparison theorems we used where the
    maximum principle was invoked. Nonetheless, we do expect the analogue of Theorem
    \ref{t:percolation_cluster} to hold in this context.
    
    Alternatively, one could also take i.i.d conductances $\A$ on the edges of $\Z^4$, and define a Gaussian field
    on $V_N$ as the preimage of discrete white noise (i.e. i.i.d standard Gaussians) under the operator $-\nabla\cdot\A\nabla$. This is the approach used in \cite{CR22}. Again, we would expect that the analogue of Theorem
    \ref{t:percolation_cluster} holds in this context.
   \item The disordered GFF in Theorem \ref{t:percolation_cluster} and the disordered membrane model all only involve finite range
     %nearest-neighbor 
     interactions in the Hamiltonian, and so have some domain Markov property. However, this is not used in Theorem \ref{t:mainthm}. In particular, the theorem should also be applicable when considering independent random conductances $\A(u,v)$ for non-neighboring $u,v$, provided that these conductances decay sufficiently fast as $|u-v|\to\infty$.  
   
  \item As in \cite{DRZ17}, we focused in this paper on lattice models.
    One expects similar results in the continuum, where one uses 
    mollifications of the continuous GFF defined using a disordered reversible second order
    operator. The quantitative homogenization results for such operators
    are available in \cite{AKM19}, while the convergence in law of the maximum
    in the non-disordered case was derived in \cite{A16}. 
\end{enumerate}

\subsection{Outline of the proofs}
The proof of Theorem \ref{t:mainthm} builds on the argument of \cite{DRZ17}. We recall the latter. Starting with a LCGF $\varphi_v$ satisfying a uniform version of Assumptions \ref{as:main}, one first proves the tightness of the maximum (shifted by $m_N$) and the following crucial property: local maxima of the field that are within
a bounded distance from $m_N$ must be macroscopically separated. Equipped with this fact, one  
constructs approximating fields $\xi_v$ such so that the variance  of $\xi_v$ matches that of $\phi_v$, and the covariances match at microscopic and macroscopic scales (while remaining a bounded distance away in mesoscopic scales). This is achieved by dividing $V_N$ into macroscopic boxes (of side length $N/J$) and microscopic boxes (of side length $J'$). One then constructs 
three Gaussian fields, denoted $\xi_\cdot^{ \rm mac}$, $\xi_\cdot^{\rm meso}$, and $\xi_\cdot^{ \rm mic}$, called the macroscopic, mesoscopic and microscopic fields,
which are independent of each other and have the following properties:
\begin{itemize} 
\item The field $\xi_\cdot^{ \rm mac}$ is piecewise constant on macroscopic boxes, and its covariance matches the function $h$ in assumption \ref{a:macro}.
\item The field $\xi_\cdot^{ \rm meso}$ is a modified branching random walk (see \eqref{eq-MBRW} below for a definition) with covariance adjusted 
to match the mesoscopic scale of the LCGF, up to additive error.
\item The field $\xi_\cdot^{ \rm mic}$ consists of independent copies (between microscopic boxes) of a field whose covariance matches the local covariance of the field determined by the functions $f,g$ in assumption \ref{a:micro}. 
\item The field $\xi_\cdot= \xi_\cdot^{ \rm mac}+\xi_\cdot^{\rm meso}+\xi_\cdot^{ \rm mic}$ matches the covariance of $\phi_\cdot$ at microscopic and macroscopic 
scales.
\end{itemize}
(The actual construction decomposes $J=KL$ and $J'=K'L'$, in order to obtain good matching
of covariances, but we overlook this detail in this high level description.)
We now describe how the proof is adapted to the non-homogeneous setup of the field $\varphi^N_\cdot$.
First, one controls the maximum of the field over  the ``bad points'', i.e. those vertices $v$ with large values of $\T_v$, $\Rone_v$ or which lie in boxes with large $\Rtwo_w$, and shows that those do not contribute, see Lemma
\ref{l:upper_right_tail_badset}, whose proof (together with some auxiliary results) takes up Section \ref{subsec-badset}; the statement and its proof slightly sharpen the corresponding ones in \cite{DRZ17}. Having controlled the bad points, one controls the tails of the distribution of the maximum following up the recipe of \cite{DRZ17}, culminating with the proof of tightness of the centered maximum (Theorem \ref{t:tightness}) and the macroscopic separation between local maxima (Theorem \ref{t:geom_nearmax}). The fact that the good points now form a cube with small ``holes'' and thus have a rather complicated geometry causes some additional difficulty, but generally the arguments are very similar to \cite{DRZ17}.

The main point of departure from the analysis of \cite{DRZ17} comes in constructing our approximating fields, see \eqref{eq-approxfield}. The local inhomogeneity of $\varphi^N_\cdot$
and the need to match variances forces one to introduce local correction terms in the form of additive independent Gaussians; using the construction of \cite{DRZ17} would have led us to
a lack of uniform control on these corrections. Instead, we use corrections at  both the microscopic and macroscopic levels, which possess a uniform control. The main technical step is then to prove matching asymptotics for the right tail of the maximum over the good points (Lemma \ref{l:right_tail}). Here one needs extra care to compare the maximum over the good points with the maximum over all points (to which one can apply Assumption \ref{a:lln}). Having established these asymptotics, Theorem \ref{t:mainthm} follows by a coupling argument similar to the one in \cite{DRZ17}.
The details of the construction
are provided in Section \ref{sec-3}. 

\medskip

The proof of Theorem \ref{t:percolation_cluster} (which implies 
Theorems
\ref{t:random_conductances} and \ref{t:perc_unit_cond})
consists of checking the hypotheses of Theorem \ref{t:mainthm}. A major tool in
doing that is quantitative homogenization theory on the percolation cluster, 
which we use in the version developed in \cite{AD18,DG21} 
and review in Section \ref{sec-4.1}. This theory provides us with asymptotics for the (full-space) Green's function (Theorem \ref{t:green_asympt}) as well as estimates for the difference between the solution of a discrete Dirichlet problem and its continuous counterpart (Theorem \ref{t:homog_dirich}). Both these results hold on lengthscales beyond some random lengthscale (and these two random lengthscales will reoccur in $\Rone$ and $\Rtwo$, respectively). While in Theorem \ref{t:homog_dirich} the error term is only estimated in an $L^2$-sense, we can upgrade this to pointwise control by also using the large-scale $C^{0,1}$-regularity of $\A$-harmonic functions of \cite{AD18}. 
Assumptions \ref{a:logbd}, \ref{a:macro}, \ref{a:micro} then follow from these results with some work, relying mainly on the maximum principle with well-chosen comparison functions.

Regarding Assumption \ref{a:sparseR}, the results from \cite{AD18,DG21} come with a stretched-exponential tail bound on the random scales. This tail-bound allows us to control the first moment of the number of points with large $\Rone$ or $\Rtwo$. To establish \ref{a:sparseR}, however, such an annealed estimate is not good enough. Instead we also need an estimate on the second moment of the number of points with large $\Rone$ or $\Rtwo$. That is, we need to quantify that the random scales arising from \cite{AD18,DG21} are decorrelated on large scales. This issue has not been adressed in the previous literature. One way to establish such decorrelation would be to go through \cite{AD18,DG21}, keeping track of the dependence of the various scales there as they arise. Fortunately for us, there is an easier way, and we can explicitly write down local events that approximate the events that $\Rone$ or $\Rtwo$ are large. For Theorem \ref{t:homog_dirich} defining such an event is easy; for Theorem \ref{t:green_asympt} this is more tricky (see Lemma \ref{l:green_asympt_improved}). Finally Assumption \ref{a:lln} follows directly from the translation-invariance of the law of $\PP$. All this holds for any $p>p_c=1/2$.

Extra work is needed to verify Assumptions \ref{a:logupp} 
and \ref{a:sparseT}.
One part requires us to study boundary behavior of the Green
function; this is handled by reflections and an a-priori 
estimate in half-spaces, see Lemma \ref{l:green_asympt_halfspace}. More 
importantly, we need to check that points with exceptionally high variance (such as points in the bulk and end of long ``pipes'' connected to the percolation cluster)
are sparse enough  so that \ref{a:logupp} and \ref{a:sparseT} still hold. This we can only do if $p$ is sufficiently close to 1. The argument builds on a combination of isoperimetric control on
the percolation cluster, large deviation estimates for the chemical distance on the cluster (going back to \cite{AP96}), a-priori Peierls-like 
large deviations estimates for the size of the cluster (all recalled
in Section \ref{sec-4.3}) and a multiscale argument due to \cite{BK05},
see Lemma \ref{l:tail_bound_green} and its proof. Once these are established,
it is routine to verify Assumptions \ref{a:logupp} 
and \ref{a:sparseT}, and complete the proof of Theorem \ref{t:percolation_cluster}.

\subsection{Additional notation and conventions}
In all the following, $c$ and $C$ will denote generic constants whose precise value may change from line to line. If we want to emphasize that the value of $C$ depends on other quantities, we add them as index to $C$. 

We use $|\cdot|_p$ to denote the $\ell^p$ norm ($p\in [1,\infty]$)
on $\R^d$ or $\Z^d$, and use $|\cdot|=|\cdot|_2$. For a set $A$ we denote by $|A|$ its cardinality.

\section{Properties of the maximum}

We want to deduce estimates on the maximum of $\varphi$ by Gaussian comparison inequalities. To do so, let us introduce the objects that we will use for comparison, following \cite{DRZ17}.
Recall \eqref{eq:calW} and define
\begin{align}
  \label{eq:calV}
\V_{N,L}(w)&=\left\{V_L(w')\colon w'\in\W_{N,L}(w)\right\}, \quad 
\tilde\V_{N,L}(w)=\left\{V_L(w')\colon w'\in V_N(w)\right\}.
\end{align}

Assume that $N=2^n$ is a power of two, and that $w\in\Z^d$. Take a collection of i.i.d. Gaussians $X_{j,Q_j}$ of variance $\log2$, where $j\in\{0,\ldots,n\}$ and $Q_j\in\V_{N,2^j}$. Then the branching random walk $\theta^{N,w}=(\theta^{N,w}_v)_{v\in V_N(w})$ is defined as
\begin{equation}\label{eq-BRW}\theta^{N,w}_v=\sum_{j=0}^n\sum_{\substack{Q_j\in \V_{N,2^j}(w)\\v\in Q_j}}X_{j,Q_j}.\end{equation}
Similarly, take a collection of i.i.d. Gaussians 
$Y_{j,Q_j}$ of variance $2^{-jd} \log2$, where $j\in\{0,\ldots,n\}$ and $Q_j\in\tilde\V_{N,2^j}(w)$, and define the modified branching random walk $\tilde\theta^{N,w}=(\tilde\theta^{N,w}_v)_{v\in V_N(w})$ as
\begin{equation}
\label{eq-MBRW}\tilde\theta^{N,w}_v=\sum_{j=0}^n\sum_{\substack{Q_j\in \tilde\V_{N,2^j}(w)\\v\in Q_j+N\Z^d}}Y_{j,Q_j}.\end{equation}
It is well known (see e.g. \cite[Lemma 2.3]{DRZ17}) that for the BRW we have the estimates
\begin{align*}
\left|\E(\theta^{N,w}_v)^2-\log N\right|&\le C,\\
\E\theta^{N,w}_v\theta^{N,w}_u-\log N+\log_+|u-v|&\le C,
\end{align*}
and that for the MBRW we have
\[\left|\E\tilde\theta^{N,w}_v\tilde\theta^{N,w}_u-\log N+\log_+|u-v|_{\sim,N}\right|\le C,\]
where $|x|_{\sim,N}=\inf_{z\in N\Z^d}|x+z|$ is the quotient norm. This implies in particular that for any $\delta<\frac12$ and any $u,v\in V_N^\delta(w)$ we have
\begin{equation}\label{e:log_corr_MBRW}
\left|\E\tilde\theta^{N,w}_v\tilde\theta^{N,w}_u-\log N+\log_+|u-v|_{\sim,N}\right|\le C_\delta.
\end{equation}

\subsection{Tailbounds on the "bad set"}
\label{subsec-badset}

For the study of the maximum we will need to argue that the maximum of $\varphi^{N}$ is unlikely to occur on points that are exceptional in the sense that e.g. $\T_\cdot$ is large, $\Rone_\cdot$ is large, or the points are too close to the boundary of $V_N$. Thus, in this subsection we give upper bounds on the maximum of $\varphi^{N}$ in such exceptional sets.
The main technical result is the following.
\begin{lemma}\label{l:upper_right_tail_badset}
Under Assumptions \ref{a:logupp} and \ref{a:sparseT} there are $C,c>0$ such that for $J,N\in\N$ with $J\le\frac{N}{2}$, $w'\in\W_{N,\left\lfloor\frac{N}{J}\right\rfloor}$, $z\in\R$, $T$ sufficiently large, and $A\subset V_{\left\lfloor\frac{N}{J}\right\rfloor}(w')$,
\begin{equation}\label{e:upper_right_tail_badset}
\pp^{N}\bigg(\max_{\substack{v\in V_{\left\lfloor\frac{N}{J}\right\rfloor}(w'):\\\T_v\ge T\textrm{, or }v\in A}}\varphi^{N}_v\ge m_N+z
\bigg)\le C\!\Big(\e^{dT}\frac{|A|}{N^d}\Big(1+T+\log\Big(\frac{N^d}{|A|}\Big)\Big)^{19/8}+\e^{-cT}\!\Big)(z\vee 1)\e^{-\sqrt{2d}z}.
\end{equation}
\end{lemma}
\noindent
In particular,  choosing $J=1$ and $w'=0$ yields
an estimate on all of $V_N$.

To see why this estimate is useful, note that, for example, under assumptions \ref{a:logbd} and \ref{a:sparseR}  the set $A_{N,R,\delta}$ of points in $V_N$ where $\dist(v,\partial^+ V_N)\le \delta N$ or $\Rone_v\ge R$ satisfies
\[\lim_{\substack{\delta\to0\\R\to\infty}}\frac{|A_{N,R,\delta}|}{N^d}=0\]
(where here and in the following we write $\limsup_{\substack{\delta\to0\\R\to\infty}}$ to denote that we take $R\to\infty$ and $\delta\to0$ while the order of the limits does not matter).
So Lemma \ref{l:upper_right_tail_badset} (with $J=1$) implies that under assumptions \ref{a:logupp}, \ref{a:logbd}, \ref{a:sparseT} and \ref{a:sparseR} we have
\[\limsup_{T\to\infty}\limsup_{\substack{\delta\to0\\R\to\infty}}\frac{\e^{\sqrt{2d}z}}{z\vee 1}\pp^{N}\bigg(\max_{\substack{v\in V_N\\\T_v\ge T\textrm{, or }\Rone_v\ge R\text{, or }\dist(v,\partial^+ V_N)\le \delta N}}\varphi^{N}_v\ge m_N+z\bigg)=0.\]
That is, the maximum of $\varphi^N_v$ over the exceptional points in $V_N$ is indeed unlikely to be large, with a quantitative estimate on the right tail. 
We will use this and similar estimates repeatedly in the following. Combined with a lower bound on the maximum (that we will show in the next section), this estimate allows us to conclude that the maximum is unlikely to occur at an exceptional point.

We will later also need a tail bound for the maximum of the fields on the microscopic cubes $V_{J'}(w'')$ contained in $V_{\left\lfloor{N}/{J}\right\rfloor}(w')$. Of course, such an estimate cannot hold for each microscopic cube, as we have no control over $\Rone$ and $T$ in all cubes. However, the estimate does hold when averaging over all cubes.

\begin{lemma}\label{l:upper_right_tail_badset_micro}
Under Assumptions \ref{a:logupp} and \ref{a:sparseT} there are $C,c>0$ such that for $J,J',N\in\N$ with $J\le{N}/{2}$, $J'\le{N}/({2J})$, $w'\in\W_{N,\left\lfloor{N}/{J}\right\rfloor}$, $z\in\R$, $T$ sufficiently large, and $A\subset V_{\left\lfloor{N}/{J}\right\rfloor}(w')$ we have
\begin{equation}\label{e:upper_right_tail_badset_micro}
\begin{split}
&\sum_{w''\in\W_{\left\lfloor\frac{N}{J}\right\rfloor,J'}(w')}
\pp^{J',w''}\bigg(\max_{\substack{v\in V_{J'}(w''):\\\T_v\ge T\textrm{, or }v\in A}}\varphi^{J',w''}_v\ge m_{L'}+z\bigg)\\
&\quad\le C\left(\frac{N}{JJ'}\right)^d\left(\e^{dT}\frac{|A|}{N^d}\left(1+T+\log\left(\frac{N^d}{|A|}\right)\right)^{19/8}+\e^{-cT}\right)(z\vee 1)\e^{-\sqrt{2d}z}.
\end{split}
\end{equation}
\end{lemma}

The main step in the proofs of Lemma \ref{l:upper_right_tail_badset} and Lemma \ref{l:upper_right_tail_badset_micro} is an estimate on the right tail of the maximum on subsets where $\T_\cdot$ is bounded. The same argument implies an estimate for subsets where $\Rone_\cdot$ is bounded, and the following lemma collects both these estimates.

\begin{lemma}\label{l:upper_right_tail_sparse}
Under Assumption \ref{a:logupp} let $N\ge 2$, $w\in\Z^d$ and suppose that $A\subset V_N(w)$. Let $T=\max_{v\in A}\T_v$. Then for any $z\in\R$,
\begin{equation}\label{e:upper_right_tail_sparse}
\pp^{N,w}\left(\max_{v\in A}\varphi^{N,w}_v\ge m_N+z\right)\le C\e^{dT}\frac{|A|}{N^d}\left(1+T+\log\left(\frac{N^d}{|A|}\right)\right)^{19/8}(z\vee 1)\e^{-\sqrt{2d}z}.
\end{equation}
%If instead of \ref{a:logupp} we have \ref{a:logbd}, then for $\delta>0$ and $R=\max_{v\in A}\Rone_v$ we have
%\begin{equation}\label{e:upper_right_tail_sparse_R}
%\pp^{N,w}\left(\max_{v\in A\cap V_N^\delta(w)}\varphi^{N,w}_v\ge m_N+z\right)\le C\e^{d\alpha_{\delta}(R)}\frac{|A|}{N^d}\left(1+\alpha_{\delta}(R)+\log\left(\frac{N^d}{|A|}\right)\right)^{19/8}(z\vee 1)\e^{-\sqrt{2d}z}
%\end{equation}
\end{lemma}

%Note that in the lemma $N$ and $w$ are arbitrary. In particular, we can apply it to a box $ V_{\left\lfloor\frac{N}{J}\right\rfloor}(w')$ and some subset $A$ of it to conclude an upper bound on the maximum on that (macroscopic) subbox of $V_N$.
This estimate should be optimal (except possibly regarding the exponent of the logarithmic correction term).
Similar estimates can be found in various places in the literature, e.g. in \cite[Proposition 1.1]{DRZ17} (but with exponent $1/2$ instead of $1$ for $\frac{|A|}{N^d}$), or in \cite[Lemma 10.4]{B20} (but only for the case of the Gaussian free field). We need the almost optimal 
dependence on $\frac{|A|}{N^d}$, as it is directly related to the exponents admissible in \ref{a:sparseT}.

For the proof we proceed similary to \cite{DRZ17}, and use Gaussian comparison inequalities to reduce the problem to an estimate for BRW, for which explicit calculations are possible.

\begin{proof}
%For \eqref{e:upper_right_tail_sparse}, n
Note that by Assumption \ref{a:logupp} we have for any $u,v\in A$ that
\begin{align}
\Var\varphi^{N,w}_v\le\log N+T\label{e:upper_right_tail_sparse_5},\\
\Var\varphi^{N,w}_v-\E\varphi^{N,w}_v\varphi^{N,w}_u\le \log_+|u-v|+2T.\label{e:upper_right_tail_sparse_6}
\end{align}

Let $N'=2^{n'}\ge N$ be an integer to be chosen shortly. We can define an upscaling map $\Psi^{w,w'}_{N,N'}\colon V_N(w)\to V_{N'}(w')$ by
\begin{equation}\label{e:upscaling}
\Psi^{w,w'}_{N,N'}(v)=w'+\left\lfloor\frac{N'}{N}\right\rfloor(v-w).
\end{equation}

Proceeding as in the proof of \cite[Lemma 2.5]{DRZ17}, we see that there is $N'=2^{n'}$ such that $\log N'-\log N\le T+C$ and such that
\[\pp^{N,w}\left(\max_{v\in A}\varphi^{N,w}_v\ge m_N+z\right)\le 2\pp^{N',w}\left(\max_{v\in \Psi^{w,w}_{N,N'}(A)}\theta^{N',w}_v\ge m_N+z\right),\]
for any $z\in\R$\footnote{In \cite[Lemma 2.5]{DRZ17} $N'$ is not a power of 2, but $\frac{N'}{N}$ is. This does not matter for the argument.}.
 This reduces the proof to the study of $\theta^{N,w}$. As in \cite[Section 3.4]{BDZ16} we can associate to $\theta^{N',w}$ a $2d$-ary branching Brownian motion such that $\theta^{N',w}_v$ is the value at time $n'$ of the branch corresponding to $v$.

Let $a$ be the smallest integer such that $|A|\le2^{da}$. We claim that
\begin{equation}\label{e:upper_right_tail_sparse_1}
\pp^{N',w}\left(\max_{v\in \Psi^{w,w}_{N,N'}(A)}\theta^{N',w}_v\ge m_{N'}+z\right)\le \frac{C}{2^{d(n'-a')}}(1+(n'-a'))^{19/8}(z\vee 1)\e^{-\sqrt{2d}z}.
\end{equation}
Once we have shown this, the lemma easily follows. Indeed, we have that $m_{N'}\le m_N+\sqrt{2d}(T+C)$ and so \eqref{e:upper_right_tail_sparse_1} implies that
\begin{align*}
&\pp^{N,w}\left(\max_{v\in A}\varphi^{N,w}_v\ge m_N+z\right)\le 2\pp^{N',w}\left(\max_{v\in\Psi^{w,w}_{N,N'}(A)}\theta^{N',w}_v\ge m_N+z\right)\\
&\le\frac{C}{2^{d(n'-a')}}(1+(n'-a'))^{19/8}((z-(m_N'-m_N))\vee 1)\e^{-\sqrt{2d}(z-(m_N'-m_N))}\\
&\le\frac{C|A|^d}{N'^d}\left(1+\log\left(\frac{N'^d}{|A|}\right)\right)^{19/8}\e^{-\sqrt{2d}z}\e^{2dT},
\end{align*}
which implies \eqref{e:upper_right_tail_sparse}.

It remains to prove \eqref{e:upper_right_tail_sparse_1}. We can assume that $a\ge1$. As $2^a\le2N$, while $2^{n'}\ge2^9N$, we know that $n'>a+8$. For later use we note that for $1\le a\le n'-8$ we have 
\begin{equation}\label{e:upper_right_tail_sparse_3}
\frac{\log n'-\log a}{n'-a}\le\frac{\log8}{7}.
\end{equation}

If $z\le-\sqrt{\frac{d}{2}}(n'-a)\log2$ then the right-hand side of \eqref{e:upper_right_tail_sparse_1} is at least one (if $C$ is large enough), and so we can assume that $z>-\sqrt{\frac{d}{2}}(n'-a)\log2$. 
Let $\beta=z+\sqrt{\frac{d}{2}}(n'-a)\log2>0$, and consider the event
\[G(\beta)=\bigcup_{v\in \Psi_{N,N'}(A)}\bigcup_{0\le t\le a}\left\{\theta^{N',w}_v(t)\ge\beta+1+\frac{t m_{2^{a}}}{a}+10\log_+(t\wedge(a-t))\right\}\]
By \cite[Lemma 3.7]{BDZ16} (or rather its obvious extension to $d$ dimensions), we have that
\begin{align}
\pp(G(\beta))&\le C(\beta\vee1)\e^{-\sqrt{2d}\beta}
\le C\left((z\vee 1)+\frac{\sqrt{d}(n'-a)\log2}{\sqrt{2}}\right)\e^{-\sqrt{2d}z}\frac{1}{2^{d(n'-a)}}\nonumber\\
&\le C(z\vee 1)(1+(n'-a))\e^{-\sqrt{2d}z}\frac{1}{2^{d(n'-a)}},
\label{eq-190222}
\end{align}
which can be absorbed into the right side of 
\eqref{e:upper_right_tail_sparse_1}. So 
it remains to consider the case that $G(\beta)$ does not occur.
We have that
\begin{equation}\label{e:upper_right_tail_sparse_2}
\begin{split}
&\pp^{N',w}\left(\max_{v\in \Psi^{w,w}_{N,N'}(A)}\theta^{N',w}_v\ge m_{N'}+z,
G_N(\beta)^\complement\right) \\
&\leq
\sum_{v\in \Psi^{w,w}_{N,N'}(A)}\pp^{N',w}\Big(\theta^{N',w}_v(n')\ge m_{N'}+z,
\\
&\qquad \qquad \qquad \qquad\qquad 
\theta^{N',w}_v(j)\le\beta+1+\frac{j m_{2^{a}}}{a}+10\log_+(t\wedge(a-t))\ \forall j\in\{1,\ldots a\}\Big)\,.
\end{split}
\end{equation}
Denote by $\chi$ the density of 
\[\pp^{N',w}\left(\theta^{N',w}_v(j)\le\beta+1+\frac{jm_{2^{a}}}{a}+10\log_+(t\wedge(a-t))\ \forall j\in\{1,\ldots a\},\theta^{N',w}_v(a)-m_{2^{a}}\in\cdot\right)\]
with respect to one-dimensional Lebesgue measure. A calculation similar to those in the proofs of \cite[Lemma 3.7]{BDZ16} and \cite[Lemma 2.4]{BDZ16b} shows that
\begin{equation}\label{e:upper_right_tail_sparse_4}
\begin{split}
	\chi(x)&\le\frac{C}{2^{da}}(\beta+1)(\beta+1-x)\exp\left(-\left(\sqrt{2d}-C'\frac{\log n}{n}\right)x-\frac{x^2}{2(\log 2)a}\right)\\
	&\le\frac{C}{2^{da}}(\beta+1)(\beta+1-x)\e^{-\sqrt{2d}x}.
	\end{split}
\end{equation}

We can now rewrite \eqref{e:upper_right_tail_sparse_2} by conditioning on the value of $\theta^{N',w}_v(a)$, and using that $\theta^{N',w}_v(n')-\theta^{N',w}_v(a)$ is Gaussian with variance $(n'-a)\log2$ and independent of $\theta^{N',w}_v$ for times less than $a$. We obtain
\begin{align*}
&\pp^{N',w}\bigg(\max_{v\in\Psi^{w,w}_{N,N'}(A)}\theta^{N',w}_v\ge m_{N'}+z, G_N(\beta)^\complement\bigg)\\
&\le\sum_{v\in \Psi^{w,w}_{N,N'}(A)}\int_0^\infty\chi(\beta+1-x)\pp^{N',w}\left(\theta^{N',w}_v(n')-\theta^{N',w}_v(a)\ge m_{N'}+z-(m_{2^{a}}+\beta+1-x)\right)\ud x.
\end{align*}
Note that 
\begin{align*}
m_{N'}-m_{2^{a}}-\beta+z&=\sqrt{2d}(n'-a)\log2-\frac{3}{2\sqrt{2d}}(\log n'-\log a)-\sqrt{\frac{d}{2}}(n'-a)\log2\\
&=\sqrt{\frac{d}{2}}(n'-a)\log2-\frac{3}{2\sqrt{2d}}(\log n'-\log a)
\end{align*}
and that, using \eqref{e:upper_right_tail_sparse_3}, the right-hand side is bounded below by 1 for any $d$. In particular, $m_{N'}+z-(m_{2^{a}}+\beta+1-x)\ge x>0$ for any $x>0$.
Applying now \eqref{e:upper_right_tail_sparse_4} and standard Gaussian tail estimates, we see that
\begin{align*}
&\pp^{N',w}\bigg(\max_{v\in \Psi^{w,w}_{N,N'}(A)}\theta^{N',w}_v\ge m_{N'}+z, G_N(\beta)^\complement\bigg)\\
&\le2^{da}\int_0^\infty\frac{C}{2^{da}}(\beta+1)x\exp\left(-\sqrt{2d}(\beta+1-x)\right)\\
&\qquad\qquad \qquad 
\times \frac{\sqrt{(n'-a)\log2}}{m_{N'}+z-(m_{2^{a}}+\beta+1-x)}\exp\left(-\frac{(m_{N'}+z-(m_{2^{a}}+\beta+1-x))^2}{2(n'-a)\log2}\right)\ud x\\
&\le \int_0^\infty \frac{C\sqrt{n'-a}((n'-a)+z+1)x}{x}
\exp\Big(-\sqrt{2d}\big(z+\sqrt{\frac{d}{2}}(n'-a)\log2+1-x\big)-\\
&\qquad\qquad \qquad\qquad\qquad\qquad\qquad\qquad
\qquad 
\frac{\big(\sqrt{\frac{d}{2}}(n'-a)\log2-\frac{3}{2\sqrt{2d}}(\log n'-\log a)-1+x\big)^2}{2(n'-a)\log2}\Big)\ud x\\
&\le C(1\vee z)(n'-a)^{3/2}\e^{\sqrt{2d}z}\\
&\qquad\times\int_0^\infty\exp\Big(\sqrt{\frac{d}{2}}x-\frac{5d}{4}(n'-a)\log2+\frac{3(\log n'-\log a)}{4\sqrt{2d}(n'-a)\log2}x-\frac{x^2}{2(n'-a)\log2}\Big)\ud x,
\end{align*}
where we have absorbed various bounded error terms into the constant $C$. If the replace the integral over $[0,\infty)$ with one over $\R$, we recognize the integral of a Gaussian function. We find, after some computation,
\begin{align*}
&\pp^{N',w}\bigg(\max_{v\in \Psi^{w,w}_{N,N'}(A)}\theta^{N',w}_v\ge m_{N'}+z, G_N(\beta)^\complement\bigg)\\
%&\le C(1\vee z)(n'-a)^{3/2}\e^{\sqrt{2d}z}\cdot\sqrt{n'-a}\exp\left(-\frac{5d}{4}(n'-a)\log2+\frac14\cdot2(n'-a)\log2\left(\sqrt{\frac{d}{2}}+\frac{3(\log n'-\log a)}{4\sqrt{2d}\log2(n'-a)}\right)^2\right)\\
&\le C(1\vee z)(n'-a)^2\e^{\sqrt{2d}z}\exp\left(-d(n'-a)\log2+\frac38(\log n'-\log a)\right)\\
&\le\frac{C}{2^{d(n'-a)}}(1\vee z)(n'-a)^{19/8}\e^{\sqrt{2d}z}
\end{align*}
where we used ${n'}/{a}\le2(n'-a)$ in the last step. Together with
\eqref{eq-190222}, this completes the proof of \eqref{e:upper_right_tail_sparse_1} and therefore yields \eqref{e:upper_right_tail_sparse}.
\end{proof}

As a consequence of Lemma \ref{l:upper_right_tail_sparse}, we can give an estimate of the right tail of the maximum on those points where $\T_\cdot$ is large.

\begin{lemma}\label{l:tail_max_bad_points}
Under Assumptions \ref{a:logupp} and \ref{a:sparseT} there are $C,c>0$ such that for $J,N\in\N$ with $J\le\frac{N}{2}$, any $w'\in\W_{N,\left\lfloor\frac{N}{J}\right\rfloor}$, $z\in\R$, and any $T$ sufficiently large
\begin{equation}\label{e:tail_max_bad_points}
\pp^{N}\bigg(\max_{\substack{v\in V_{\left\lfloor\frac{N}{J}\right\rfloor}(w')\\\T_v\ge T}}\varphi^{N}_v\ge m_N+z\bigg)\le C(z\vee1)\e^{-\sqrt{2d}z-cT}.
\end{equation}
\end{lemma}
\begin{proof}
We can apply a union bound to estimate
\[
\pp^{N}\bigg(\max_{\substack{v\in V_{\left\lfloor\frac{N}{J}\right\rfloor}(w')\\\T_v\ge T}}\varphi^{N}_v\ge m_N+z\bigg)
\le\sum_{t=T}^\infty\pp^{N}\bigg(\max_{\substack{v\in V_{\left\lfloor\frac{N}{L}\right\rfloor}(w')\\t\le \T_v<t+1}}\varphi^{N}_v\ge m_N+z\bigg).
\]
By Assumption \ref{a:sparseT}, there is $\ep>0$ such that for $t$ sufficiently large the set of points $v\in V_{\left\lfloor\frac{N}{J}\right\rfloor}(w')$ where $t\le \T_v<t+1$ has cardinality at most $C\left(\frac{N}{J}\right)^d\e^{-(d+\ep)t}$. So an application of Lemma \ref{l:upper_right_tail_sparse} implies for any $r$ sufficiently large that
\begin{align*}
\pp^{\corO{N}}\bigg(\max_{\substack{v\in V_{\left\lfloor\frac{N}{J}\right\rfloor}(w')\\\T_v\geq T}}\varphi^{N}_v\ge m_N+z\bigg)
&\le C\sum_{t=T}^\infty\frac{\e^{d(t+1)}}{\e^{(d+\ep)t}L^d}\left(1+t+(d+\ep)t\log L^d\right)^{19/8}(z\vee 1)\e^{-\sqrt{2d}z}\\
&\le C_\ep(z\vee 1)\e^{-\sqrt{2d}z}\sum_{t=T}^\infty \e^{-\ep t/2}.
\end{align*}
The lemma follows.
\end{proof}

Lemma \ref{l:upper_right_tail_badset} is now a straightforward consequence of Lemma \ref{l:upper_right_tail_sparse} and Lemma \ref{l:tail_max_bad_points}.

\begin{proof}[Proof of Lemma \ref{l:upper_right_tail_badset}]
First use Lemma \ref{l:tail_max_bad_points} to estimate the maximum over the points $v$ where $\T_v\ge T$, and then use Lemma \ref{l:upper_right_tail_sparse} to estimate the maximum over the points $v\in A$ where $\T_v\le T$.
\end{proof}

For the proof of Lemma \ref{l:upper_right_tail_badset_micro} we need a version of Lemma \ref{l:tail_max_bad_points} where one takes the average over microscopic cubes.
\begin{lemma}\label{l:tail_max_bad_points_micro}
Under assumptions \ref{a:logupp} and \ref{a:sparseT} there are  $C,c>0$ such that for any $J,N\in\N$ with $J\le\frac{N}{2}$, any $w'\in\W_{N,\left\lfloor\frac{N}{J}\right\rfloor}$, $z\in\R$, and any $T$ sufficiently large,
\begin{equation}\label{e:tail_max_bad_points_micro}
\sum_{w''\in\W_{\left\lfloor\frac{N}{J}\right\rfloor,J'}(w')}
\pp^{J',w''}\bigg(\max_{\substack{v\in V_{J'}(w'')\\\T_v\ge T}}\varphi^{J',w''}_v\ge m_{J'}+z\bigg)\le C\left(\frac{N}{JJ'}\right)^d(z\vee1)\e^{-\sqrt{2d}z-cT}.
\end{equation}
\end{lemma}
\begin{proof}
As in the proof of Lemma \ref{l:tail_max_bad_points} we can apply an union bound and then Lemma \ref{l:upper_right_tail_sparse} to obtain that
the left hand side of \eqref{e:tail_max_bad_points_micro} is bounded above by
\begin{equation}\label{e:tail_max_bad_points_micro1}
\begin{split}
&\sum_{w''\in\W_{\left\lfloor\frac{N}{J}\right\rfloor,J'}(w')}\sum_{t=T}^\infty\pp^{J',w''}\bigg(\max_{\substack{v\in V_{L'}(w'')\\t\le \T_v<t+1}}\varphi^{J',w''}_v\ge m_{J'}+z\bigg)\\
&\quad\le C\sum_{t=T}^\infty\sum_{w''\in\W_{\left\lfloor\frac{N}{J}\right\rfloor,J'}(w')} \e^{d(t+1)}\frac{|A_{j,w''}|}{L'^d}\Big(1+t+\log\big(\frac{J'^d}{|A_{t,w''}|}\big)\Big)^{19/8}(z\vee 1)\e^{-\sqrt{2d}z},
\end{split}
\end{equation}
where 
\[A_{t,w''}=\left\{v\in V_{L'}(w'')\colon t\le \T_v<t+1\right\}.\]
We have 
\[\sum_{w''\in\W_{\left\lfloor\frac{N}{J}\right\rfloor,J'}(w')}\frac{|A_{t,w''}|}{L'^d}\le C\left(\frac{N}{JJ'}\right)^d\e^{-(d+\ep)t}\]
by Assumption \ref{a:sparseT}, and trivially $0\le{|A_{t,w''}|}/{J'^d}\le1$ for any $w''$. Now we can apply Jensen's inequality to the function $x\mapsto x(1+t+\log x)^{19/8}$ (which is concave on the interval $[0,1]$ for any $t\ge1$) on the right-hand side of \eqref{e:tail_max_bad_points_micro1} to see that
the left hand side of \eqref{e:tail_max_bad_points_micro} is bounded above by
%&\sum_{w''\in\W_{\left\lfloor\frac{N}{J}\right\rfloor,J'}(w')}
%\pp^{J',w''}\left(\max_{\substack{v\in V_{J'}(w'')\\\T_v\ge T}}\varphi^{J',w''}_v\ge m_{J'}+z\right)\\
\[ C\left(\frac{N}{JJ'}\right)^d\sum_{t=T}^\infty\frac{\e^{d(t+1)}}{\e^{(d+\ep)t}}\left(1+t+(d+\ep)t\right)^{19/8}(z\vee 1)\e^{-\sqrt{2d}z},\]
which immediately implies the result.
\end{proof}

\begin{proof}[Proof of Lemma \ref{l:upper_right_tail_badset_micro}]
As in the proof of Lemma \ref{l:upper_right_tail_badset}, we first use Lemma \ref{l:tail_max_bad_points_micro} to estimate the maximum over the points $v$ where $\T_v\ge T$, and then use Lemma \ref{l:upper_right_tail_sparse} (combined with another application of Jensen's inequality) to estimate the maximum over the points $v\in A$ where $\T_v\le T$.
\end{proof}

\subsection{Further bounds on the tail of the maximum}
\label{subsec-further}

Lemma \ref{l:upper_right_tail_badset} allows us to show that with probability tending to 1 as $R\to\infty$ and $T\to\infty$, the maximum of $\varphi^{N,w_N}$ does not occur at points $v$ where $\T_v\ge T$ or $\Rone_v\ge R$. Thus, for later arguments we will be able to restrict attention to those points where $\T_v\le T$ and $\Rone_v\le R$ for some $R$ and $T$, taking the limit $R\to\infty$ and $T\to\infty$ at the very end. With this in mind, we next state and prove some estimates on the maximum restricted to the points where $\T_v\le T$ or $\Rone_v\le R$. For reasons that will become clear later, we need these estimates not just on $V_N$, but on all the macroscopic subcubes in $\V_{N,\left\lfloor{N}/{J}\right\rfloor}$.

We begin with an upper bound on the right tail, that is just a small variant of Lemma \ref{l:upper_right_tail_sparse}. This time we do not care about the precise dependency in $T$, but need to keep an additional factor $\exp(-z^2/(C\log N))$ on the right hand side.

\begin{lemma}\label{l:upper_right_tail}
Let $N\in\N$, $w\in\Z^d$ and suppose that $A\subset V_N(w)$. Under Assumption \ref{a:logupp} let $T=\max_{v\in A}\T_v$. Then for any $z\in\R$ we have
\begin{equation}\label{e:upper_right_tail}
\pp^{N,w}\Big(\max_{v\in A}\varphi^{N,w}_v\ge m_N+z\Big)\le C(z\vee 1)\e^{-\sqrt{2d}z-\frac{z^2}{C(\log N+T\vee1)}+C(T\vee1)}.
\end{equation}
If instead of \ref{a:logupp} we have \ref{a:logbd}, then for $\delta>0$ and $R=\max_{v\in A}\Rone_v$ we have
\begin{equation}\label{e:upper_right_tail_R}
\pp^{N,w}\Big(\max_{v\in A\cap V_N^\delta(w)}\varphi^{N,w}_v\ge m_N+z\Big)\le C(z\vee 1)\e^{-\sqrt{2d}z-\frac{z^2}{C(\log N+\alpha_{\delta}(R)\vee1)}+C(\alpha_{\delta}(R)\vee1)}.
\end{equation}
\end{lemma}

Note that in the lemma $N$ and $w$ are again arbitrary. So we can apply the lemma for $ V_{\left\lfloor{N}/{J}\right\rfloor}(w')$ and some subset $A$ of it to conclude an upper bound on the maximum on that (macroscopic) subbox of $V_N$. The analogous remark holds true for Lemma \ref{l:lower_right_tail} and \ref{l:upper_left_tail} below.

\begin{proof}
We begin with the proof of \eqref{e:upper_right_tail}. 
As in the proof of Lemma \ref{l:upper_right_tail_sparse} we find an integer $N'=2^{n'}$ such that $T-C\le\log N'-\log N\le T+C$ and such that
\[\pp^{N,w}\Big(\max_{v\in A}\varphi^{N,w}_v\ge m_N+z\Big)\le 2\pp^{N',w}\Big(\max_{v\in \Psi^{w,w}_{N,N'}(A)}\theta^{N',w}_v\ge m_N+z\Big).\]
Obviously $\Psi^{w,w}_{N,N'}(A)\subset V_{N'}(w)$ and so we see that
\begin{equation}\label{e:upper_right_tail1}
	\pp^{N,w}\Big(\max_{v\in A}\varphi^{N,w}_v\ge m_N+z\Big)\le2\pp^{N',w}\Big(\max_{v\in V_{N'}(w)}\theta^{N',w}_v\ge m_N+z\Big),
\end{equation}
and so we only need a right tail bound for the maximum of a BRW. Such an estimate can be shown as in the two-dimensional case in \cite[Lemma 3.8]{BDZ16}, and we find
\[\pp^{N',w}\Big(\max_{v\in V_{N'}(w)}\theta^{N',w}_v\ge m_{N'}+z\Big)\le C(z\vee 1)\e^{-\sqrt{2d}z-\frac{z^2}{C\log N'}}.\]
Combining this with \eqref{e:upper_right_tail1} and using $m_{N'}\le m_N+\sqrt{2d}T+C$ we can now calculate that
\begin{align*}
	&\pp^{N,w}\Big(\max_{v\in A}\varphi^{N,w}_v\ge m_N+z\Big)
	\le 2\pp^{N',w}\Big(\max_{v\in V_{N'}(w)}\theta^{N',w}_v\ge m_{N'}+z-\sqrt{2d}T-C\Big)\\
	&\le C((z-\sqrt{2d}T-C)\vee 1)\exp\Big(-\sqrt{2d}(z-\sqrt{2d}T-C)-(-\frac{(z-\sqrt{2d}T-C)^2}{C\log N'}\Big)\\
	&\le C(z\vee 1)\exp\Big(-\sqrt{2d}z-\frac{z^2}{C\log N'}+C\big(1+T+\frac{T^2}{\log N'}\big)\Big),
\end{align*}
and \eqref{e:upper_right_tail} follows when taking into account that $T-C\le\log N'-\log N\le T+C$.

The proof of \eqref{e:upper_right_tail_R} is analogous. In fact, we used \ref{a:logupp} only via the bounds \eqref{e:upper_right_tail_sparse_5} and \eqref{e:upper_right_tail_sparse_6}, and under Assumption \ref{a:logbd} those bounds hold on $A\cap V_N^\delta(w)$ with $\alpha_\delta(R)$ in place of $T$.
\end{proof}
These upper bounds on the right tail were straightforward. We state next the complementary lower bound on the right tail, which will be slightly more difficult to show. The difference is that for the upper bound it is only helpful to consider the maximum over a subset, while for the lower bound this might be harmful. However, we will assume that ${|A|}/{N^d}$ is large enough. By the pigeon-hole principle we can then find subcubes of $V_N(w)$ which contain many points of $A$, and this is enough to proceed with a Gaussian comparison argument as in \cite{DRZ17}.
\begin{lemma}\label{l:lower_right_tail}
Under Assumption \ref{a:logbd} there are constants $\gamma$, $c$ and $C$ (depending on $d$ only) with the following property: Let $N\in\N$, $w\in\Z^d$ and $A\subset V_N(w)$ with $|A|\ge(1-2^{-(d+1)})N^d$. Let $R=\max_{v\in A}\Rone_v$ and suppose that $N\ge\gamma\exp\left(\alpha_{(d2^{d+3})^{-1}}(R)\right)$. Then for any $z\in[1,\sqrt{\log N}-C(\alpha_{(d2^{d+3})^{-1}}(R)\vee 1)]$ we have
\begin{equation}\label{e:lower_right_tail}
\pp^{N,w}\Big(\max_{v\in A}\varphi^{N,w}_v\ge m_N+z\Big)\ge cz\e^{-\sqrt{2d}z-C\alpha_{(d2^{d+3})^{-1}}(R)}.
\end{equation}
\end{lemma}
\begin{proof}
As the function $\alpha_{\delta}$ is increasing, we know that $\alpha_{\delta}(\Rone_v)\le\alpha_{\delta}(R)$ for any $v\in A$. 
Fix $\delta={1}/{(d2^{d+3})}$. For any $u,v\in A\cap V_N^\delta(w)$ we know that
\begin{equation}\label{e:lower_right_tail1}
\left|\E\varphi^{N,w}_v\varphi^{N,w}_u-\log N+\log_+|u-v|\right|\le \alpha_{\delta}(R).\end{equation}

By a comparison argument similar to that in Lemma \ref{l:upper_right_tail} (see \cite[Lemma 2.6]{DZ14} for details,
%\todo{Write the details? Ofer: I think here, not needed}),
 there is a constant $C_0$ (depending on $d$ only) with the following property: there is $N'=2^{n'}$ with $\log N-\log N'\le \alpha_{\delta}(R)+C_0$, such that if $w'\in V_N(w)$ and $\Psi^{w,w'}_{N,N'}$ is the upscaling map from the proof of Lemma \ref{l:upper_right_tail_sparse}, then
\begin{equation}\label{e:lower_right_tail2}
\pp^{N,w}\Big(\max_{v\in \Psi^{w,w'}_{N,N'}(A\cap V^{2\delta}_{N'}(w'))}\varphi^{N,w}_v\ge m_N+z\Big)\ge \frac12\pp^{N',w'}\Big(\max_{v\in A\cap V^{2\delta}_{N'}(w')}\tilde\theta^{N',w'}_v\ge m_N+z\Big).
\end{equation}
Here we used that $\Psi^{w,w'}_{N,N'}(V^{2\delta}_{N'}(w'))\subset V^\delta_N(w)$, and so \eqref{e:lower_right_tail1} applies to any $v\in\Psi^{w,w'}_{N,N'}(A\cap V^{2\delta}_{N'}(w'))$.

Choosing $\gamma\ge\e^{C_0}$ ensures that we find such an $N'$ with $N'\ge1$, for which \eqref{e:lower_right_tail2} holds. We want to choose $w'$ such that the right-hand side of \eqref{e:lower_right_tail2} is as large as possible. For that purpose note that
\[\Big|A\cup\bigcup_{w'\in\W_{N,N'}(w)}V_{N'}(w')\Big|\ge\Big(1-\frac{1}{2^{d+1}}\Big)N^d+\Big(N'\Big\lfloor\frac{N}{N'}\Big\rfloor\Big)^d-N^d\ge \frac{N^d}{2^{d+1}},\]
while $|\W_{N,N'}(w)|=\left\lfloor\frac{N}{N'}\right\rfloor^d\le\frac{N^d}{N'^d}$. Thus, by the pigeonhole principle there is some $w'$ such that
\[\left|A\cap V_{N'}(w')\right|\ge\frac{1}{2^{d+1}}(N')^d.\]
Our choice of $\delta={1}/{(d2^{d+3})}$ ensures that $V_{N'}^{2\delta}(w')$ contains at least $(N')^d-2d\delta (N')^{d-1}\ge$
$\big(1-\frac{1}{2^{d+2}}\big)(N')^d$ points for any $N\ge1$. Thus
\[\left|A\cap V_{N'}^{2\delta}(w')\right|\ge\frac{1}{2^{d+1}}(N')^d-\frac{1}{2^{d+2}}(N')^d=\frac{1}{2^{d+2}}(N')^d.\]

Let $A':=A\cap V_{N'}^{2\delta}(w')$. We have just seen that 
$|A'|\ge(N')^d/2^{d+2}$, and because of \eqref{e:lower_right_tail2} we only need to give a lower bound for the right tail of $\max_{v\in A'}\tilde\theta^{N',w'}_v$. To be precise, we claim that for any $y\in[1,\sqrt{\log N'}]$,
\begin{equation}\label{e:lower_right_tail3}
\pp^{N',w'}\left(\max_{v\in A'}\tilde\theta^{N',w'}_v\ge m_{N'}+y\right)\ge cy\e^{-\sqrt{2d}y}.
\end{equation}
If instead of $\max_{v\in A'}$ we had $\max_{v\in V_{N'}(w')}$, this would be 
analogous to the proof of \cite[Lemma 3.7]{DZ14}. Our $A'$ contains a positive fraction of the points in $V_{N'}(w')$, and so we can hope that \eqref{e:lower_right_tail3} still holds. Indeed, the proof of \cite[Lemma 3.7]{DZ14} uses a second-moment method, and for that it is only helpful if one considers fewer events, as long as the first moment is not too reduced. The details
follow.

As in \cite[Section 4.2]{BDZ16}, we associate to $\tilde\theta^{N',w'}$ a family of Brownian motions $\tilde\theta^{N',w'}(t)$
such that $\tilde\theta^{N',w'}_v=\tilde\theta^{N',w'}_v(n')$ and 
$\E^{N',w'} \tilde\theta^{N',w'}_v\tilde\theta^{N',w'}_u=\E^{N',w'} \tilde\theta^{N',w'}_v(n')\tilde\theta^{N',w'}_u(n')$.
%is the value at time $n'$ of that Brownian motion.
%\todo{The barrier events here are in discrete time. Later, in the proof of Lemma 3.5 they are in continuous time. Should we do this consistently?}
Consider the events
\[\Ev_{F,v}^{N'}(y)=\left\{\tilde\theta^{N',w'}_v(t)\le y+\frac{tm_{N'}}{n'}\ \forall t\in [1,n'],
%\{1,\ldots,n'\},
\tilde\theta^{N',w'}_v(n')\in [m_{N'}+y,m_{N'}+y+1]\right\}\]
and the random variable
\[\Lambda_F(y)=\sum_{v\in A'}\I_{\Ev_{F,v}^{N'}(y)}.\]
A calculation analogous to the proof of \cite[Lemma 3.7]{DZ14} shows that
\[\pp^{N',w'}\left(\Ev_{F,v}^{N'}(y)\right)\ge \frac{c}{2^{dn'}}z\e^{-\sqrt{2d}y},\]
and thus
\[
\E^{N',w'}\Lambda_F(y)\ge cy\e^{-\sqrt{2d}y}.
\]
The second moment of $\Lambda_F(y)$ is bounded by that of $\sum_{v\in V_{N'}(w')}\I_{\Ev_{F,v}^{N'}(y)}$, and the latter can be bounded again like in the proof of \cite[Lemma 3.7]{DZ14}. We obtain that
\[\E^{N',w'}\Lambda(y)^2\le Cy\exp\left(\frac{m_{N'}}{\log N'}y\right)\le Cy\e^{-\sqrt{2d}y}\]
From the Paley-Zygmund inequality we can thus conclude
\[\pp^{N',w'}\left(\max_{v\in A\cap V^{2\delta}_{N'}(w')}\tilde\theta^{N',w'}_v\ge m_{N'}+y\right)\ge\pp^{N',w'}(\Lambda(y)>0)\ge\frac{\left(\E^{N',w'}\Lambda_F(y)\right)^2}{\E^{N',w'}\Lambda_F(y)^2}\ge cy\e^{-\sqrt{2d}y}.\]
This shows \eqref{e:lower_right_tail3}. We now recall \eqref{e:lower_right_tail2} and use that $m_N-m_{N'}\le \sqrt{2d}\alpha_{\delta}(R)+C$ to see that for any $z\in[1,\sqrt{\log N'}-\sqrt{2d}\alpha_{\delta}(R)-C]$ we have
\begin{align*}
&\pp^{N,w}\Big(\max_{v\in \Psi^{w,w'}_{N,N'}(A\cap V^{2\delta}_{N'}(w'))}\varphi^{N,w}_v\ge m_N+z\Big)\\
&\ge \frac12\pp^{N',w'}\Big(\max_{v\in A\cap V^{2\delta}_{N'}(w')}\tilde\theta^{N',w'}_v\ge m_{N'}+z+\sqrt{2d}\alpha_{\delta}(R)+C\Big)\\
&\ge \frac12\pp^{N',w'}\Big(\max_{v\in A\cap V^{2\delta}_{N'}(w')}\tilde\theta^{N',w'}_v\ge m_{N'}+z+\sqrt{2d}\alpha_{\delta}(R)+C\Big)\\
&\ge c(z+\sqrt{2d}\alpha_{\delta}(R))\exp\left(-\sqrt{2d}(z+\sqrt{2d}\alpha_{\delta}(R)\right)
\ge cz\e^{-\sqrt{2d}z-C\alpha_{\delta}(R)},
\end{align*}
which implies \eqref{e:lower_right_tail}.
\end{proof}

We can use this lower bound on the right tail directly to deduce an upper bound on the left tail. Again, this result follows from Gaussian comparison arguments together with the pigeon-hole principle.

\begin{lemma}\label{l:upper_left_tail}
Under Assumption \ref{a:logbd} there are constants $\gamma'$, $c$ and $C$ (depending on $d$ only) with the following property: Let $N\in\N$, $w\in\Z^d$ and $A\subset V_N(w)$ with $|A|\ge(1-2^{-(d+3)})N^d$. Let $R=\max_{v\in A}\Rone_v$ and suppose that $N\ge\gamma'\exp\left(\alpha_{(d2^{d+2})^{-1}}(R)\right)$. Then for any $z\in[1,2\sqrt{2d}\log N-C(\alpha_{(d2^{d+2})^{-1}}(R)\vee 1)]$ we have
\begin{equation}\label{e:upper_left_tail}
\pp^{N,w}\left(\max_{v\in A}\varphi^{N,w}_v\le m_N-z\right)\le C\e^{-cz+C\alpha_{(d2^{d+2})^{-1}}(R)}.
\end{equation}
\end{lemma}
\begin{proof}
We use a comparison argument very similar to the one of the proof of Lemma \ref{l:lower_right_tail}. Note that the comparison argument for the analogous result in \cite[Lemma 2.1]{DRZ17}, is set up somewhat differently.

As in Lemma \ref{l:lower_right_tail} we fix $\delta={1}/{(d2^{d+2})}$ 
and observe that for any $u,v\in A\cap V_N^\delta(w)$ we have the correlation bound \eqref{e:lower_right_tail1}.
We now compare $\varphi^{N,w}_{ \Psi^{w,w'}_{N,N'}(v)}$ with 
$\tilde\theta^{N',w'}_v+a_v X$, where $X$ is a standard Gaussian
independent of everything else and the $a_v$ are constants chosen so that
the variance of $\varphi^{N,w}_{ \Psi^{w,w'}_{N,N'}(v)}$ equals that of
$\tilde\theta^{N',w'}_v+a_v X$. Slepian's lemma then implies that there is a choice of $N'=2^{n'}$ with $C_0-1\le\log N-\log N'\le \alpha_{\delta}(R)+C_0$ such that for any $w'\in V_N(w)$,
\begin{equation}\label{e:upper_left_tail1}
\pp^{N,w}\Big(\max_{v\in \Psi^{w,w'}_{N,N'}(A\cap V^{2\delta}_{N'}(w'))}\!\!\varphi^{N,w}_v\le m_N-z\Big)
\le \frac12\pp^{N',w'}\Big(\max_{v\in A\cap V^{2\delta}_{N'}(w')}
\!\!\tilde\theta^{N',w'}_v\le m_N-\frac z2\Big)+\pp\big(X\le-\frac z2\big).
\end{equation}
We choose $\gamma=\e^{C_0}$ which ensures that a choice of $N'$ with $N'\ge1$ is possible.

The second summand in \eqref{e:upper_left_tail1} is obviously bounded by $Cz\e^{-cz}$, and so we focus on the first one. 
We are free to choose $w'$ in \eqref{e:upper_left_tail1}. For that purpose we use again the pigeonhole principle, however this time we need to be slightly more careful in the estimates. Because $\frac{N}{N'}\ge\e^{C_0-1}$, we can control $\left\lfloor\frac{N}{N'}\right\rfloor$ by $\frac{N}{N'}$ and $C_0$, and in fact, by making $C_0$ larger, if necessary, we can ensure that
\[\Big|\bigcup_{w'\in\W_{N,N'}(w)}V_{N'}(w')\Big|=(N')^d\Big\lfloor\frac{N}{N'}\Big\rfloor^d\ge\big(1-\frac{1}{2^{d+3}}\big)N^d.\]
Recall that we assumed $|A|\ge(1-2^{-(d+3)})N^d$. By the pigeon-hole principle, there is now $w'\in\W_{N,N'}(w)$ with 
\[\Big|A\cap V_{N'}(w')\Big|\ge\big(1-\frac{1}{2^{d+2}}\big)(N')^d,\]
and hence also
\[\Big|A\cap V^{2\delta}_{N'}(w')\Big|\ge\big(1-\frac{1}{2^{d+1}}\big)(N')^d.\]
We fix this choice of $w'$ for the remainder of the proof.

Defining $A'=A\cap V_{N'}^{2\delta}(w')$, we claim that for any $y\in[0,2\sqrt{2d}\log N']$ that
\begin{equation}\label{e:upper_left_tail2}
\pp^{N',w'}\left(\max_{v\in A'}\tilde\theta^{N',w'}_v\le m_{N'}-y\right)\le C\e^{-cy}
\end{equation}
To prove this, we proceed similarly to the proof of \cite[Lemma 2.8]{DRZ17} (where the analogous result with $\max_{v\in V_{N'}}$ instead of $\max_{v\in A'}$ was shown). That is, we pick an integer $N''=2^{n''}$ to be fixed later, and consider the boxes $V_{N''}(w'')$ for $w''\in\W_{N',2N''}(w')$. This is a collection of $2^{d(n'-n''-1)}$ boxes of sidelength $N''$ and with pairwise distance at least $N''$. We can now compare the maximum of $\tilde\theta^{N',w'}$ with the maxima of the (independent) MBRWs $\tilde\theta^{N'',w''}$ for $w''\in\W_{N',2N''}(w')$. As in the proof of \cite[Lemma 2.8]{DRZ17} we see that (for $X$ again a standard Gaussian independent of everything else)
\begin{equation}\label{e:upper_left_tail4}
\begin{split}
&\pp^{N',w'}\left(\max_{v\in A'}\tilde\theta^{N',w'}_v\le m_{N'}-y\right)\\
&\le\prod_{w''\in\W_{N',2N''}(w')}\pp^{N'',w''}\left(\max_{v\in A'\cap V_{N''}(w'')}\tilde\theta^{N'',w''}_v\le m_{N'}-\frac y2\right)+\pp\left(\sqrt{\log N'-\log N''}X\le -\frac y2\right).
\end{split}
\end{equation}
As $|A'|\ge\left(1-\frac{1}{2^{d+1}}\right)(N')^d$, we know that 
\[\Big|A'\cap\bigcup_{w''\in\W_{N',2N''}(w')}V_{N''}(w'')\Big|\ge\frac{1}{2^{d+1}}(N')^d=\frac12\Big|\bigcup_{w''\in\W_{N',2N''}(w')}V_{N''}(w'')\Big|.\]
Therefore, by the pigeon-hole principle for at least $\frac23$ of the possible $w''$ we must have 
\begin{equation}\label{e:upper_left_tail3}
\left|A'\cap V_{N''}(w'')\right|\ge\frac14(N'')^d.
\end{equation}
Denote by $\mathcal{S}$ the collection of $w''$ for which \eqref{e:upper_left_tail3} holds. We have just argued that $|\mathcal{S}|\ge\frac23 2^{d(n'-n''-1)}\ge c\left(\frac{N'}{N''}\right)^d$.

An argument analogous to that 
in the proof of Lemma \ref{l:lower_right_tail} 
shows that there is a universal constant $c_0$ such that 
\[\pp^{N'',w''}\left(\max_{v\in A'\cap V_{N''}(w'')}\tilde\theta^{N'',w''}_v\ge m_{N''}\right)\ge c_0\]
whenever $\left|A'\cap V_{N''}(w'')\right|\ge c\left|V_{N''}(w'')\right|$ and so in particular whenever $w''\in\mathcal{S}$.

If we choose $N''$ as the smallest power of 2 larger than $N'\exp\left(-\frac{y}{2\sqrt{2d}}\right)\ge1$ then $m_{N'}-m_{N''}\le\frac y2$ and so
\begin{align*}
	\pp^{N'',w''}\left(\max_{v\in A'\cap V_{N''}(w'')}\tilde\theta^{N'',w''}_v\le m_{N'}-\frac y2\right)&\le \pp^{N'',w''}\left(\max_{v\in A'\cap V_{N''}(w'')}\tilde\theta^{N'',w''}_v\le m_{N''}\right)\le 1-c_0.
\end{align*}
This allows to estimate the first term in \eqref{e:upper_left_tail4}. For the second term we just note $\log N'-\log N''$ is of order $y$, and so that term is exponentially small in $y$. In summary, we see that
\begin{align*}
	\pp^{N',w'}\left(\max_{v\in A'}\tilde\theta^{N',w'}_v\le m_{N'}-y\right)&\le(1-c_0)^{|\mathcal{S}|}+C\e^{-cy}
	\le C\e^{-c\e^{cy}}+C\e^{-cy}\le C\e^{-cy}.
\end{align*}
This establishes \eqref{e:upper_left_tail2}. We can now complete the proof of the lemma. Using \eqref{e:upper_left_tail2} to estimate the first summand in \eqref{e:upper_left_tail1}, we obtain for any $z$ such that $z-\sqrt{2d}\alpha_{\delta}(R)\le2\sqrt{2d}\log N'$
\begin{align*}
	&\pp^{N,w}\Big(\max_{v\in \Psi^{w,w'}_{N,N'}(A\cap V^{2\delta}_{N'}(w'))}\varphi^{N,w}_v\le m_N-z\Big)\\
	&\le \frac12\pp^{N',w'}\Big(\max_{v\in A\cap V^{2\delta}_{N'}(w')}\tilde\theta^{N',w'}_v\le m_{N'}+\sqrt{2d}\alpha_{\delta}(R)-\frac z2\Big)+C\e^{-cz}\\
	&\le C\e^{-c((z-\sqrt{2d}\alpha_{\delta}(R))\vee0)}+C\e^{-cz}
	\le C\e^{-cz+C\alpha_{\delta}(R)},
\end{align*}
which implies \eqref{e:upper_left_tail}.
\end{proof}

As the next theorem shows,
the tail bounds of this subsection 
imply the tightness of the maximum of $\varphi^{N}$.
\begin{theorem}
\label{t:tightness}
Suppose that Assumptions \ref{a:logupp}, \ref{a:logbd}, \ref{a:sparseT}, \ref{a:sparseR} hold. Then for each relatively open subset $U\subset[0,1]^d$,
the sequence of random variables 
$\max_{\substack{v\in V_N\\v/(N-1)\in U}}\varphi^{N}_v-m_N$
 is tight.
In particular, $\max_{v\in V_N}\varphi^{N}_v-m_N$ is tight.
\end{theorem}
\begin{proof}
By Lemma \ref{l:upper_right_tail_badset} there is a large $T$ such that for any sufficiently large $N$ we have the bound
\[\pp^{N}\Big(\max_{\substack{v\in V_N\\v/(N-1)\in U\\}}\varphi^{N,w}_v\ge m_N+z\Big)\le C(z\vee1)\e^{-\sqrt{2d}z-cT}+C(z\vee1)\e^{-\sqrt{2d}z+CT}.\]
Taking the limit $N\to\infty$ and then $z\to\infty$, we obtain the 
tightness of the upper tail.

For the lower tail we note that because $U$ is open, for every $J$ sufficiently large and for every $N\ge J$ at least one of the cubes $V_{\left\lfloor\frac{N}{J}\right\rfloor}(w')$ for $w'\in\W_{N,\left\lfloor\frac{N}{J}\right\rfloor}$ is such that $\frac{1}{N-1}V_{\left\lfloor\frac{N}{J}\right\rfloor}(w')\subset U$. We fix such a choice of $J$, and denote the corresponding sequence of points by $w'_N$. 
By Assumption \ref{a:sparseR} there exist $R>0$ such that for all $N$ sufficiently large at most $\frac{1}{2^{d+3}}\left(\left\lfloor\frac{N}{J}\right\rfloor\right)^d$ of the points in $V_{\left\lfloor\frac{N}{J}\right\rfloor}(w')$ satisfy $\Rone_\cdot\ge R$. Applying Lemma \ref{l:upper_left_tail} to the cube $V_{\left\lfloor\frac{N}{J}\right\rfloor}(w'_N)$ and the set of points where $\Rone_\cdot\ge R$, we obtain the tightness of the lower tail as well.
\end{proof}

\subsection{Geometry of the near-maximizers}
For the following arguments it is important to understand the
set of near-maximizers, i.e. those vertices $v$ where $\varphi^{N}_v$ is of the order of $m_N$. It turns out that, if we restrict attention to those points where $\T$ (or $\Rone$) stay bounded, then any two near-maximizers are either microscopically close (i.e. at distance of order 1 from each other), or macroscopically far apart (i.e. at distance of order $N$ from each other). Let us make this precise.
\begin{lemma}\label{l:geom_nearmax}
Under Assumption \ref{a:logbd} for $\delta>0$ there is a constant $c$ such that for any sequence $(A_N)_{N\in\N}$ of subsets of $\Z^d$ with $A_N\subset V_N$, if $R:=\sup_{N\in\N}\max_{v\in A_N}\Rone_v<\infty$, we have
\begin{equation}\label{e:geom_nearmax}
\limsup_{L\to\infty}\limsup_{N\to\infty}\pp^{N}\Big(\exists u,v\in A_N\cap V_N^\delta\colon L\le |u-v|\le \frac NL,\varphi^{N,w_N}_u\wedge\varphi^{N,w_N}_v\ge m_N-c\log\log L\Big)=0.
\end{equation}
\end{lemma}
Before we begin the proof, we point out that Lemma \ref{l:geom_nearmax} can be combined with the tailbounds of the previous subsection to yield the analogous statement on the geometry of the near-maximizers over all of $V_N$ (the analogue of \cite[Theorem 1.1]{DZ14}). Although not used in the present paper, this might be of independent interest, and so we state it separately.

\begin{theorem}\label{t:geom_nearmax}
Suppose that Assumptions \ref{a:logupp}, \ref{a:logbd}, \ref{a:sparseT}, \ref{a:sparseR} hold. Then there is a constant $c$ such that
\begin{equation}\label{e:geom_nearmax2}
\limsup_{L\to\infty}\limsup_{N\to\infty}\pp^{N}\left(\exists u,v\cap V_N\colon L\le |u-v|\le \frac NL,\varphi^{N}_u\wedge\varphi^{N}_v\ge m_N-c\log\log L\right)=0.
\end{equation}
\end{theorem}
\begin{proof}
In view of Lemma \ref{l:geom_nearmax}, it suffices to show that for each fixed $L$
\[
\limsup_{T\to\infty}\limsup_{\substack{\delta\to0\\R\to\infty}}\limsup_{N\to\infty}\pp^{N}\Big(\max_{\substack{v\in V_N\\v\notin V^\delta_N\text{ or }\T_v\ge T\text{ or }\Rone_v\ge R}}\varphi^{N}_v\ge m_N-c\log\log L\Big)=0.
\]
But this follows immediately from Lemma \ref{l:upper_right_tail_badset}.
\end{proof}

We can now turn to the proof of Lemma \ref{l:geom_nearmax}. For this we could proceed similarly as in \cite{DZ14} or \cite{DRZ17}. The main difference would be that we restrict $u,v$ to a smaller set, and this only helps. However, there is an even faster way. Namely we will use comparison inequalities to deduce our result directly from the one in \cite[Lemma 3.3]{DRZ17}.

\begin{proof}[Proof of Lemma \ref{l:geom_nearmax}]
We use Slepian's inequality to compare the maxima of the two Gaussian processes
\[\Big\{\varphi^{N}_u+\varphi^{N}_v+a_{u,v}X\Big| u,v\in A_N\cap V_N^\delta\colon L\le |u-v|\le \frac NL\Big\}\]
and 
\[\Big\{\tilde\theta^{N'}_{\Psi_{N,N'}^{0,0}(u)}+\tilde\theta^{N'}_{\Psi_{N,N'}^{0,0}(v)}\Big| u,v\in A_N\cap V_N^\delta\colon L\le |u-v|\le \frac NL\Big\}.\]
Here $\Psi$ is an in \eqref{e:upscaling}, $X$ is a standard Gaussian independent of everything else and $a_{u,v}$ is chosen in such a way that the variances match. As in previous arguments, we find that there is a choice of $N'$ with $\log N'-\log N\le 4\alpha_{\delta}(R)+C$ such that the assumptions of Slepian's inequality are satisfied, and we conclude that for any $z\in\R$
\begin{align*}
	&\pp^{N}\Big(\max\Big\{\varphi^{N}_u+\varphi^{N}_v\Big| u,v\in A_N\cap V_N^\delta, L\le |u-v|\le \frac NL\Big\}\ge 2m_N-z\Big)\\
	&\le\pp^{N'}\Big(\max\Big\{\tilde\theta^{N'}_{\Psi_{N,N'}^{0,0}(u)}+\tilde\theta^{N'}_{\Psi_{N,N'}^{0,0}(v)}\Big| u,v\in A_N\cap V_N^\delta, L\le |u-v|\le \frac NL\Big\}\ge 2m_N-z\Big).
\end{align*}
The left-hand side here is clearly an upper bound for the probability in \eqref{e:geom_nearmax}. On the other hand, we can make the right-hand side larger by loosening the restrictions on $u,v$. Thereby we see that
\begin{align*}
&\pp^{N}\Big(\exists u,v\in A_N\cap V_N^\delta\colon L\le |u-v|\le \frac NL,\varphi^{N}_u\wedge\varphi^{N}_v\ge m_N-c_0\log\log L\Big)\\
&\le\pp^{N'}\Big(\max\Big\{\tilde\theta^{N'}_u+\tilde\theta^{N'}_v\Big| u,v\in V_{N'}, L\le |u-v|\le \frac{N'}{L}\Big\}\ge 2m_N-2c_0\log\log L\Big)\\
&\le\pp^{N'}\Big(\max\Big\{\tilde\theta^{N'}_u+\tilde\theta^{N'}_v\Big| u,v\in V_{N'}, L\le |u-v|\le \frac{N'}{L}\Big\}\ge 2m_{N'}-C\alpha_{\delta}(R)-2c_0\log\log L\Big).
\end{align*}
If the maximum of $\tilde\theta^{N'}_u+\tilde\theta^{N'}_v$ exceeds $2m_{N'}-2C_0$ for some constant $C_0$, then either there is a pair of points $u,v$ where $\tilde\theta^{N'}_\cdot$ is at least $m_{N'}-3C_0$, or there must exist one point $v$ where $\tilde\theta^{N'}_\cdot$ is at least $m_{N'}+C_0$. Thus,
\begin{equation}\label{e:geom_nearmax4}
\begin{split}
&\pp^{N}\Big(\exists u,v\in A_N\cap V_N^\delta\colon L\le |u-v|\le \frac NL,\varphi^{N}_u\wedge\varphi^{N}_v\ge m_N-c_0\log\log L\Big)\\
&\le\pp^{N'}\Big(\exists u,v\in V_{N'}\colon L\le |u-v|\le \frac{N'}{L},\tilde\theta^{N'}_u\wedge\tilde\theta^{N'}_v\ge m_{N'}-3C\alpha_{\delta}(R)-6c_0\log\log L\Big)\\
&\quad+\pp^{N'}\Big(\exists v\in V_{N'}\colon\tilde\theta^{N'}_v\ge m_{N'}+C\alpha_{\delta}(R)+2c_0\log\log L\Big).
\end{split}
\end{equation}
The second summand here can be bounded using bounds for the right tail of a MBRW (as implied for example by Lemma \ref{l:upper_right_tail}, or by \cite[Lemma 2.7]{DRZ17}). Taking the limit $N\to\infty$ and then $L\to\infty$, this summand vanishes (for any $c_0$). For the first summand in \eqref{e:geom_nearmax4}, we can apply \cite[Lemma 3.3]{DRZ17} to see that for a sufficiently small $c_0$  it also vanishes in the limit $N\to\infty$ and then $L\to\infty$. This completes the proof.
\end{proof}

\section{Convergence of the maximum of the Gaussian field: proof of Theorem \ref{t:mainthm}}
\label{sec-3}
As in \cite{DRZ17}, the proof of convergence of the maximum of $\varphi^N_\cdot$ is built on 
constructing an easier to analyze Gaussian field, and using comparison theorems for Gaussian processes to relate the two. This section is devoted to the construction of the approximating field and to a proof of Theorem \ref{t:mainthm}. 
After introducing in Section \ref{sec-3.1}
a quick comparison with processes augmented with independent Gaussians,
we provide in Section \ref{sec-approx} the construction of the approximating
fields. Section \ref{sec-3.3} then provides the proof of Theorem \ref{t:mainthm}.
\subsection{Preliminary results}
\label{sec-3.1}
In the next  subsection we will construct an approximation to $\varphi^N$ for which we can control the behaviour of the maximum, and show that in a suitable limit the maxima of $\varphi^N$ and the approximation are close. In this section we will lay the groundwork for that by showing that various modifications do not  significantly change 
 the maximum of a log-correlated Gaussian field. The two results are similar to results in \cite{DRZ17}, and the proofs will also be straightforward adaptions of the proofs there.

As in \cite{DRZ17} we use the Levy metric
\[\dist(\nu_1,\nu_2)=\inf\left\{\delta>0\middle|\nu_1(U^\delta)\le\nu_2(U)+\delta\ \forall U\subset \R\text{ open}\right\}\]
on probability measures on $\R$ (where $U^\delta=\{x\in\R\colon\dist(x,U)<\delta\}$), as well as the one-sided variant
\[\dist^{\le}(\nu_1,\nu_2)=\inf\left\{\delta>0\middle|\nu_1((x,\infty))\le\nu_2((x-\delta,\infty))+\delta\ \forall x\in\R\right\}.\]
It is well-known that $\dist(\cdot,\cdot)$ induces the topology of weak convergence. Furthermore $\dist^{\le}$ measures approximate stochastic domination (in the sense that $\dist^{\le}(\nu_1,\nu_2)=0$ if and only $\nu_2$ stochastically dominates $\nu_1$). Clearly $\dist^{\le}$ is not a metric, however when one symmetrizes it, one obtains a metric that also induces the topology of weak convergence.

Given $L,L'\in\N$ and $\sigma,\sigma'\in\R$, 
consider standard Gaussians $X_B$ for $B\in\V_{N,\left\lfloor{N}/{L}\right\rfloor}$ and $X_{B'}$ for $B'\in\V_{N,L'}$, such that they are all independent.
We define a variant $\tilde\varphi^{N}$ of $\varphi^{N}$ by setting 
$\tilde\varphi^{N}_v=\varphi^{N}_v+\sigma X_B+\sigma'X_{B'}$ on $B\cap B'\cap V_N$.
As we next show,
the law of the maximum of $\tilde \varphi^N$ is (up to a deterministic shift) close to the law of the maximum of $\varphi^{N}$ itself.

\iffalse
The definition of $\tilde\varphi^{N}$ is as follows. Consider standard Gaussians $X_B$ for $B\in\V_{N,\left\lfloor\frac{N}{L}\right\rfloor}$ and $X_{B'}$ for $B'\in\V_{N,L'}$, such that they are all independent. Then we define $\tilde\varphi^{N}_v=\varphi^{N}_v+\sigma X_B+\sigma'X_{B'}$ on $B\cap B'\cap V_N$.
\fi
\begin{lemma}\label{l:approx_fields}
Under Assumption \ref{a:logbd} let $\delta>0$.
%, and let $(w_N)_{N\in\N}$ be a sequence of points in $\Z^d$.
 Let $(A_N)_{N\in\N}$ be a sequence of subsets of $\Z^d$ with $A_N\subset V_N$ and $|A_N|\ge(1-2^{-(d+3)})N^d$, and assume that $R:=\limsup_{N\to\infty}\max_{v\in A_N}\Rone_v<\infty$. Then
\begin{equation}\label{e:approx_fields}
\lim_{L,L'\to\infty}\lim_{N\to\infty}\dist\Big(\max_{v\in A_N\cap V_N^\delta}\varphi^{N},\max_{v\in A_N\cap V_N^\delta}\tilde\varphi^{N}-(\sigma^2+\sigma'^2)\sqrt{\frac{d}{2}}\Big)=0.
\end{equation}
\end{lemma}
\begin{proof}
The proof of \cite[Lemma 3.1]{DRZ17} carries over with very minor changes, as it does not use the geometry of the domain other than via the application of Lemma \ref{l:geom_nearmax}. So we only mention the most important steps.

We define yet another variant of $\varphi^{N}$. Namely let $\varphi'^{N}$ be an independent copy of $\varphi^{N}$ (realized on the same probability space), and define \[\hat\varphi^{N}_v=\varphi^{N}_v+\sqrt{\frac{\sigma_1^2+\sigma_2^2}{\log N}}\varphi'^{N}_v.\]

Clearly, $\hat\varphi^{N}$ is equal in distribution to $\sqrt{1+\frac{\sigma_1^2+\sigma_2^2}{\log N}}\varphi^{N}$. By Lemma \ref{l:upper_right_tail} we know that $\Ev:=\left\{\max_{v\in A_N\cap V_N^\delta}|\varphi^{N}_v|\le2\sqrt{2d}\log N\right\}$ occurs with probability tending to 1 as $N\to\infty$. On the event $\Ev$ we can Taylor-expand the square root and find that 
\[\Big|\max_{v\in A_N\cap V_N^\delta}\sqrt{1+\frac{\sigma_1^2+\sigma_2^2}{\log N}}\varphi^{N}_v-\max_{v\in A_N\cap V_N^\delta}\varphi^{N}_v-m_N-(\sigma^2+\sigma'^2)\sqrt{\frac{d}{2}}\Big|\le \frac{C}{\log N}\]
which implies that
\[\lim_{N\to\infty}\dist\Big(\max_{v\in A_N\cap V_N^\delta}\varphi^{N},\max_{v\in A_N\cap V_N^\delta}\hat\varphi^{N}-(\sigma^2+\sigma'^2)\sqrt{\frac{d}{2}}\Big)=0.\]
Thus we only have to show that
\begin{equation}\label{e:approx_fields1}
\lim_{L,L'\to\infty}\lim_{N\to\infty}\dist\Big(\max_{v\in A_N\cap V_N^\delta}\tilde\varphi^{N},\max_{v\in A_N\cap V_N^\delta}\hat\varphi^{N}\Big)=0.
\end{equation}

For that purpose let $\kappa>0$, and let\footnote{Note that in the corresponding definition in the proof of \cite[Proposition 3.9]{DRZ17}, who use $\kappa=\delta$, 
the second intersection was omitted; this is a mistake there.}
\begin{equation}\label{e:approx_fields3}
V_N^{\delta,\kappa}=V_N^\delta\cap\Big(\bigcup_{w'\in\V_{N,\left\lfloor{N}/{L}\right\rfloor}}V_{\left\lfloor\frac{N}{L}\right\rfloor}^\kappa(w')\Big)\cap\Big(\bigcup_{w''\in\V_{N,L'}}V_{L}^\kappa(w'')\Big).
\end{equation}
Then $|V_N^\delta-V_N^{\delta,\kappa}|\le C_d\kappa N^d$, and so by Lemma \ref{e:upper_right_tail_sparse} we have that the probability of the event 
\[\Big\{\max_{v\in A_N\cap V_N^\delta}\tilde\varphi^{N}_v\neq \max_{v\in A_N\cap V_N^{\delta,\kappa}}\tilde\varphi^{N}_v\Big\}\cup\Big\{\max_{v\in A_N\cap V_N^\delta}\hat\varphi^{N}_v\neq \max_{v\in A_N\cap V_N^{\delta,\kappa}}\hat\varphi^{N}_v\Big\}\]
vanishes in the limit $N\to\infty$ and then $\kappa\to0$. Therefore, \eqref{e:approx_fields1} follows once we show that
\begin{equation}\label{e:approx_fields2}
\lim_{L,L'\to\infty}\lim_{\kappa\to0}\lim_{N\to\infty}\dist\Big(\max_{v\in A_N\cap V_N^{\delta,\kappa}}\tilde\varphi^{N},\max_{v\in A_N\cap V_N^{\delta,\kappa}}\hat\varphi^{N}\Big)=0.
\end{equation}
To see \eqref{e:approx_fields2}, we can proceed exactly as in the proof of \cite[Proposition 3.9]{DRZ17}, by constructing a coupling between the two fields. The crucial point is that we can apply Lemma \ref{l:upper_left_tail} and Lemma \ref{l:geom_nearmax} not just to $\varphi^N$, but also to $\tilde\varphi^N$ and $\hat\varphi^N$. The former lemma ensures the lower tightness of the maxima, and the latter lemma is used to ensure that near-maximizers that share the macroscopic box $B$ also share the microscopic box $B'$.
We omit further details.\end{proof}

\begin{lemma}
Under Assumption \ref{a:logbd} let $\delta>0$, $R>0$. Then there is a function 
$\iota_{t}\colon(0,\infty)\to(0,\infty)$ with $\lim_{\ep\to0}\iota_{t}(\ep)=0$, which depends on $t>0$ only,
%\corO{$\iota_{\alpha_\delta(R)}\colon(0,\infty)\to(0,\infty)$ with $\lim_{\ep\to0}\iota_{\alpha_\delta(R)}(\ep)=0$ which depends on $\alpha_\delta(R)$ only,
 with the following property. Let $(A_N)_{N\in\N}$ be a sequence of subsets of $\Z^d$ with $A_N\subset V_N$ and suppose that 
 $$R\ge\limsup_{N\to\infty}\max_{v\in A_N}\Rone_v<\infty.$$ Let $\bar\varphi^{N}$ be another sequence of Gaussian fields on $V_N$, and suppose that there is $\ep>0$ such that for all $N$ and all $u,v\in A_N$
\begin{align*}
\left|\Var \varphi^{N}_v-\Var\bar\varphi^{N}_v\right|\le\ep,\qquad
\E\bar\varphi^{N}_u\bar\varphi^{N}_v-\E\varphi^{N}_u\varphi^{N}_v\le\ep.
\end{align*}
Then
\[\limsup_{N\to\infty}\dist^{\le}\left(\max_{v\in A_N\cap V_N^\delta}\varphi^{N}-m_N,\max_{v\in A_N\cap V_N^\delta}\bar\varphi^{N}-m_N\right)\le\iota_{\alpha_\delta(R)}(\ep).\]
\end{lemma}
\begin{proof}
The proof is analogous to the proof of Lemma 3.2 in \cite{DRZ17}.
\end{proof}

\subsection{An approximating field}
\label{sec-approx}
Following the approach in \cite[Section 4.1]{DRZ17}, we construct an approximation $\xi^{N,J,J',\delta}$ to $\varphi^{N}$ that consists of rescaled versions of the field on macroscopic and microscopic scales, with a modified branching random walk in between. {Our construction 
differs from that  in  \cite{DRZ17} in one important aspect: for the approach of \cite{DRZ17} to work, it is necessary that
%, but there is one important difference: It is necessary for the argument that
the variances of $\varphi^{N}$ and the approximating field agree, at all vertices.
%. We can achieve this suitable correction terms at each vertex.
 In \cite{DRZ17} the correction terms have variance up to $C\alpha_\delta(R)$ (in our notation), and are controlled uniformly. In our setting, these correction terms blow up as $R\to\infty$, which leads to a loss of control of the tail behaviour of the field. Instead, we estimate the relevant variances more precisely, and this allows us to use correction terms that are independent of $R$ and $T$. However, this comes at the cost of having to use two sets of correction terms, one at macroscopic and one at microscopic scale.
%is similar to the one in \cite{DRZ17}, but there is one important difference: It is necessary for the argument that the variances of $\varphi^{N}$ and the approximating field agree. We can achieve this suitable correction terms at each vertex. In \cite{DRZ17} the correction terms have variance up to $C\alpha_\delta(R)$ (in our notation). In our setting, these correction terms blow up as $R\to\infty$, which leads to a loss of control of the tail behaviour of the field. We estimate the relevant variances more precisely, and this allows us to use correction terms that are independent of $R$ and $T$. However, this comes at the prize of having to use two sets of correction terms: One at macroscopic and one at microscopic scale.

In this and the next section we take particular care to use subscripts and superscripts to indicate on which variables a certain object depends on. For example, our approximating field $\xi^{N,J,J',R,\delta}$ will depend on $N,J,J',R\in\N$ and $\delta>0$, but not on other quantities.

In order to construct $\xi^{N,J,J',R,\delta}$, let $J,J',R\in\N$, and $\delta>0$.
We subdivide $V_N$ into macroscopic boxes of sidelength $\left\lfloor{N}/{J}\right\rfloor$, and then we subdivide these into microscopic boxes of sidelength $J'$. Of course, in general $N$ will not be divisible by $JJ'$, and so there will be some points left over. However, there will be $o(N^d)$, and hence they will not be relevant for the maximum (using Lemma \ref{l:upper_right_tail_badset}).\footnote{In \cite{DRZ17} it is directly assumed that $N$ is a multiple of $JJ'$. We cannot do so here, as this would result in shifting the boxes in $V_N$ in a way that is incompatible with Assumptions \ref{a:micro}, \ref{a:macro} and in particular \ref{a:lln}.} 

To make this precise, let $N^*=JJ'\left\lfloor{N}/{JJ'}\right\rfloor$ be the largest multiple of $JJ'$ that is not larger than $N$. We consider the macroscopic cubes $V_{\left\lfloor{N}/{J}\right\rfloor}(w')$ for $w'\in\W_{N,\left\lfloor{N}/{J}\right\rfloor}$, and in each of them we consider the microscopic cubes $V_{J'}(w'')$ for $w''\in\W_{\left\lfloor{N}/{J}\right\rfloor,J'}(w')$. In this way, the number of microscopic cubes in each of the macroscopic cubes is exactly 
\[\left\lfloor\frac{1}{J'}\left\lfloor\frac{N}{J}\right\rfloor\right\rfloor=\left\lfloor\frac{N}{JJ'}\right\rfloor=\frac{N^*}{JJ'}\]
where the first equality is an easy exercise in arithmetic. We will define $\xi^{N}$ on the the union of these microscopic cubes and then extend it by 0 to $V_N$. The reader is referred to figure \ref{fig:outer} for
a pictorial description of the various boxes entering the construction.

\begin{figure}[h]
%\begin{center}
\begin{tikzpicture}[scale=0.7]

        %    \fill (3,3) circle [radius=3cm] (0, 0) -- (12,6) -- (12, 0) -- cycle (9,3) circle [radius=3cm];
 \fill[fill=gray!70] (0,0) -- (0,10) -- (0.5,10) -- (0.5,0) -- cycle;
 \fill[fill=gray!70] (0,10) -- (10,10) -- (10,9.5) -- (0,9.5) -- cycle;
 \fill[fill=gray!70] (0,0.5) -- (10,0.5) -- (10,0) -- (0,0) -- cycle;
 \fill[fill=gray!70] (10,0) -- (10,10) -- (9.5,10) -- (9.5,0) -- cycle;

 \fill[fill=gray!70] (9,0) -- (9,10) -- (10,10) -- (10,0) -- cycle;
 \fill[fill=gray!70] (0,9) -- (0,10) -- (10,10) -- (10,9) -- cycle;

  \fill[fill=gray!70] (0,1.35) -- (9,1.35) -- (9,1.575) -- (0,1.575) -- cycle;
  \fill[fill=gray!70] (0,2.85) -- (9,2.85) -- (9,3.075) -- (0,3.075) -- cycle;
  \fill[fill=gray!70] (0,4.35) -- (9,4.35) -- (9,4.575) -- (0,4.575) -- cycle;
  \fill[fill=gray!70] (0,5.85) -- (9,5.85) -- (9,6.075) -- (0,6.075) -- cycle;
   \fill[fill=gray!70] (0,7.35) -- (9,7.35) -- (9,7.575) -- (0,7.575) -- cycle;
  \fill[fill=gray!70] (0,8.85) -- (9,8.85) -- (9,9.075) -- (0,9.075) -- cycle;

   \fill[fill=gray!70] (1.35,0) -- (1.35,9) -- (1.575,9) -- (1.575,0) -- cycle;
   \fill[fill=gray!70] (2.85,0) -- (2.85,9) -- (3.075,9) -- (3.075,0) -- cycle;
 \fill[fill=gray!70] (4.35,0) -- (4.35,9) -- (4.575,9) -- (4.575,0) -- cycle (4.425,0);
   \fill[fill=gray!70] (5.85,0) -- (5.85,9) -- (6.075,9) -- (6.075,0) -- cycle;
   \fill[fill=gray!70] (7.35,0) -- (7.35,9) -- (7.575,9) -- (7.575,0) -- cycle;
   \fill[fill=gray!70] (8.85,0) -- (8.85,9) -- (9.075,9) -- (9.075,0) -- cycle (8.925,0);

\draw (0,0) rectangle (10,10);
\draw[dashed] (0.5,0.5) rectangle (9.5,9.5);

\draw(0,0) rectangle (1.5,1.5);
\draw(0,1.5) rectangle (1.5,3);
\draw(0,3) rectangle (1.5,4.5);
\draw(0,4.5) rectangle (1.5,6);
\draw(0,6) rectangle (1.5,7.5);
\draw(0,7.5) rectangle (1.5,9);
%\draw(0,9) rectangle (1.5,1.5);

%\draw[dashed](0.075,0.075) rectangle (1.35,1.35);
%\draw[dashed](0.075,1.575) rectangle (1.35,2.85);
%\draw[dashed](0.075,3.075) rectangle (1.35,4.35);
%\draw[dashed](0.075,4.575) rectangle (1.35,5.85);
%\draw[dashed](0.075,6.075) rectangle (1.35,7.35);
%\draw[dashed](0.075,7.575) rectangle (1.35,8.85);
%%\draw(0.1,9) rectangle (1.4,1.4);

\draw(1.5,0) rectangle (3,1.5);
\draw(1.5,1.5) rectangle (3,3);
\draw(1.5,3) rectangle (3,4.5);
\draw(1.5,4.5) rectangle (3,6);
\draw(1.5,6) rectangle (3,7.5);
\draw(1.5,7.5) rectangle (3,9);
%\draw(1.5,9) rectangle (3,1.5);

%\draw[dashed](1.575,0.075) rectangle (2.85,1.35);
%\draw[dashed](1.575,1.575) rectangle (2.85,2.85);
%\draw[dashed](1.575,3.075) rectangle (2.85,4.35);
%\draw[dashed](1.575,4.575) rectangle (2.85,5.85);
%\draw[dashed](1.575,6.075) rectangle (2.85,7.35);
%\draw[dashed](1.575,7.575) rectangle (2.85,8.85);

\draw(3,0) rectangle (4.5,1.5);
\draw(3,1.5) rectangle (4.5,3);
\draw(3,3) rectangle (4.5,4.5);
\draw(3,4.5) rectangle (4.5,6);
\draw(3,6) rectangle (4.5,7.5);
\draw(3,7.5) rectangle (4.5,9);
%\draw(3,9) rectangle (1.5,1.5);

%\draw[dashed](3.075,0.075) rectangle (4.35,1.35);
%\draw[dashed](3.075,1.575) rectangle (4.35,2.85);
%\draw[dashed](3.075,3.075) rectangle (4.35,4.35);
%\draw[dashed](3.075,4.575) rectangle (4.35,5.85);
%\draw[dashed](3.075,6.075) rectangle (4.35,7.35);
%\draw[dashed](3.075,7.575) rectangle (4.35,8.85);

\draw(4.5,0) rectangle (6,1.5);
\draw(4.5,1.5) rectangle (6,3);
\draw(4.5,3) rectangle (6,4.5);
\draw(4.5,4.5) rectangle (6,6);
\draw(4.5,6) rectangle (6,7.5);
\draw(4.5,7.5) rectangle (6,9);

%\draw[dashed](4.575,0.075) rectangle (5.85,1.35);
%\draw[dashed](4.575,1.575) rectangle (5.85,2.85);
%\draw[dashed](4.575,3.075) rectangle (5.85,4.35);
%\draw[dashed](4.575,4.575) rectangle (5.85,5.85);
%\draw[dashed](4.575,6.075) rectangle (5.85,7.35);
%\draw[dashed](4.575,7.575) rectangle (5.85,8.85);

\draw(6,0) rectangle (7.5,1.5);
\draw(6,1.5) rectangle (7.5,3);
\draw(6,3) rectangle (7.5,4.5);
\draw(6,4.5) rectangle (7.5,6);
\draw(6,6) rectangle (7.5,7.5);
\draw(6,7.5) rectangle (7.5,9);

%\draw[dashed](6.075,0.075) rectangle (7.35,1.35);
%\draw[dashed](6.075,1.575) rectangle (7.35,2.85);
%\draw[dashed](6.075,3.075) rectangle (7.35,4.35);
%\draw[dashed](6.075,4.575) rectangle (7.35,5.85);
%\draw[dashed](6.075,6.075) rectangle (7.35,7.35);
%\draw[dashed](6.075,7.575) rectangle (7.35,8.85);

\draw(7.5,0) rectangle (9,1.5);
\draw(7.5,1.5) rectangle (9,3);
\draw(7.5,3) rectangle (9,4.5);
\draw(7.5,4.5) rectangle (9,6);
\draw(7.5,6) rectangle (9,7.5);
\draw(7.5,7.5) rectangle (9,9);

%\draw[dashed](7.575,0.075) rectangle (8.85,1.35);
%\draw[dashed](7.575,1.575) rectangle (8.85,2.85);
%\draw[dashed](7.575,3.075) rectangle (8.85,4.35);
%\draw[dashed](7.575,4.575) rectangle (8.85,5.85);
%\draw[dashed](7.575,6.075) rectangle (8.85,7.35);
%\draw[dashed](7.575,7.575) rectangle (8.85,8.85);

\draw[|<->|] (0,10.2) -- (10,10.2);
\node at (5,10.5) {\small $N$};

\draw[|<->|] (0,-0.2) -- (9,-0.2);
\node at (5,-0.5) {\small $J\lfloor N/J\rfloor$};

\draw[|<->|] (-0.2,0) -- (-0.2,1.5);
\node at (-1,0.75) {\small $\lfloor N/J\rfloor$};

\draw[|<->|] (10.2,0) -- (10.2,0.5);
\node at (10.8,0.25) {\small $\delta N$};

%\draw[thick] (-0.2,8.85) -- (-0.2,9);
%\draw[thick,->|] (-0.2,8.65) -- (-0.2,8.85);
%\draw[thick,->|] (-0.2,9.2) -- (-0.2,9);
%\node at (-1.2,8.925) {\small $\delta \lfloor N/J\rfloor$};

% Inner block below
\draw (14.5,4.5) circle(3.6);

\draw [thick] plot [smooth, tension=1] coordinates {  (16.5,2.5) (13,0.5) (7.5,1.5)};

\draw [thick] plot [smooth, tension=1] coordinates {  (12.5,6.5) (8,5.5) (6,3)};

%\draw[-] (13.25,3)--(13.25,7);
%\draw[-] (13.5,3)--(13.5,7);
%\draw[-] (13.75,3)--(13.75,7);
%\draw[-] (14,3)--(14,7);
%\draw[-] (14.25,3)--(14.25,7);
%\draw[-] (14.5,3)--(14.5,7);
%\draw[-] (14.75,3)--(14.75,7);
%\draw[-] (15,3)--(15,7);
%\draw[-] (15.25,3)--(15.25,7);
%\draw[-] (15.5,3)--(15.5,7);
%\draw[-] (15.75,3)--(15.75,7);
%\draw[-] (16,3)--(16,7);
%\draw[-] (16.25,3)--(16.25,7);
%\draw[-] (16.5,3)--(16.5,7);
%\draw[-] (16.75,3)--(16.75,7);
%
%\draw[-] (13,3)--(17,3);
%\draw[-] (13,3.25)--(17,3.25);
%\draw[-] (13,3.5)--(17,3.5);
%\draw[-] (13,3.75)--(17,3.75);
%\draw[-] (13,4)--(17,4);
%\draw[-] (13,4.25)--(17,4.25);
%\draw[-] (13,4.5)--(17,4.5);
%\draw[-] (13,4.75)--(17,4.75);
%\draw[-] (13,5)--(17,5);
%\draw[-] (13,5.25)--(17,5.25);
%\draw[-] (13,5.5)--(17,5.5);
%\draw[-] (13,5.75)--(17,5.75);
%\draw[-] (13,6)--(17,6);
%\draw[-] (13,6.25)--(17,6.25);
%\draw[-] (13,6.5)--(17,6.5);
%\draw[-] (13,6.75)--(17,6.75);

 \fill[fill=gray!70] (12.5,2.5) -- (12.5,6.5) -- (12.6,6.5) -- (12.6,2.5) -- cycle;
 \fill[fill=gray!70] (12.72,2.5) -- (12.72,6.5) -- (12.78,6.5) -- (12.78,2.5) -- cycle;
  \fill[fill=gray!70] (12.97,2.5) -- (12.97,6.5) -- (13.03,6.5) -- (13.03,2.5) -- cycle;
 \fill[fill=gray!70] (13.22,2.5) -- (13.22,6.5) -- (13.28,6.5) -- (13.28,2.5) -- cycle;
 \fill[fill=gray!70] (13.47,2.5) -- (13.47,6.5) -- (13.53,6.5) -- (13.53,2.5) -- cycle;
  \fill[fill=gray!70] (13.72,2.5) -- (13.72,6.5) -- (13.78,6.5) -- (13.78,2.5) -- cycle;
 \fill[fill=gray!70] (13.97,2.5) -- (13.97,6.5) -- (14.03,6.5) -- (14.03,2.5) -- cycle;
 \fill[fill=gray!70] (14.22,2.5) -- (14.22,6.5) -- (14.28,6.5) -- (14.28,2.5) -- cycle;
 \fill[fill=gray!70] (14.47,2.5) -- (14.47,6.5) -- (14.53,6.5) -- (14.53,2.5) -- cycle;
  \fill[fill=gray!70] (14.72,2.5) -- (14.72,6.5) -- (14.78,6.5) -- (14.78,2.5) -- cycle;
 \fill[fill=gray!70] (14.97,2.5) -- (14.97,6.5) -- (15.03,6.5) -- (15.03,2.5) -- cycle;
 \fill[fill=gray!70] (15.22,2.5) -- (15.22,6.5) -- (15.28,6.5) -- (15.28,2.5) -- cycle;
 \fill[fill=gray!70] (15.47,2.5) -- (15.47,6.5) -- (15.53,6.5) -- (15.53,2.5) -- cycle;
   \fill[fill=gray!70] (15.72,2.5) -- (15.72,6.5) -- (15.78,6.5) -- (15.78,2.5) -- cycle;
 \fill[fill=gray!70] (15.97,2.5) -- (15.97,6.5) -- (16.03,6.5) -- (16.03,2.5) -- cycle;
 \fill[fill=gray!70] (16.22,2.5) -- (16.22,6.5) -- (16.28,6.5) -- (16.28,2.5) -- cycle;
% \fill[fill=gray!70] (16.97,2.5) -- (16.97,6.5) -- (17,6.5) -- (17,2.5) -- cycle (16.97,2.5);
 \fill[fill=gray!70] (16.22,2.5) -- (16.22,6.5) -- (16.5,6.5) -- (16.5,2.5) -- cycle (16.72,2.5);

\fill[fill=gray!70] (12.5,2.5) -- (16.5,2.5) -- (16.5,2.6) -- (12.5,2.6) -- cycle;
  \fill[fill=gray!70] (12.5,2.72) -- (16.5,2.72) -- (16.5,2.78) -- (12.5,2.78) -- cycle;
  \fill[fill=gray!70] (12.5,2.97) -- (16.5,2.97) -- (16.5,3.03) -- (12.5,3.03) -- cycle;
  \fill[fill=gray!70] (12.5,3.22) -- (16.5,3.22) -- (16.5,3.28) -- (12.5,3.28) -- cycle;
  \fill[fill=gray!70] (12.5,3.47) -- (16.5,3.47) -- (16.5,3.53) -- (12.5,3.53) -- cycle;
  \fill[fill=gray!70] (12.5,3.72) -- (16.5,3.72) -- (16.5,3.78) -- (12.5,3.78) -- cycle;
  \fill[fill=gray!70] (12.5,3.97) -- (16.5,3.97) -- (16.5,4.03) -- (12.5,4.03) -- cycle;
  
  \fill[fill=gray!70] (12.5,4.22) -- (16.5,4.22) -- (16.5,4.28) -- (12.5,4.28) -- cycle;
  \fill[fill=gray!70] (12.5,4.47) -- (16.5,4.47) -- (16.5,4.53) -- (12.5,4.53) -- cycle;
  \fill[fill=gray!70] (12.5,4.72) -- (16.5,4.72) -- (16.5,4.78) -- (12.5,4.78) -- cycle;
  \fill[fill=gray!70] (12.5,4.97) -- (16.5,4.97) -- (16.5,5.03) -- (12.5,5.03) -- cycle;
  
  \fill[fill=gray!70] (12.5,5.22) -- (16.5,5.22) -- (16.5,5.28) -- (12.5,5.28) -- cycle;
  \fill[fill=gray!70] (12.5,5.47) -- (16.5,5.47) -- (16.5,5.53) -- (12.5,5.53) -- cycle;
  \fill[fill=gray!70] (12.5,5.72) -- (16.5,5.72) -- (16.5,5.78) -- (12.5,5.78) -- cycle;
  \fill[fill=gray!70] (12.5,5.97) -- (16.5,5.97) -- (16.5,6.03) -- (12.5,6.03) -- cycle;
  
  \fill[fill=gray!70] (12.5,6.22) -- (16.5,6.22) -- (16.5,6.28) -- (12.5,6.28) -- cycle;

%  \fill[fill=gray!70] (12.5,6.97) -- (16.5,6.97) -- (16.5,6.5) -- (12.5,6.5) -- cycle (12.5,6.97);
\fill[fill=gray!70] (12.5,6.22) -- (16.5,6.22) -- (16.5,6.5) -- (12.5,6.5) -- cycle;

  \fill[fill=black] (14,4.5)--(14,4.75)--(14.25,4.75)--(14.25,4.5)-- cycle;
\fill[fill=black] (15.5,4)--(15.5,4.25)--(15.75,4.25)--(15.75,4)-- cycle;
%
    %\draw[thick,<->] (12.8,5) -- (12.8,5.25);

\draw(12.5,2.5) rectangle (16.5,6.5);

\draw (12.3,5) -- (12.3,5.25);
\draw[->|] (12.3,4.8) -- (12.3,5);
\draw[->|] (12.3,5.45) -- (12.3,5.25);
\node at (11.9,5.1) {\small $J'$};
 \draw[|<->|] (12.5,6.7) -- (16.5,6.7);
 \node at (14.5,7) { $\lfloor N/J\rfloor$};
 \draw[|<->|] (12.5,2.3) -- (16.22,2.3);
\node at (14.4,2) {\small $N^*/J$};
 \draw (16.7,2.5) -- (16.7,2.6);
\draw[->|] (16.7,2.3) -- (16.7,2.5);
\draw[->|] (16.7,2.8) -- (16.7,2.6);
\node at (17.15,3) {\tiny $\delta\lfloor N/J\rfloor$};
 
 %Inner inner block below
\draw (18,9) circle(1.8);

\draw [thick] plot [smooth, tension=1] coordinates {  (17,10) (16,8) (15,5)};

\draw [thick] plot [smooth, tension=1] coordinates {  (19,8) (17,6) (15.25,4.75)};

  \fill[fill=black] (18,8.5)--(18.1,8.5)--(18.1,8.6)--(18,8.6)-- cycle (18,8.5);

  \draw[|<->|] (17,10.1) -- (19,10.1);
 \node at (18,10.4) {\small  $J'$};
 \draw (19.2,8) -- (19.2,8.1);
\draw[->|] (19.2,8.3) -- (19.2,8.1);
\draw[->|] (19.2,7.8) -- (19.2,8);
\node at (19.4,8.5) {\tiny $\delta J'$};

 \fill[fill=gray!70] (17,9.9) -- (19,9.9) -- (19,10) -- (17,10) -- cycle (17,9.9);
 \fill[fill=gray!70] (17,10) -- (17,8) -- (17.1,8) -- (17.1,10) -- cycle (17,10);
 \fill[fill=gray!70] (17,8) -- (19,8) -- (19,8.1) -- (17,8.1) -- cycle (17,8);
 \fill[fill=gray!70] (19,8) -- (19,10) -- (18.9,10) -- (18.9,8) -- cycle (19,8);
   \draw(17,8) rectangle (19,10);
%\draw[thick,-] (15,3) -- (7.5,1.5);
\end{tikzpicture}
%\end{center}
\caption{Block partitions and good points (see \eqref{eqdef-good}). The grey areas are various $\delta$-neighborhoods of the boundary, at various scales, together with rounding effects (the strips on the top and right sides of the $N$-box and the $\lfloor N/J\rfloor$-boxes). The solid black boxes depict regions with $\Rtwo_{w'}>\lfloor N/J\rfloor$ (larger box) or
vertices with $\Rone_v>R$ (smaller black box). The points which are left after removing all the gray and black points will be the good points.} 
\label{fig:outer}
\end{figure}
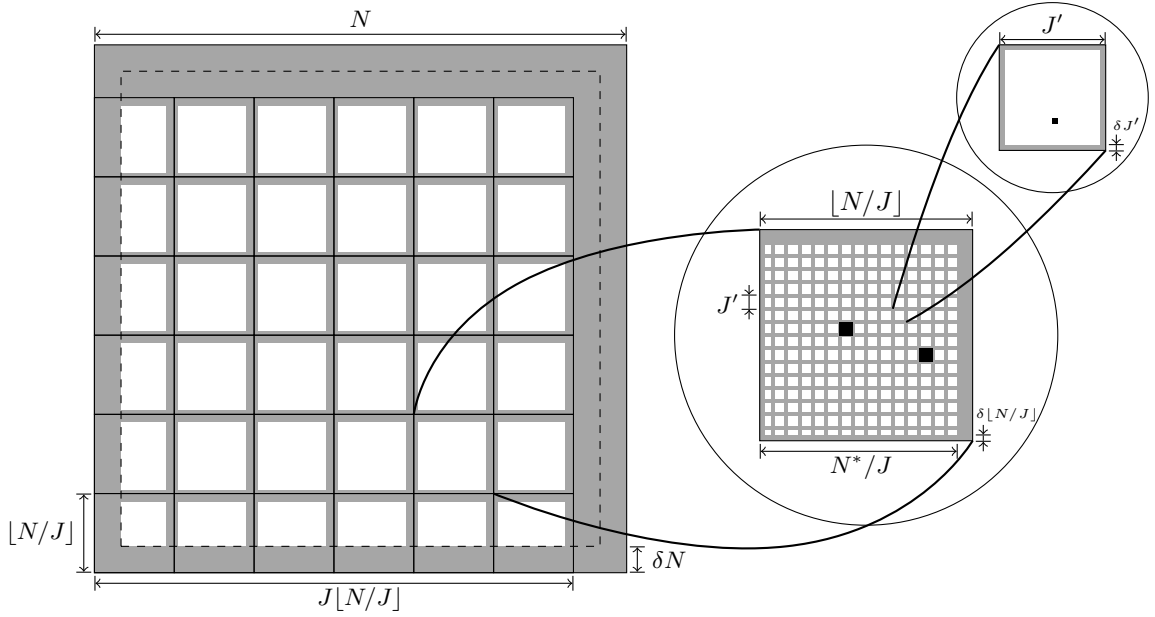

\medskip

Our definition of $\xi^{N,J,J',\delta}$ will be as the sum of fields on macro-, meso- and microscopic scale, and some independent Gaussians that serve as correction terms. We introduce all these objects in the following.

We begin with the macroscopic scale, which will be an upscaling of $\varphi^{J}$. That is, given $\varphi^{J}$ distributed according to $\pp^{J}$, we let \[\xi^{N,J,\mathrm{mac}}_v=\varphi^{J}_{w'^*}\] for all $w'\in\W_{N,\left\lfloor{N}/{J}\right\rfloor}$, where $w'^*\in V_J$ is the preimage of $w'$ under $\Phi^{0,0}_{J,N}$ (the upscaling map $\Phi$ was defined in \eqref{e:upscaling}).

For the microscopic scale we patch together the fields $\varphi^{J',w''}$. That is, given a random field $\varphi^{J',w''}$ distributed according to $\pp^{J',w''}$ for $w''\in\bigcup_{w'\in\W_{N,\left\lfloor{N}/{J}\right\rfloor}}\W_{\left\lfloor{N}/{J}\right\rfloor,J'}$ (and where we take the fields on different microscopic boxes to be independent), we define 
\[\xi^{N,J,J',w_N,\mathrm{mic}}_v=\varphi^{J,w''}_v\]
for all $w''\in\bigcup_{w'\in\W_{N,\left\lfloor{N}/{J}\right\rfloor}}\W_{\left\lfloor{N}/{J}\right\rfloor,J'}$ and all $v\in V_{J'}(w'')$.

Next, for the mesoscopic scale we patch together independent copies of an upscaled MBRW. This is easy to do if it happens that ${N^*}/{JJ'}$ is a power of 2. For the general case we need to make some small adjustments\footnote{This detail is overlooked in \cite{DRZ17}.}.
For that purpose let $i=\left\lceil \log_2\left({N^*}/{JJ'}\right)\right\rceil$, and $I=2^i$, so that $I$ is the smallest power of 2 not smaller than ${N^*}/{JJ'}$. Now let $\tilde\theta^{I}$ be distributed as a MBRW on $V_I$. Then $\Var\tilde\theta^{I}_v=\log I$ for all $v\in V_I$. We want the variance to be exactly $\log({N^*}/{JJ'})$, so define the correction factor 
\[s_{{N}/{J},J'}=\frac{\log\frac{N^*}{JJ'}}{\log I}\]
and note that
\[1\ge s_{{N}/{J},J'}\ge1-\frac{\log 2}{\log\frac{N^*}{JJ'}}.\]
Then $s_{{N}/{J},J'}\tilde\theta^{I,0}$, restricted to $V_{{N^*}/{JJ'}}$, looks like a MBRW on that domain in the sense that we have the estimate 
\begin{equation}\label{e:log_corr_MBRW_corrfact}
\Big|s_{N/J,J'}\E\tilde\theta^{I}_v\tilde\theta^{I}_u-\log \frac{N^*}{JJ'}+\log_+|u-v|_{\sim,N}\Big|\le C_\delta
\end{equation}
for $u,v\in V^\delta_{{N^*}/{JJ'}}$ (as follows easily from \eqref{e:log_corr_MBRW}).

We now take independent copies $(\tilde\theta^{I,(w')})$ of $\tilde\theta^{I}$, indexed by $w'\in \W_{N,\left\lfloor{N}/{J}\right\rfloor}$, and define
\[\xi^{N,J,J',\mathrm{mes}}_{v+w''}=s_{{N}/{J},J'}\tilde\theta^{I,0,(w')}_{(w''-w')/J'}\] for all $w'\in\W_{N,\left\lfloor{N}/{J}\right\rfloor}$, $w''\in\W_{\left\lfloor\frac{N}{J}\right\rfloor,J'}(w')$, and $v\in V_{J'}(w'')$.

To motivate the definition of the correction terms, note that Assumption \ref{a:micro} implies that for $w'\in\W_{N,\left\lfloor{N}/{J}\right\rfloor}$, $w''\in\W_{\left\lfloor{N}/{J}\right\rfloor,J'}(w')$, and $v\in V_{J'}(w'')$ we have, with the symbol $\approx$ meaning with an error that goes to $0$ as $N\to\infty$ followed by $J,J'\to\infty$,
\[
\Var\varphi^{N}_v\approx\log N+g(v,v)+f\left(\frac{v}{N}\right)\]
and
\[
\Var\left(\xi^{N,J,\mathrm{mac}}_v+\xi^{N,J,J',\mathrm{mes}}_v+\xi^{N,J,J',\mathrm{mic}}_v\right)\approx\log N+f\left(\frac{w'^*}{J}\right)+g(v,v)+g(w'^*,w'^*)+f\left(\frac{v-w''}{J'}\right)\]
(where $w'^*$ is the preimage of $w'$ under $\Phi^{0,0}_{J,N}$). By the continuity of $f$ we know that $f\left({v}/{N}\right)\approx f\left({w'}/{J}\right)$, and so the difference of the two variances is approximately $f\left(\frac{v-w''}{J'}\right)+g(w'^*,w'^*)$. These are the terms that we still need to account for.

In order to account for $f\left(\frac{v-w''}{J'}\right)$ note that by Assumption \ref{a:micro} the function $f$ is bounded from above on $(0,1)^d$. Let
\begin{equation}
\label{eq-Gamma}
\Gamma:=\sup_{x\in(0,1)^d}f(x)<\infty\end{equation}
and define $b^{J',\mathrm{mic}}_{v}\colon V_{J'}(0)\to\R$ by setting
\[(b^{J',\mathrm{mic}}_{v})^2=\Gamma-f\left(\frac{v}{J'}\right).\]

In order to account for $g(w'^*,w'^*)$, we first observe that we have an upper bound for $g$. Indeed, let
\[\Gamma_\delta:=\sup_{x\in(\delta,1-\delta)^d}f(x)\le\Gamma.\]
If Assumptions \ref{a:logbd} and \ref{a:micro} hold, we must have $g(v,v)+f(x)\le\alpha_\delta(\Rone_v)$ for any $v\in\Z^d$ and any $x\in(\delta,1-\delta)^d$, as otherwise \ref{a:micro} would contradict \ref{a:logbd} on large cubes $V_N(w)$ with $\frac{v-w}{N}\approx x$. Optimizing over $x$, this means that 
$g(v,v)\le\alpha_\delta(\Rone_v)-\Gamma_\delta$.
Let now 
\begin{equation}
\label{eq-Gammaprime}\Gamma'_{R,\delta}:=\alpha_\delta(R)-\Gamma_\delta\end{equation}
and define $\hat b^{J,R,\delta,\mathrm{mac}}\colon V_J\to[0,\infty)$ by setting
\begin{equation}\label{e:def_hatb}
(\hat b^{J,R,\delta,\mathrm{mac}}_{\hat w})^2=\begin{cases}\Gamma'_{R,\delta}-g(\hat w,\hat w)&\Rone_{\hat w}\le R\\0&\text{else}\end{cases}
\end{equation}
and then 
$b^{J,R,\delta,\mathrm{mac}}\colon \W_{N,\left\lfloor\frac{N}{J}\right\rfloor}\to[0,\infty)$ as
\[b^{J,R,\delta,\mathrm{mac}}_{w'}=\hat b^{J,R,\delta,\mathrm{mac}}_{w'^*}\]
where, once again, $w'^*\in V_J$ is the preimage of $w'$ under $\Phi^{0,0}_{J,N}$.

\medskip

Now we can finally define our approximating field $\xi^{N,J,J',R,\delta}$. We take two collections of independent standard Gaussians. The first, $X^{\mathrm{mic}}_{w''}$, is indexed by $w''\in\bigcup_{w'\in\W_{N,\left\lfloor{N}/{J}\right\rfloor}}\W_{\left\lfloor{N}/{J}\right\rfloor,J'}(w')$, and the second, $X^{\mathrm{mac}}_{w'}$ is indexed by $w'\in\W_{N,\left\lfloor{N}/{J}\right\rfloor}$. We then define for $w'\in\W_{N,\left\lfloor{N}/{J}\right\rfloor}$, $w''\in\W_{\left\lfloor{N}/{J}\right\rfloor,J'}(w')$, and $v\in V_{J'}(w'')$:
\begin{align}
\xi^{N,J,J',R,\delta,\mathrm{coarse}}_v&=\xi^{N,J,\mathrm{mac}}_v+b^{J,R,\delta,\mathrm{mac}}_{w'}X^{\mathrm{mac}}_{w'},\nonumber\\
\xi^{N,J,J',\mathrm{fine}}_v&=\xi^{N,J,J',\mathrm{mes}}_v+\xi^{N,J,J',\mathrm{mic}}_v+b^{J',\mathrm{mic}}_{v-w''}X^{\mathrm{mic}}_{w''},\label{eq-approxfield}\\
\xi^{N,J,J',R,\delta}_v&=\xi^{N,J,J',R,\delta,\mathrm{coarse}}_v+\xi^{N,J,J',\mathrm{fine}}_v.\nonumber
\end{align}

\medskip

We will shortly show that $\xi^{N,J,J',R,\delta}$ is indeed a good approximation to $\varphi^{N}$. Before doing so, we introduce some more notation. We let $J=KL$ and $J'=K'L'$ for integers $K,L,K',L'$ that we will later send to infinity in the order $K',L',K,L$. As in \cite{DRZ17}, we abbreviate
\[\limsup_{(L,K,L',K')\Rightarrow\infty}:=\limsup_{L\to\infty}\limsup_{K\to\infty}\limsup_{L'\to\infty}\limsup_{K'\to\infty} .\]
Recall that by definition $N^*$ is a multiple of $KLK'L'$. Let also $R,T$ be integers. We define a set of good points, (similarly as in \eqref{e:approx_fields3}) by setting
\begin{equation}  \label{eqdef-good}
\begin{split}
	&V_N^{K,L,K',L',R,\delta}\\
	&=V_N^\delta\\
	&\quad\cap\bigcup_{w'\in\W_{N,\left\lfloor\frac{N}{L}\right\rfloor}}V_{\left\lfloor\frac{N}{L}\right\rfloor}^\delta(w')\cap\bigcup_{w'\in\W_{N,\left\lfloor\frac{N}{KL}\right\rfloor}}V_{\left\lfloor\frac{N}{KL}\right\rfloor}^\delta(w')\\
	&\quad\cap\bigcup_{w'\in\W_{N,\left\lfloor\frac{N}{KL}\right\rfloor}}\ \bigcup_{w''\in\W_{\left\lfloor\frac{N}{KL}\right\rfloor,L'}(w')}V_{L'}^\delta(w'')\cap
\bigcup_{w'\in\W_{N,\left\lfloor\frac{N}{KL}\right\rfloor}}\ \bigcup_{w''\in\W_{\left\lfloor\frac{N}{KL}\right\rfloor,K'L'}(w')}V_{K'L'}^\delta(w'')\\
	&\quad\cap\bigcup_{\substack{w'\in\W_{N,\left\lfloor\frac{N}{KL}\right\rfloor}\\\Rtwo_{w'}\le \left\lfloor\frac{N}{KL}\right\rfloor}}V_{K'L'}(w')\cap\bigcup_{w'\in\W_{N,\left\lfloor\frac{N}{KL}\right\rfloor}}\ \bigcup_{\substack{w''\in\W_{\left\lfloor\frac{N}{KL}\right\rfloor,K'L'}(w')\\\Rtwo_{w''}\le K'L'}}V_{K'L'}(w'')\\
		&\quad\cap\bigcup_{\substack{w'\in\W_{N,\left\lfloor\frac{N}{KL}\right\rfloor}\\\Rone_{w'^*}\le R}}V_{K'L'}(w')\nonumber\\
		&\quad\cap\left\{v\in V_N:\Rone_v\le R\right\}
\end{split}
\end{equation}
and
\[V_N^{K,L,K',L',R,T,\delta}=V_N^{K,L,K',L',R,\delta}\cap\left\{v\in V_N(w_N):T_v\le T\right\}.\]
These definitions looks admittedly quite complicated, so let us give some explanations: 
\begin{itemize}
\item In the definition of $V_N^{K,L,K',L',R,\delta}$ we begin with $V_N^\delta$. In the second and third lines we exclude those points which are too close the boundary of the relevant boxes on any of the four scales. The reason is that in order to apply Assumptions \ref{a:logbd}, \ref{a:micro} and \ref{a:macro}, we need to stay away from the boundary of the corresponding box.
\item In the fourth line we exclude those points which lie in a macroscopic or a microscopic box where we know that the random scale $\Rtwo$ is smaller than the sidelength of the corresponding box. On the remaining macroscopic and microscopic boxes we can now use Assumptions \ref{a:micro} and \ref{a:macro}.
\item In the fifth line we exclude those points that lie in a macroscopic box for which $\Rone$ at the corresponding point in the reference configuration $V_J$ is too large. The reason is that on these points we do not control $b^{J,R,\delta,\mathrm{mac}}$.
\item  Finally, in the sixth line we exclude those points where $\Rone$ is too large. For $V_N^{K,L,K',L',R,T,\delta}$ we also exclude those points where $T$ is too large. This is because the assumptions in part A of \ref{as:main} are only useful together with upper bounds on $\Rone$ and $\T$.
\end{itemize}

We refer to  the points that are left after all these steps as good or typical points. We expect that most points are good, and will quantify this now.
For later use we define for $w'\in\Z^d$
\begin{align*}
V_{\left\lfloor\frac{N}{KL}\right\rfloor}^{K,L,K',L',R,\delta}(w')&:=V_{\left\lfloor\frac{N}{KL}\right\rfloor}(w')\cap V_N^{K,L,K',L',R,\delta},\\
V_{\left\lfloor\frac{N}{KL}\right\rfloor}^{K,L,K',L',R,T,\delta}(w')&:=V_{\left\lfloor\frac{N}{KL}\right\rfloor}(w')\cap V_N^{K,L,K',L',R,T,\delta}.
\end{align*}
Now, Assumptions \ref{a:sparseT} and \ref{a:sparseR} imply that most points are good in the sense that we have
\begin{align}
\liminf_{\substack{\delta\to0\\R\to\infty}}\inf_{K,L\in\N}\liminf_{L'\to\infty}\liminf_{K'\to\infty}\liminf_{N\to\infty}\min_{w'\in\W_{N,\left\lfloor\frac{N}{KL}\right\rfloor}}\frac{|V_N^{K,L,K',L',R,\delta}(w_N')|}{(N/KL)^d}=1,\label{e:few_bad_pointsR}\\
\liminf_{T\to\infty}\liminf_{\substack{\delta\to0\\R\to\infty}}\inf_{K,L\in\N}\liminf_{L'\to\infty}\liminf_{K'\to\infty}\liminf_{N\to\infty}\min_{w'\in\W_{N,\left\lfloor\frac{N}{KL}\right\rfloor}}\frac{|V_N^{K,L,K',L',R,T,\delta}(w_N')|}{(N/KL)^d}=1.\label{e:few_bad_pointsT}
\end{align}

\subsection{Convergence in distribution of the maximum}
\label{sec-3.3}

We can now begin with comparing the maxima of $\varphi^{N}$ and $\xi^{N,J,J',R,\delta}$. To do so, we first show that their covariances on the good set $V_N^{K,L,K',L',R,T,\delta}$ are close. Recall the constants $\Gamma$ and $\Gamma'_{R,\delta}$, see \eqref{eq-Gamma} and \eqref{eq-Gammaprime}. The precise statement is as follows.

\begin{lemma}\label{l:approx_variances}
Under Assumptions \ref{a:logbd}, \ref{a:micro} and \ref{a:macro} let $K,L,K',L',R$ be integers, $\delta>0$, and choose $J=KL,J'=K'L'$. Then, for all $R,T,\delta$
 there is  a constant $\Gamma''_{R,\delta}$ and  $\ep'_{N,K,L,K',L',R,T,\delta}$ with
\[\limsup_{(L,K,L',K')\Rightarrow\infty}\limsup_{N\to\infty}\ep'_{N,K,L,K',L',R,T,\delta}=0\] 
with the following property. Whenever $KL$ and $K'L'$ are large enough (depending on $R$), the following estimates hold for all $N$ large enough and $u,v\in V_N^{K,L,K',L',R,T,\delta}$.
\begin{itemize}
	\item[(i)] If $u,v\in V^\delta_{L'}(\hat w'')$ for some $\hat w''\in\W_{\left\lfloor\frac{N}{KL}\right\rfloor,L'}(w')$ and some $w'\in\W_{N,\left\lfloor\frac{N}{KL}\right\rfloor}$, then 
	\[\left|\E\varphi^{N}_v\varphi^{N}_u-\E\xi^{N,KL,K'L',R,\delta}_v\xi^{N,KL,K'L',R,\delta}_u-\Gamma- \Gamma'_{R,\delta}\right|\le\ep'_{N,K,L,K',L',R,\delta}.\]
	\item[(ii)] If $u\in V^\delta_{\left\lfloor\frac{N}{L}\right\rfloor}(w'_u)$,  $v\in V^\delta_{\left\lfloor\frac{N}{L}\right\rfloor}(w'_v)$ for some $\hat w'_u,\hat w'_v\in\W_{N,\left\lfloor\frac{N}{L}\right\rfloor}$ with $\hat w'_u\neq \hat w'_v$, then 
	\[\left|\E\varphi^{N}_v\varphi^{N}_u-\E\xi^{N,KL,K'L',R,\delta}_v\xi^{N,KL,K'L',R,\delta}_u\right|\le\ep'_{N,K,L,K',L',R,\delta}.\]
	\item[(iii)] In any case,
	\[\left|\E\varphi^{N}_v\varphi^{N}_u-\E\xi^{N,KL,K'L',R,\delta}_v\xi^{N,KL,K'L',R,\delta}_u\right|\le\Gamma''_{R,\delta}.\]
\end{itemize}
\end{lemma}
Note that in Lemma \ref{l:approx_variances}
we do not assume Assumption $\ref{a:logupp}$, i.e. we do not use that on the set $V_N^{K,L,K',L',R,T,\delta}$ we have a bound on $\T$. We will use this fact only later in the proof of Lemma \ref{l:comparison_maxima}. On the other hand, it is crucial for the proof that on $V_N^{K,L,K',L',R,T,\delta}$ we have upper bounds on $\Rone$ and $\Rtwo$.

\begin{proof}
This is shown just like \cite[Lemma 4.1]{DRZ17}. The idea for part (i) is that if $u,v\in V^\delta_{L'}(\hat w'')$ then they must lie in the same microscopic subbox $V^\delta_{K'L'}(w'')$. In particular the contributions from $\xi^{N,J,\mathrm{mac}}$, $b^{J,R,\delta,\mathrm{mac}}X^{\mathrm{mac}}$ and $\xi^{N,J,J',\mathrm{mes}}$ are the same for $u,v$, and the nontrivial part is to estimate the covariance of $\xi^{N,\mathrm{mic}}$. We must have $\Rone_v\le R$, $\Rtwo_{\hat w''}\le K'L'$, as otherwise $u,v$ would not be in the set of good points $V_N^{K,L,K',L',R,T,\delta}$. So, as soon as $K'L'\ge R$, we can apply Assumption \ref{a:micro} to the cube $V^\delta_{K'L'}(\hat w'')$, and obtain easily the claimed estimate.

The argument for part (ii) is similar. This time we have to estimate the covariance of $\xi^{N,J,\mathrm{mac}}$. This field was defined using $\varphi^{KL}$. For $KL$ large enough we have $\Rtwo_0\le KL$. Therefore we can apply Assumption \ref{a:macro} to $V_{KL}$ once $KL$ is large enough, and again obtain the conclusion.

Finally, for part (iii) we use that on $V_N^{K,L,K',L',R,T,\delta}$ we can have an upper bound on $\Rone$, so that we can apply \ref{a:logbd} while having a uniform upper bound for $\Rone$. It turns out that we can choose $\Gamma''_{R,\delta}=4\alpha_\delta(R)+\Gamma+\Gamma'_{R,\delta}$.
\end{proof}

Using the results of the previous subsection, Lemma \ref{l:approx_variances} implies that also the maxima of $\varphi^{N}$ and the approximating field are close.

\begin{lemma}\label{l:comparison_maxima}
Under Assumptions \ref{a:logupp}, \ref{a:logbd}, \ref{a:micro}, \ref{a:macro}, \ref{a:sparseT} and \ref{a:sparseR} we have
\begin{align*}
&\limsup_{T\to\infty}\limsup_{\substack{\delta\to0\\R\to\infty}}\limsup_{(L,K,L',K')\Rightarrow\infty}\limsup_{N\to\infty}\\
&\qquad
\dist\Big(\max_{v\in V_N}\varphi^{N}_v,\max_{v\in V_N^{K,L,K',L',R,T,\delta}}\xi^{N,KL,K'L',R,\delta}_v-\sqrt{\frac{d}{2}}(\Gamma+\Gamma'_{R,\delta})\Big)=0.
\end{align*}
\end{lemma}
\begin{proof}
By Assumptions \ref{a:sparseT} and \ref{a:sparseR}, \eqref{e:few_bad_pointsR} and Lemma \ref{l:upper_right_tail_badset} we have an upper bound on the right tail of the maximum of $\varphi^{N}$ on the set $V_N\setminus V_N^{K,L,K',L',R,T,\delta}$. On the other hand, by Lemma \ref{l:upper_left_tail} and Assumption \ref{a:sparseR} we have a lower bound of the maximum on $V_N^{K,L,K',L',R,T,\delta}$. When combined, these two facts together imply that in the limit $R\to\infty$, $\delta\to0$, $T\to\infty$ the probability that the maxima of $\varphi^{N}$ on $V_N$ and $V_N^{K,L,K',L',R,T,\delta}$ are different vanishes. Thus we can assume that the maximum of $\varphi^{N}$ is assumed at a point in $V_N^{K,L,K',L',R,T,\delta}$.% As $\xi^{N,KL,K'L',R,T,\delta,w_N}=0$ outside of $V_N^{K,L,K',L',R,T,\delta}(w_N)$, the same holds true for $\xi^{N,KL,K'L',R,T,\delta,w_N}$ for trivial reasons.

So we can restrict our attention to $V_N^{K,L,K',L',R,T,\delta}$. On this set, we can argue just as in \cite[Lemma 4.2]{DRZ17}, using Lemma \ref{l:approx_fields}, to show that the two maxima are close in distribution.
\end{proof}

By Lemma \ref{l:comparison_maxima} we now only have to investigate the maximum of the random variable
\[\max_{v\in V_N^{K,L,K',L',R,T,\delta}}\xi^{N,KL,K'L',R,\delta}_v-m_N-\sqrt{\frac{d}{2}}(\Gamma+\Gamma'_{R,\delta}).\]
For that purpose, the main ingredient are precise asymptotics for the right tail of the fine field 
$\xi^{N,KL,K'L',\mathrm{fine}}$, where we recall that
\[\xi^{N,J,J',\mathrm{fine}}_v=\xi^{N,J,J',\mathrm{mes}}_v+\xi^{N,J,J',\mathrm{mic}}_v+b^{J',\mathrm{mic}}_{v-w''}X^{\mathrm{mic}}_{w''}.\]
Note that this field does not depend on $R,T,\delta$.

\begin{lemma}\label{l:right_tail}
Under Assumptions \ref{a:logupp}, \ref{a:logbd}, \ref{a:micro}, \ref{a:macro}, \ref{a:sparseT}, \ref{a:sparseR} and \ref{a:lln} there are there are constants $\Gamma^\pm$ such that for any $R,T$ sufficiently large and $\delta$ sufficiently small there is a parameter $\ep''_{R,T,\delta}$ such that 
\[\limsup_{T\to\infty}\limsup_{\substack{\delta\to0\\R\to\infty}}\ep''_{R,T,\delta}=0\]
and such that the following hold.

For any $K',L'$ sufficiently large there are constants $\beta_{K'L'}$ that depend on $K'L'$ only, and satisfy $\Gamma^-\le\beta_{K'L'}\le\Gamma^+$ such that for any $K,L,N$ sufficiently large and every sequence $a\in V_{KL}$, setting $w_{a,KL,N}'=a\left\lfloor \frac{N}{KL}\right\rfloor$, we have  
\begin{equation}\label{e:right_tail}
\begin{split}
&\limsup_{z\to\infty}\limsup_{K'L'\to\infty}\limsup_{N\to\infty}\\
&\quad\Bigg|\frac{\e^{\sqrt{2d}z}}{z}\pp^{N}\bigg(\max_{v\in V_{\left\lfloor\frac{N}{KL}\right\rfloor}^{K,L,K',L',R,T,\delta}(w_{a,KL,N}')}\xi^{N,KL,K'L',\mathrm{fine}}_v\ge m_{\left\lfloor\frac{N}{KL}\right\rfloor}+z\bigg)-\beta_{K'L'}\Bigg|\le\ep
''_{R,T,\delta}.
\end{split}
\end{equation}
\end{lemma}

\begin{proof}
This is the most technical proof of the whole paper. The proof is similar to the proof of \cite[Proposition 4.3]{DRZ17}, although extra care is needed because of the additional parameters $R$ and $T$ and the fact that the distribution of $\xi^{N,J,J',\mathrm{mic}}$ on different microscopic boxes is not the same, but only that Assumption \ref{a:lln} ensures some averaging of these distribution. The proof of \cite[Proposition 4.3]{DRZ17} is outlined in the appendix there, while making extensive reference to \cite{BDZ16} for various intermediate steps. We proceed similary, by giving details where some new ideas are required, and otherwise referring the reader to \cite{BDZ16,DRZ17} for various calculations.

\emph{Step 1: Preliminaries}\\
We begin with the observation that by Assumption \ref{a:logbd} and Lemma \ref{l:approx_variances} (i) we have for $K,L,K',L'$ large enough that
\begin{equation}\label{e:right_tail1}
\left|\E\xi^{N,KL,K'L',\mathrm{fine}}_v\xi^{N,KL,K'L',\mathrm{fine}}_u-\log N+\log_+|u-v|\right|\le \alpha_{\delta}(R)+\Gamma+1
\end{equation}
for all $u,v\in V_N^{K,L,K',L',R,T,\delta}$. Thus the field $\xi^{N,KL,K'L',\mathrm{fine}}$ satisfies a bound like in Assumption \ref{a:logbd}, just with $\alpha_{\delta}(\cdot)+\Gamma+1$ instead of $\alpha_{\delta}(\cdot)$, and so Lemma \ref{l:upper_right_tail} implies an upper bound on the right tail of the form
\begin{equation}\label{e:right_tail2}
\pp^{N}\bigg(\max_{v\in V_{\left\lfloor\frac{N}{KL}\right\rfloor}^{K,L,K',L',R,T,\delta}(w'_{a,KL,N})}\xi^{N,KL,K'L',\mathrm{fine}}_v\ge m_{\left\lfloor\frac{N}{KL}\right\rfloor}+z\bigg)\le C_{\alpha_\delta(R)}z\e^{-\sqrt{2d}z}
\end{equation}
for all $z\ge 1$.
Similarly, \eqref{e:right_tail1} together with Lemma \ref{l:lower_right_tail} and Assumptions \ref{a:sparseT} and \ref{a:sparseR} imply that for $R,T$ sufficiently large and $\delta$ sufficiently small we have a lower bound on the right tail of the form
\begin{equation}\label{e:right_tail3}
\pp^{N}\bigg(\max_{v\in V_{\left\lfloor\frac{N}{KL}\right\rfloor}^{K,L,K',L',R,T,\delta}(w'_{a,KL,N})}\xi^{N,KL,K'L',\mathrm{fine}}_v\ge m_{\left\lfloor\frac{N}{KL}\right\rfloor}+z\bigg)\ge c_{\alpha_\delta(R)}z\e^{-\sqrt{2d}z}
\end{equation}
for all $\sqrt{\log N}-C_{\alpha_\delta(R)}\ge z\ge 1$.

Now \eqref{e:right_tail2} and \eqref{e:right_tail3} applied for some fixed $R,\delta$, clearly imply that if the constant $\beta_{K'L'}$ in \eqref{e:right_tail} exists, then necessarily $\Gamma^-\le\beta_{K'L'}\le\Gamma^+$. So it remains to show the existence of $\beta_{K'L'}$.

\emph{Step 2: Definition of barrier events}\\
By definition, the field $\xi^{N,KL,K'L',\mathrm{fine}}$ on the macroscopic box $V_{\left\lfloor\frac{N}{KL}\right\rfloor}(w'_{a,KL,N})$ is defined only on the microscopic boxes $V_{KL}(w'')$ for $w''\in\W_{\left\lfloor\frac{N}{KL}\right\rfloor,KL'}(w'_{a,KL,N})$, and there are exactly $\left(\frac{N^*}{KLK'L'}\right)^d$ of these microscopic boxes in the macroscopic box $V_{\left\lfloor\frac{N}{KL}\right\rfloor}(w'_{a,KL,N})$.

As in \cite{DRZ17}, we switch notation to make the connection to (M)BRWs more obvious. We let $\bar N:=\frac{N^*}{KL}$, $J':=K'L'$, and let $\bar n:=\log_2\bar N$, $j'=\log_2J'$. Furthermore, we let $\Xi_{\bar N,J'}:=\W_{\left\lfloor\frac{N}{KL}\right\rfloor,K'L'}(w'_N)$. For $w''\in\Xi_{\bar N,J'}$ we write 
\begin{align*}
X^{N,KL,K'L'}_{w''}&:=\xi^{N,KL,K'L',\mathrm{mes}}_{w''}\\
Y^{N,K,L,K',L',R,T,\delta}_{w''}&:=\max_{v\in V_{J'}(w'')\cap V^{K,L,K',L',R,T,\delta}_N}\left(\xi^{N,KL,K'L',\mathrm{mic}}_{v}+b^{K'L',\mathrm{mic}}_vX_{w''}\right)
\end{align*}
where we use the convention that the maximum of the empty set is $-\infty$.
With these definitions we have that
\begin{equation}\label{e:right_tail4}
\max_{v\in V_{\left\lfloor\frac{N}{KL}\right\rfloor}^{K,L,K',L',R,T,\delta}(w'_N)}\xi^{N,KL,K'L',\mathrm{fine}}_v=\max_{w''\in \Xi_{\bar N,J'}}\left(X^{N,KL,K'L'}_{w''}+Y^{N,K,L,K',L',R,T,\delta}_{w''}\right),
\end{equation}
and so it suffices to study the behaviour of the right-hand side.

%\begin{equation}\label{e:right_tail4}
%\begin{split}
%&\max_{v\in V_{J'}(w'')}\xi^{N,KL,K'L',\delta,w_N,\mathrm{fine}}_v\\
%&\quad=\begin{cases}X^{N,KL,K'L'}_{w''}+Y^{N,K,L,K',L',R,T,\delta}_{w''}&V_{J'}(w'')\cap V^{K,L,K',L',R,T,\delta}_N(w_N)\neq\varnothing\\0&\text{else}\end{cases}
%\end{split}
%\end{equation}
%We claim that microscopic boxes where $V_{J'}(w'')\cap V^{K,L,K',L',R,T,\delta}_N(w_N)=\varnothing$ have an negligible effect on the maximum of $\xi^{N,KL,K'L',\delta,w_N,\mathrm{fine}}$. Indeed, by an upper bound on the right tail of the MBRW (as in the proof of Lemma \ref{l:upper_right_tail_sparse}) we have
%\[\pp^{N,w_N}\left(\max_{w''\in \Xi_{\bar N,J'}}X^{N,KL,K'L'}_{w''}\ge m_{\frac{\bar N}{J'}}+z\right)\le Cz\e^{-\sqrt{2d}z}\]
%for $z\ge1$. As $\bar N\to\infty$ and then $J'\to\infty$, $m_{\frac{\bar N}{J'}}$ grows slower than $m_{\bar N}$. So, as claimed, we have
%\begin{equation}\label{e:right_tail5}
%\begin{split}
%&\limsup_{J'\to\infty}\limsup_{\bar N\to\infty}\pp^{N,w_N}\Bigg(\max_{v\in V_{\left\lfloor\frac{N}{KL}\right\rfloor}(w'_N)}\xi^{N,KL,K'L',R,T,\delta,w_N,\mathrm{fine}}_v\ge m_{\left\lfloor\frac{N}{KL}\right\rfloor}+z,\\
%&\quad \max_{v\in V_{\left\lfloor\frac{N}{KL}\right\rfloor}(w'_N)}\xi^{N,KL,K'L',R,T,\delta,w_N,\mathrm{fine}}_v\neq \max_{w''\in \Xi_{\bar N,J'}}X^{N,KL,K'L'}_{w''}+Y^{N,K,L,K',L',R,T,\delta}_{w''}\Bigg)=0
%\end{split}
%\end{equation}
%for each $z\ge 1$. Thus, to establish \eqref{e:right_tail5}, it suffices to study the right tail of the maximum of $X^{N,KL,K'L'}_{w''}+Y^{N,K,L,K',L',R,T,\delta}_{w''}$.

Recall that $\xi^{N,KL,KL',\mathrm{mes}}_{w''}$ was defined as the value of a certain MBRW together with a small correction factor, that is
\[X^{N,KL,K'L'}_{w''}=s_{\bar N,J'}\tilde\theta^{2^{\lceil \bar n-j'\rceil}}_{w''}\]
where 
\[s_{\bar N,J'}=\frac{\log 2^{\bar n-j'}}{\log 2^{\lceil\bar n-j'\rceil}}=\frac{\bar n-j'}{\lceil\bar n-j'\rceil}.\]
For each $w''$, $X^{N,KL,K'L'}_{w''}$ is distributed as a centered Gaussian with variance $\lceil\bar n-j'\rceil\log 2$. As in the proof of Lemma \ref{l:upper_right_tail_sparse}, we can interpret $s_{\bar N,J'}X^{N,KL,K'L'}_{w''}$ as the value at time $\bar n-j'$ of a $2d$-ary branching Brownian motion that branches at times $0,s_{\bar N,J'},2s_{\bar N,J'},\ldots,\bar n-j' -s_{\bar N,J'}$. For that branching Brownian motion, we can define the barrier events

\begin{align*}
\Ev^{N,K,L,K',L',R,T,\delta}_{E,w''}(z)&=\Bigg\{X^{N,KL,K'L'}_{w''}(t)\le z+\frac{t m_{\bar N}}{\bar n}\ \forall0\le t\le \bar n-j',\\
&\qquad \qquad\qquad\qquad  X^{N,KL,K'L'}_{w''}(\bar n-j')+Y^{N,K,L,K',L',R,T,\delta}_{w''}\ge m_{\bar N}+z\Bigg\},\\
\Ev^{N,K,L,K',L',R,T,\delta}_{F,w''}(z)&=\Bigg\{X^{N,KL,K'L'}_{w''}(\bar n-j')+Y^{N,K,L,K',L',R,T,\delta}_{w''}\ge m_{\bar N}+z,\\
&\!\!\!\!\!\! \!\!\!\!\!\!\!\! \!\!X^{N,KL,K'L'}_{w''}(t)\le z+\frac{t m_{\bar N}}{\bar n}+10\log_+(t\vee(\bar n-j'-t))+z^{1/20}\ \forall0\le t\le \bar n-j'\Bigg\},\\
%,\\
\Ev^{N,K,L,K',L',R,T,\delta}_G(z)&=\\
&\!\!\!\!\!\! \!\!\!\!\!\!\!\! \!\!\!\!\!\!\!\! \!\!
 \bigcup_{w''\in\Xi_{\bar N,J'}}\bigcup_{0\le t\le \bar n-j'}\left\{X^{N,KL,K'L'}_{w''}(t)\ge z+\frac{t m_{\bar N}}{\bar n}+10\log_+(t\vee(\bar n-j'-t))+z^{1/20}\right\},
\end{align*}
and the random variables
\begin{align*}
\Lambda^{N,K,L,K',L',R,T,\delta}_{E}(z)&=\sum_{w''\in\Xi_{\bar N,J'}}\I_{\Ev^{N,K,L,K',L',R,T,\delta}_{E,w''}(z)},\\
\Lambda^{N,K,L,K',L',R,T,\delta}_{F}(z)&=\sum_{w''\in\Xi_{\bar N,J'}}\I_{\Ev^{N,K,L,K',L',R,T,\delta}_{F,w''}(z)}.
\end{align*}

\emph{Step 3: First and second moment argument}\\
Our goal in this step is to show that the probability that the maximum of $\xi^{N,KL,K'L',R,T,\delta,w_N,\mathrm{fine}}_v$ is at least $m_{\bar N}+z$ is asymptotically the same as the expectation of $\Lambda^{N,K,L,K',L',R,T,\delta}_E(z)$. To do so, we need to compare the various barrier events, and show that $\Lambda^{N,K,L,K',L',R,T,\delta}_E(z)$ is concentrated around its expectation. For the latter we will use a second moment argument.

We begin by observing that by the same argument as in \cite{BDZ16} and \cite{DRZ17} the event $\Ev^{N,K,L,K',L',R,T,\delta}_G(z)$ is negligible in the sense that
\begin{equation}\label{e:right_tail6}
\pp^{N}\left(\Ev^{N,K,L,K',L',R,T,\delta}_G(z)\right)\le C\e^{-\sqrt{2d}z}
\end{equation}
for some absolute constant $z$. 

Next we study $\Lambda^{N,K,L,K',L',R,T,\delta}_E(z)$ and $\Lambda^{N,K,L,K',L',R,T,\delta}_F(z)$. We have the trivial estimate $\Lambda^{N,K,L,K',L',R,T,\delta}_E(z)\le\Lambda^{N,K,L,K',L',R,T,\delta}_F(z)$, and we claim that this is asymptotically sharp in the sense that\footnote{Note that the corresponding result in \cite{DRZ17}, Equation (59), contains several typos: The $\limsup$s there should be $\liminf$s, and the expectation in the numerator and denominator is missing.}
\begin{equation}\label{e:right_tail7}
\lim_{z\to\infty}\liminf_{J'\to\infty}\liminf_{\bar N\to\infty}\frac{\E\Lambda^{N,K,L,K',L',R,T,\delta}_E(z)}{\E\Lambda^{N,K,L,K',L',R,T,\delta}_F(z)}=1.
\end{equation}
For this, the proof from \cite{DRZ17} does not carry over directly, as the $\Ev^{N,K,L,K',L',R,T,\delta}_{E,w''}(z)$ for different $w''$ do not have the same probability (and the same holds for the $\Ev^{N,K,L,K',L',R,T,\delta}_{F,w''}(z)$). But this is just a minor problem: For each fixed $w''$ it could be that $V_{J'}(w'')\cap V^{K,L,K',L',R,T,\delta}=\varnothing$, and in that case $Y^{N,K,L,K',L',R,T,\delta}_{w''}=-\infty$, and so both $\Ev^{N,K,L,K',L',R,T,\delta}_{E,w''}(z)$ and $\Ev^{N,K,L,K',L',R,T,\delta}_{F,w''}(z)$ have probability 0.
In the generic case that $V_{J'}(w'')\cap V^{K,L,K',L',R,T,\delta}_N(w_N)\neq\varnothing$, we have that both events $\Ev^{N,K,L,K',L',R,T,\delta}_{E,w''}(z)$ and $\Ev^{N,K,L,K',L',R,T,\delta}_{F,w''}(z)$ have positive probability. We can compare them using the same argument as in \cite[Lemma 4.10]{BDZ16}, which uses as its only information on $Y^{N,K,L,K',L',R,T,\delta}_{w''}$ the Gaussian bound on the right tail that is implied by Lemma \ref{l:upper_right_tail}. In this way we find
\begin{align*}
&\limsup_{z\to\infty}\limsup_{J'\to\infty}\limsup_{\bar N\to\infty}\max_{\substack{w''\in\Xi_{\bar N,J'}\\V_{J'}(w'')\cap V^{K,L,K',L',R,T,\delta}_N\neq\varnothing}}\\
&\qquad\qquad\qquad\frac{\pp^{N}(\Ev^{N,K,L,K',L',R,T,\delta}_{F,w''}(z))-\pp^{N}(\Ev^{N,K,L,K',L',R,T,\delta}_{E,w''}(z))}{\pp^{N}(\Ev^{N,K,L,K',L',R,T,\delta}_{E,w''}(z))}=0.\end{align*}
Summing this for all $w''$, we indeed obtain \eqref{e:right_tail7}.

Next, we need an exponential lower bound on $\E\Lambda^{N,K,L,K',L',R,T,\delta}_E(z)$. Such a lower bound follows from \eqref{e:right_tail3} together with \eqref{e:right_tail6} and \eqref{e:right_tail7}, and we find that for $R,T$ sufficiently large and $\delta$ sufficiently small
\begin{equation}\label{e:right_tail9}
\E\Lambda^{N,K,L,K',L',R,T,\delta}_E(z)\ge c'_{\alpha_\delta(R)}z\e^{-\sqrt{2d}z}
\end{equation}
for all $\sqrt{\log N}-C_{\alpha_\delta(R)}\ge z\ge 1$.

The bounds derived so far can now be combined to give upper bounds on the tail probability in question. Indeed from \eqref{e:right_tail6}, \eqref{e:right_tail7}, \eqref{e:right_tail9} and Markov's inequality we obtain for $R,T$ sufficiently large and $\delta$ sufficiently small
\begin{equation}\label{e:right_tail10}
\limsup_{z\to\infty}\limsup_{J'\to\infty}\limsup_{\bar N\to\infty}\frac{\pp^{N}\left(\max_{w''\in \Xi_{\bar N,J'}}\left(X^{N,KL,K'L'}_{w''}+Y^{N,K,L,K',L',R,T,\delta}_{w''}\right)\ge m_{\bar N}+z\right)}{\E\Lambda^{N,K,L,K',L',R,T,\delta}_E(z)}\le1.
\end{equation}

For the corresponding lower bound, we need to bound the second moment of $\Lambda^{N,K,L,K',L',R,T,\delta}_E(z)$. This is a very lengthy calculation, but fortunately this calculation also uses as its only information on $Y^{N,K,L,K',L',R,T,\delta}_{w''}$ the Gaussian bound on the right tail. This means that the calculation in \cite[Lemma 4.11]{BDZ16} carries over directly to our setting, and we find
\begin{equation}\label{e:right_tail8}
\lim_{z\to\infty}\limsup_{J'\to\infty}\limsup_{\bar N\to\infty}\frac{\E\Lambda^{N,K,L,K',L',R,T,\delta}_E(z)^2}{\E\Lambda^{N,K,L,K',L',R,T,\delta}_E(z)}=1
\end{equation}
where we note that the right-hand side is trivially bounded below by 1.

Now \eqref{e:right_tail8} and the Paley-Zygmund inequality imply that
\begin{equation}\label{e:right_tail11}
\liminf_{z\to\infty}\liminf_{J'\to\infty}\liminf_{\bar N\to\infty}\frac{\pp^{N,w_N}\left(\max_{w''\in \Xi_{\bar N,J'}}\left(X^{N,KL,K'L'}_{w''}+Y^{N,K,L,K',L',R,T,\delta}_{w''}\right)\ge m_{\bar N}+z\right)}{\E\Lambda^{N,K,L,K',L',R,T,\delta}_E(z)}\ge1.
\end{equation}

Now \eqref{e:right_tail10} and \eqref{e:right_tail11} are matching upper and lower bounds. Taking them together and recalling \eqref{e:right_tail4}, we obtain the desired asymptotics on the right tail, namely
\begin{align}\label{e:right_tail12}
&\limsup_{z\to\infty}\limsup_{J'\to\infty}\limsup_{\bar N\to\infty}\Bigg|\frac{\pp^{N,w_N}\bigg(\max_{v\in V_{\left\lfloor\frac{N}{KL}\right\rfloor}^{K,L,K',L',R,T,\delta}(w'_N)}\xi^{N,KL,K'L',\delta,w_N,\mathrm{fine}}_v\ge m_{\left\lfloor\frac{N}{KL}\right\rfloor}+z\bigg)}{\E\Lambda^{N,K,L,K',L',R,T,\delta}_E(z)}-1\Bigg|\nonumber\\
&\qquad\qquad\qquad\qquad\qquad\qquad\qquad\qquad=0
\end{align}
for any $R,T$ large and $\delta$ small.

\emph{Step 4: Asymptotics for the first moment}\\
After the hard work of the previous step, it remains to derive the asymptotics, as $z\to\infty$,  of $\E\Lambda^{N,K,L,K',L',R,T,\delta}_E(z)$.  This step is somewhat different than the corresponding argument in \cite{DRZ17}, as the 
$Y^{N,K,L,K',L',R,T,\delta}_{w''}$ for different $w''$ are not  identically distributed. Instead we will employ Assumption \ref{a:lln} to control the average. In order to so, we will first need to get rid of the restriction to points in $V^{K,L,K',L',R,T,\delta}_N$ in the definition of $Y^{N,K,L,K',L',R,T,\delta}_{w''}$.

Denote by $\chi_{N,KL,K'L',w''}$ the density of 
\[\pp^{N}\left(X^{N,KL,K'L'}_{w''}(t)\le z+\frac{t m_{\bar N}}{\bar n}\ \forall0\le t\le \bar n-j',X^{N,KL,K'L'}_{w''}(\bar n-j')-\frac{(\bar n-j') m_{\bar N}}{\bar n}\in\cdot\right)\]
with respect to one-dimensional Lebesgue measure. Then we have
\begin{equation}\label{e:right_tail13}
\begin{split}
&\E\Lambda^{N,K,L,K',L',R,T,\delta}_E(z)=\sum_{w''\in\Xi_{\bar N,J'}}\pp^{N}\left(E^{N,K,L,K',L',R,T,\delta}_{w''}(z)\right)\\
&=\sum_{w''\in\Xi_{\bar N,J'}}\int_{-\infty}^z\chi_{N,KL,K'L',w''}(y)\pp^{N}\left( Y^{N,K,L,K',L',R,T,\delta}_{w''}\ge \frac{j' m_{\bar N}}{\bar n}+z-y\right)\ud y.
\end{split}
\end{equation}
The main contribution to the integrals on the right-hand side should come from $y$ which are of the order of $j'^{1/2}$. Indeed, for an interval $I\subset\R$ we define
\[\lambda^{N,K,L,K',L',R,T,\delta}_{I,w''}(z):=
\int_{I}\chi_{N,KL,K'L',w''}(y)\pp^{N}\left( Y^{N,K,L,K',L',R,T,\delta}_{w''}\ge \frac{j' m_{\bar N}}{\bar n}+z-y\right)\ud y.\]
Now consider any sequence $(x_{w'',\bar N})_{\bar N\in\N,w''\in \Xi_{\bar N,J'}}$ with $|x_{w'',\bar N}|\le j'^{1/5}$, and let $I_{j'}=[-j',-j'^{2/5}]$. Then we claim as in \cite{DRZ17} that
\begin{equation}\label{e:right_tail14}
\liminf_{z\to\infty}\liminf_{J'\to\infty}\liminf_{\bar N\to\infty}\frac{\sum_{w''\in\Xi_{\bar N,J'}}\lambda^{N,K,L,K',L',R,T,\delta}_{x_{w'',\bar N}+I_{j'},w''}(z)}{\E\Lambda^{N,K,L,K',L',R,T,\delta}_E(z)}=1.
\end{equation}
Note that the right-hand side here is trivially bounded above by 1. The actual estimate \eqref{e:right_tail14} follows just as in \cite{DRZ17}. Namely using the Gaussian upper bound on the right tail of the random variables $Y^{N,K,L,K',L',R,T,\delta}_{w''}$ one bounds the integrand on $(-\infty,0]\setminus x_{w'',\bar N}+I_{j'}$ by terms that are negligigle in comparison to the lower bound in \eqref{e:right_tail9}.

Next, we want to argue that the numerator on the left-hand side of \eqref{e:right_tail14} depends only weakly on $R,T,\delta$. To that end, define
\[\bar Y^{N,KL,K'L'}_{w''}:=\max_{v\in V_{J'}(w'')}\left(\xi^{N,KL,K'L',w_N,\mathrm{mic}}_{v}+b^{K'L',\mathrm{mic}}_vX_{w''}\right)\]
and, for an interval $I\subset\R$,
\[\bar\lambda^{N,KL,K'L'}_{I,w''}(z):=
\int_{I}\chi_{N,KL,K'L',w''}(y)\pp^{N}\left( \bar Y^{N,KL,K'L'}_{w''}\ge \frac{j' m_{\bar N}}{\bar n}+z-y\right)\ud y.\]
Note that $\bar\lambda^{N,KL,K'L'}_{I,w''}(z)$ is independent of $R,T,\delta$.

Our claim then is that for $(x_{w'',\bar N})$ as in \eqref{e:right_tail14} we have
\begin{equation}\label{e:right_tail15}
\liminf_{z\to\infty}\liminf_{L'\to\infty}\liminf_{K'\to\infty}\liminf_{\bar N\to\infty}\frac{\sum_{w''\in\Xi_{\bar N,J'}}\lambda^{N,K,L,K',L',R,T,\delta}_{I_{j'},w''}(z)}{\sum_{w''\in\Xi_{\bar N,J'}}\bar\lambda^{N,KL,K'L'}_{I_{j'},w''}(z)}\ge 1-\ep''_{R,T,\delta}
\end{equation}
where $\ep''_{R,T,\delta}$ is a parameter that depends on $R,T,\delta$ only and satisfies
\[\limsup_{T\to\infty}\limsup_{\substack{\delta\to0\\R\to\infty}}\ep''_{R,T,\delta}=0.\]
Again we note that the fraction in \eqref{e:right_tail15} is trivially bounded above by 1. To see \eqref{e:right_tail15} itself, it suffices to show that
\begin{equation}\label{e:right_tail16}\liminf_{z\to\infty}\liminf_{J'\to\infty}\liminf_{\bar N\to\infty}\inf_{y\in I_{\ell}}\frac{\sum_{w''\in\Xi_{\bar N,J'}}\pp^{N}\Big( \bar Y^{N,KL,K'L'}_{w''}\ge \frac{j' m_{\bar N}}{\bar n}+z-y\Big)}{\sum_{w''\in\Xi_{\bar N,J'}}\pp^{N}\Big( Y^{N,K,L,K',L',R,T,\delta}_{w''}\ge \frac{j' m_{\bar N}}{\bar n}+z-y\Big)}\ge 1-\ep''_{R,T,\delta}
\end{equation}
as then \eqref{e:right_tail15} follows by integrating over $y$.
To that end, let
\[z':=\frac{j' m_{\bar N}}{\bar n}+z-y-m_{J'}\]
and note that by Lemma \ref{l:upper_right_tail_badset_micro},
\begin{equation}\label{e:right_tail17}
\begin{split}
&\sum_{w''\in\Xi_{\bar N,J'}}\bigg(\pp^{N}\Big( \bar Y^{N,KL,K'L'}_{w''}\ge \frac{j' m_{\bar N}}{\bar n}+z-y\Big)-
\pp^{N}\Big( Y^{N,K,L,K',L',R,T,\delta}_{w''}\ge \frac{j' m_{\bar N}}{\bar n}+z-y\Big)\bigg)\\
&\le\sum_{w''\in\Xi_{\bar N,J'}}\quad\pp^{N}\Big(\max_{v\in V_{J'}(w'')\setminus V^{K,L,K',L',R,T,\delta}_N}\big(\xi^{N,KL,K'L',\mathrm{mic}}_{v}+b^{K'L',\mathrm{mic}}_vX_{w''}\big)\ge m_{J'}+z'\Big)\\
&\le C\gamma^{K,L,K',L',R,\delta}_{w'_{a,KL,N}}\left(\frac{N}{JJ'}\right)^dz'\e^{-\sqrt{2d}z'}
\end{split}
\end{equation}
where 
\[\begin{split}
&\gamma^{K,L,K',L',R,\delta}_{w'_{a,KL,N}}:=\\
&\left(\e^{d(T+\Gamma)}\frac{\Big|V_{\left\lfloor\frac{N}{KL}\right\rfloor}^{K,L,K',L',R,\delta}(w'_{a,KL,N})\Big|}{N^d}\Bigg(1+T+\log\bigg(\frac{N^d}{\Big|V_{\left\lfloor\frac{N}{KL}\right\rfloor}^{K,L,K',L',R,\delta}(w'_N)\Big|}\bigg)\bigg)^{19/8}+\e^{-c(T-\Gamma)}\right).
\end{split}\]
By \eqref{e:few_bad_pointsR}, we have
\[\limsup_{T\to\infty}\limsup_{\substack{\delta\to0\\R\to\infty}}\sup_{K,L\in\N}\limsup_{L'\to\infty}\limsup_{K'\to\infty}\limsup_{N\to\infty}\max_{a\in V_{KL}(0)}\gamma^{K,L,K',L',R,\delta}_{w'_{a,KL,N}}=0.\]
These estimates show that the difference between denominator and numerator in \eqref{e:right_tail15} is small. On the other hand we have a lower bound for the numerator. Indeed we can pick some  $R_0$ large such that at least $\frac34$ of the points in $\bigcup_{w''\in\Xi_{\bar N,J'}}V_{J'}(w'')$ satisfy $\Rone_v\ge R_0$. Then by the pigeonhole principle we can apply the lower bound on the right tail, Lemma \ref{l:lower_right_tail}, with $R=R_0+\Gamma$ on at least half of the cubes $V_{J'}(w'')$, and summing over these we find
\[\sum_{w''\in\Xi_{\bar N,J'}}\pp^{N}\left( \bar Y^{N,KL,K'L'}_{w''}\ge \frac{j' m_{\bar N}}{\bar n}+z-y\right)\ge cz'\left(\frac{N}{JJ'}\right)^dz'\e^{-\sqrt{2d}z'}.\]
Together with \eqref{e:right_tail17}, we find
\[\frac{\sum_{w''\in\Xi_{\bar N,J'}}\lambda^{N,K,L,K',L',R,T,\delta}_{I_{j'},w''}(z)}{\sum_{w''\in\Xi_{\bar N,J'}}\bar\lambda^{N,KL,K'L'}_{I_{j'},w''}(z)}\ge 1-C\gamma^{K,L,K',L',R,\delta}_{w'_{a,KL,N}}\]
and from this estimate we immediately obtain \eqref{e:right_tail16} and then also \eqref{e:right_tail15}.

Combining now \eqref{e:right_tail15} with \eqref{e:right_tail14}, we have shown that
\begin{equation}\label{e:right_tail19}
\limsup_{z\to\infty}\limsup_{L'\to\infty}\limsup_{K'\to\infty}\limsup_{\bar N\to\infty}\left|\frac{\E\Lambda^{N,K,L,K',L',R,T,\delta}_E(z)}{\sum_{w''\in\Xi_{\bar N,J'}}\bar\lambda^{N,KL,K'L'}_{I_{j'},w''}(z)}-1\right|\le\ep''_{R,T,\delta}.
\end{equation}
So to show the existence of $\beta_{K'L'}$, it suffices to show that $\sum_{w''\in\Xi_{\bar N,J'}}\bar\lambda^{N,KL,K'L'}_{I_{j'},w''}(z)$ has a limit that only depends on $K'L'$.
Recall that
\[\sum_{w''\in\Xi_{\bar N,J'}}\bar\lambda^{N,KL,K'L'}_{I_{j'},w''}(z)=
\sum_{w''\in\Xi_{\bar N,J'}}\int_{I_{j'}}\chi_{N,KL,K'L',w''}(y)\pp^{N,w_N}\left( \bar Y^{N,KL,K'L'}_{w''}\ge \frac{j' m_{\bar N}}{\bar n}+z-y\right)\ud y.\]
Now we have that
\[\frac{j' m_{\bar N}}{\bar n}=\sqrt{2d}\log 2j'+O\left(j'\frac{\log\bar n}{\bar n}\right)\]
and therefore, by shifting the domain of integration by $O\left(j'\frac{\log\bar n}{\bar n}\right)$ and using the asymptotics for $\chi_{N,KL,K'L',w''}(y)$ from \cite[Equation (67)]{DRZ17} to estimate the change in that term, we see that
\begin{align*}
&\sum_{w''\in\Xi_{\bar N,J'}}\bar\lambda^{N,KL,K'L'}_{I_{j'},w''}(z)\\
&=\left(1+O\left(j'^3\frac{\log\bar n}{\bar n}\right)\right)\int_{I_{j'}}\chi_{N,KL,K'L',w''}(y)\sum_{w''\in\Xi_{\bar N,J'}}\pp^{N}\left( \bar Y^{N,KL,K'L'}_{w''}\ge \sqrt{2d}\log 2l+z-y\right)\ud y\\
&=\left(1+O\left(j'^3\frac{\log\bar n}{\bar n}\right)\right)\int_{I_{j'}}\chi_{N,KL,K'L',w''}(y)\\
& \qquad\qquad \qquad\qquad \qquad \sum_{w''\in\Xi_{\bar N,J'}}\pp^{J',w''}\left(\max_{v\in V_{J'}(w'')}\varphi^{J',w''}+b_{J',v}X_{w''}\ge \sqrt{2d}\log 2l+z-y\right)\ud y.
\end{align*}
The right-hand side is now finally in a form to which we can apply our ergodicity Assumption \ref{a:lln}. By that assumption and dominated convergence we know for every fixed $y\le0$ that
\begin{align*}
&\lim_{N\to\infty}\left(\frac{KLK'L'}{N}\right)^d\sum_{w''\in\Xi_{\bar N,J'}}\pp^{N}\left(\max_{v\in V_{J'}(w'')}\varphi+b^{K'L',\mathrm{mic}}_v x\ge \sqrt{2d}\log 2l+z-y\right)\\
&\quad=\lim_{N\to\infty}\int_{\R}\frac{\e^{-x^2/2}}{\sqrt{2\pi}}\left(\frac{KK'LL''}{N}\right)^d\sum_{w''\in\Xi_{\bar N,J'}}\pp^{N}\left(\max_{v\in V_{J'}(w'')}\varphi+b^{K'L',\mathrm{mic}}_v x \ge \sqrt{2d}\log 2l+z-y\right)\ud x\\
&\quad=\int_{\R}\int\frac{\e^{-x^2/2}}{\sqrt{2\pi}}\pp^{N}\left(\max_{v\in V_{J'}(w'')}\varphi+b^{K'L',\mathrm{mic}}_v x\ge \sqrt{2d}\log 2l+z-y\right)\,\nu_{J'}(\mathrm{d}\pp)\ud x.
\end{align*}
Using dominated convergence a second time and also the asymptotics from \cite[Equation (67)]{DRZ17}, we obtain
\begin{equation}\label{e:right_tail18}
\begin{split}
&\lim_{N\to\infty}\sum_{w''\in\Xi_{\bar N,J'}}\bar\lambda^{N,KL,K'L'}_{I_{j'},w''}(z)\\
&\quad=\int_{I_{j'}}\frac{z(z-y)}{\sqrt{2\pi\log2}\e^{\sqrt{2d}y}}\int_{\R}\int\frac{\e^{-x^2/2}}{\sqrt{2\pi}}\\
&\qquad\qquad\qquad\qquad\qquad\qquad \qquad \pp^{N}\left(\max_{v\in V_{J'}(w'')}\varphi+b^{K'L',\mathrm{mic}}_v x\ge \sqrt{2d}\log 2l+z-y\right)\,\nu_{J'}(\mathrm{d}\pp)\ud x\ud y\\
&\quad=\int_{I_{j'}}\frac{z(z-y)}{\sqrt{2\pi\log2}\e^{\sqrt{2d}y}}\int\pp\left(\max_{v\in V_{J'}(w'')}\varphi+b_{J',v}X\ge \sqrt{2d}\log 2l+z-y\right)\nu_{J'}(\mathrm{d}\pp)\ud y.
\end{split}
\end{equation}
Denoting the right-hand side of \eqref{e:right_tail18} by $\Lambda^{*,J'}(z)$ and recalling \eqref{e:right_tail19} and \eqref{e:right_tail12}, we have thus shown that there is $\Lambda^{*,K'L'}(z)$ depending on $K'L'$ and $z$ only, such that
\begin{equation}\label{e:right_tail20}
\begin{split}
&\limsup_{z\to\infty}\limsup_{L'\to\infty}\limsup_{K'\to\infty}\limsup_{\bar N\to\infty}\\
&\qquad\qquad \Bigg|\frac{\pp^{N}\bigg(\max_{v\in V_{\left\lfloor\frac{N}{KL}\right\rfloor}^{K,L,K',L',R,T,\delta}(w'_{a,KL,N})}\xi^{N,KL,K'L',\mathrm{fine}}_v\ge m_{\left\lfloor\frac{N}{KL}\right\rfloor}+z\bigg)}{\Lambda^{*,K'L'}(z)}-1\Bigg|\le\ep''_{R,T,\delta}.
\end{split}
\end{equation}
Now the only remaining task is to asymptotics of $\Lambda^{*,K'L'}(z)$ for large $z$. This can be done just as in \cite{DRZ17}, using \eqref{e:right_tail14} as well as the asymptotics from \cite[Equation (67)]{DRZ17}. We omit the details.
\end{proof}

Using Lemma \ref{l:right_tail}, we can complete the proof of Theorem \ref{t:mainthm} just as in \cite{DRZ17}. Let us mention the most important steps.

As a first step, we note that with little additional effort we can deduce from \ref{l:right_tail} a seemingly stronger version.
\begin{lemma}\label{l:right_tail_unif}
In the setting of Lemma \ref{l:right_tail} there is a function $\gamma\colon\N^3\times(0,1)\to\R$ such that for all $R,T,\delta$ we have 
\[\lim_{J'\to\infty}\gamma(J',R,T,\delta)=0\] and such that
\begin{equation}\label{e:right_tail_unif}
\begin{split}
&\limsup_{z'\to\infty}\limsup_{K'L'\to\infty}\limsup_{N\to\infty}\sup_{z'\le z\le\gamma(K'L',R,T,\delta)}\\
&\quad\left|\frac{\e^{\sqrt{2d}z}}{z}\pp^{N}\bigg(\max_{v\in V_{\left\lfloor\frac{N}{KL}\right\rfloor}^{K,L,K',L',R,T,\delta}(w'_{a,KL,N})}\xi^{N,KL,K'L',\mathrm{fine}}_v\ge m_{\left\lfloor\frac{N}{KL}\right\rfloor}+z\bigg)-\beta_{K'L'}\right|\le2\ep
''_{R,T,\delta}.
\end{split}
\end{equation}
\end{lemma}
The corresponding result in \cite{DRZ17} is stated there without proof, so let us quickly give an outline of the proof here.
\begin{proof}
Abbreviate the term in absolute values in \eqref{e:right_tail_unif} by $F_{N,KL,K'L',R,T,\delta}(z)$.
The probability in \eqref{e:right_tail_unif} is a non-increasing function of $z$, while the prefactor ${\e^{\sqrt{2d}z}}/{z}$ grows only exponentially fast. These facts combined imply that we can control $F$ on short intervals by its values at the endpoints of that interval. This means we can discretize the problem in $z$. More precisely, there is $\eta>0$ depending only on $\ep
''_{R,T,\delta}$ and $\Gamma^\pm$ such that \eqref{e:right_tail_unif} follows from
\begin{equation}\label{e:right_tail_unif1}
\limsup_{z'\to\infty}\limsup_{K'L'\to\infty}\limsup_{N\to\infty}\sup_{\substack{z'\le z-\eta\le\gamma(K'L',R,T,\delta)+\eta\\z\in\eta\Z}}\left|F_{N,KL,K'L',R,T,\delta}(z)\right|\le\frac32\ep
''_{R,T,\delta}.
\end{equation}
To see \eqref{e:right_tail_unif1}, note that from Lemma \ref{l:right_tail} we know that for each fixed sufficiently large $z$ we have
\[
\limsup_{K'L'\to\infty}\limsup_{N\to\infty}\quad\left|F_{N,KL,K'L',R,T,\delta}(z)\right|\le\frac54\ep
''_{R,T,\delta}.
\]
So for each sufficiently large $z$ there is $J'_z<\infty$ such that for any $K'L'\ge J'_z$ we have
\[\limsup_{N\to\infty}\quad\left|F_{N,KL,K'L',R,T,\delta}(z)\right|\le\frac32\ep
''_{R,T,\delta}.\]
Now we can define
\[\gamma(K'L',R,T,\delta)=\sup\left\{z\in\eta\Z\colon K'L'\ge J'_z\right\}.\]
With this choice of $\gamma$ we trivially have 
\eqref{e:right_tail_unif1}, and moreover it is clear that $\lim_{J'\to\infty}\gamma(J',R,T,\delta)=\infty$.
\end{proof}

We construct a field that is independent of $N$ as follows: Our starting point is Lemma \ref{l:right_tail_unif}, and we use the objects from that Lemma (and Lemma \ref{l:right_tail}).

Let $\hat\Xi_{KL}=\frac{1}{KL}\Z^d\cap[0,1]^d$. Let $\tilde\gamma(J):=\log\log\log(J)$.\footnote{In \cite{DRZ17} the choice $\tilde\gamma=\gamma$ was made. However the functions do play different roles, and so we prefer to keep them separate here.} For each $x\in\hat\Xi_{KL}$, consider a Bernoulli random variable $\rho_{KL,K'L',R,T,\delta,x}$ with
\[\pp(\rho_{KL,K'L',R,T,\delta,x}=1)=\beta_{K'L'}\tilde\gamma(KL)\e^{-\sqrt{2d}\tilde\gamma(KL)}\]
and a random variable $Y_{KL,K'L',R,T,\delta,x}$ such that
\[\pp(Y_{KL,K'L',R,T,\delta,x}\ge z)=\frac{\tilde\gamma(KL)}{\tilde\gamma(KL)+z}\e^{-\sqrt{2d}z}.\]
The random field $Z_{KL,K'L',R,T,\delta}$ will be defined as a downscaled version of $\xi^{N,J,J',R,\delta,\mathrm{coarse}}$ (which by the downscaling becomes independent of $N$). To be precise, consider $\varphi^{KL}$ distributed according to $\pp^{KL}$, and define the random field $Z_{KL,K'L',R,T,\delta}$ by 
\[Z_{KL,K'L',R,T,\delta,x}=\varphi^{KL}_{KLx}+\hat b^{KL,R,\delta,\mathrm{mac}}_{KLx}X^{\mathrm{mac}}_{KLx}\]
where $\hat b^{KL,R,\delta,\mathrm{mac}}$ is as in \eqref{e:def_hatb} and $X^{\mathrm{mac}}$ is a collection of independent standard Gaussians.

We can assume that $(\rho_{KL,K'L',R,T,\delta,x})_{x\in\Xi_{KL}}$, $(Y_{KL,K'L',R,T,\delta,x})_{x\in\Xi_{KL}}$ and $Z_{KL,K'L',R,T,\delta}$ are all independent. This allows us to consider the random field $\hat\xi^{KL,K'L',R,T,\delta}$ on $\Xi_{KL}$, where
\[\hat\xi^{K,L,K',L',R,T,\delta}_x=Z_{KL,K'L',R,T,\delta,x}-\sqrt{2d}\log(KL)+\rho_{KL,K'L',R,T,\delta,x}\left(Y_{KL,K'L',R,T,\delta,x}+\tilde\gamma(KL)\right).\]

\begin{lemma}\label{l:coupling_max}
Under Assumptions \ref{a:logupp}, \ref{a:logbd}, \ref{a:micro}, \ref{a:macro}, \ref{a:sparseT}, \ref{a:sparseR} and \ref{a:lln} we have that
\begin{equation}\label{e:coupling_max}
\begin{split}
&\limsup_{T\to\infty}\limsup_{\substack{\delta\to0\\R\to\infty}}\limsup_{(L,K,L',K')\Rightarrow\infty}\limsup_{N\to\infty}\\
&\quad\dist\Big(\max_{v\in V_N^{K,L,K',L',R,T,\delta}}\xi^{N,KL,K'L',R,\delta}_v,\max_{x\in \hat\Xi_{KL}}\hat\xi^{K,L,K',L',R,T,\delta}_x\Big)=0
\end{split}
\end{equation}
\end{lemma}
\begin{proof}
This is shown similar to \cite[Theorem 4.5]{DRZ17}, so we just describe the most important steps.

Let $\ep>0$ be given. Our goal is to construct a coupling between the two fields such that on an event of probability $1-\ep$ their maxima are at most $\ep$ apart. As a first step, we choose $T,R$ large and $\delta$ small and fix them for the moment.

Now we proceed just like in the proof of \cite[Theorem 4.5]{DRZ17}. It is obvious from the definition of $\hat\xi^{K,L,K',L',R,T,\delta}$ how to couple it to $\xi^{N,KL,K'L',R,\delta,\mathrm{coarse}}$. To construct the coupling for the other random variables, the main observation is that with high probability the maximum of 
$\xi^{N,KL,K'L',R,T,\delta}$ is assumed at a point where $\xi^{N,KL,K'L',\mathrm{fine}}$ is exceptionally large,
and hence in the regime where the right-tail asymptotics from Lemma \ref{l:right_tail_unif} are sharp.
To make this precise, let $v_{\mathrm{max}}=\argmax_{v\in V_N^{K,L,K',L',R,T,\delta}}\xi^{N,KL,K'L',R,T,\delta}_v$. 
Note that by Lemma \ref{l:upper_right_tail} and \ref{l:upper_left_tail} the maximum of $\xi^{N,KL,K'L',R,T,\delta}$ over $V_N^{K,L,K',L',R,T,\delta}$ is tight around $m_N$, and similarly the maximum of $\xi^{N,KL,K'L',R,\delta,\mathrm{coarse}}$ over $V_N^{K,L,K',L',R,T,\delta}$ is tight around $m_J$. We have
\[\lim_{J\to\infty}\lim_{N\to\infty}m_N-m_{\left\lfloor\frac{N}{J} \right\rfloor}-m_J-\tilde\gamma(J)=\infty,\]
and so we have
\begin{equation}\label{e:coupling_max1}
\xi^{N,KL,K'L',\mathrm{fine}}_{v_{\mathrm{max}}}\ge m_{\left\lfloor\frac{N}{KL} \right\rfloor}+\tilde\gamma(KL)
\end{equation}
on a event whose probability tends to 1 in the limit $N\to\infty$ and then $(L,K,L',K')\Rightarrow\infty$. So for our purposes we can assume that \eqref{e:coupling_max} occurs.

By similar arguments we can assume that
\begin{equation}\label{e:coupling_max2}
\max_{v\in V_N^{K,L,K',L',R,T,\delta}}\xi^{N,KL,K'L',\mathrm{fine}}_v\le m_{\left\lfloor\frac{N}{KL} \right\rfloor}+KL.
\end{equation}
For $w'\in \W_{N,\left\lfloor\frac{N}{KL}\right\rfloor}$ define
\[M^{N,K,L,K',L',R,T,\delta,\mathrm{fine}}_{w'}=\max_{v\in V_N^{K,L,K',L',R,T,\delta}(w')}\xi^{N,KL,K'L',\mathrm{fine}}_v.\]
Note that Lemma \ref{l:right_tail_unif} gives us good estimates for $\pp(M^{N,K,L,K',L',R,T,\delta,\mathrm{fine}}_{w'}\ge m_{\left\lfloor{N}/{KL}\right\rfloor}+z$ when $z\in[z',\gamma(K'L',R,T,\delta)]$. In view of \eqref{e:coupling_max1} and \eqref{e:coupling_max2} we thus want to arrange our parameters such that
\begin{equation}\label{e:coupling_max3}
[\tilde\gamma(KL),KL]\subset[z',\gamma(K'L',R,T,\delta)].
\end{equation}
But this is not a problem: We first choose $z'$ large enough that the error in \eqref{e:right_tail_unif} is small, then $KL$ large enough in terms of $z'$, and finally $K'L'$ large enough in terms of $KL$.
Now it is clear how to construct the coupling. Our goal is that $\rho_{KL,K'L',R,T,\delta}=1$ if and only if the corresponding $M^{N,K,L,K',L',R,T,\delta,\mathrm{fine}}$ exceeds 
$m_{\left\lfloor{N}/{KL} \right\rfloor}+\tilde\gamma(KL)$, and in that case $Y_{KL,K'L',R,T,\delta}$ should equal $M^{N,K,L,K',L',R,T,\delta,\mathrm{fine}}-\tilde\gamma(KL)$.
We can achieve this exactly, but \eqref{e:coupling_max2}, \eqref{e:coupling_max3} and Lemma \ref{l:right_tail_unif} imply that this is possible up to errors that are bounded by $C\ep''_{R,T,\delta}$ in the limit $N\to\infty$ and then $(L,K,L',K')\Rightarrow\infty$ (see \cite{DRZ17} for more details).
Note that by \eqref{e:coupling_max1} there is at least one $x$ with $\rho_{KL,K'L',R,T,\delta,x}=1$, and so our coupling indeed ensures that $\max_{v\in V_N^{K,L,K',L',R,T,\delta}}\xi^{N,KL,K'L',R,\delta}_v$ and $\max_{x\in \hat\Xi_{KL}}\hat\xi^{K,L,K',L',R,T,\delta}_x$ are close. All errors vanish in the limit $N\to\infty$, then $(L,K,L',K')\Rightarrow\infty$, then $R\to\infty, \delta\to0$ and finally $T\to\infty$.
This allows us to conclude the proof.
%\todo{More details? Formalize the heuristics in the last paragraph? Ofer: I think this is enough}
\end{proof}

\begin{proof}[Proof of Theorem \ref{t:mainthm}]
The convergence in distribution of the recentered maximum follows directly from Lemma \ref{l:comparison_maxima} and Lemma \ref{l:coupling_max}. 

The characterisation of the limit law as a randomly shifted Gumbel distribution requires more work. However, the argument is completely analogous to the proof of \cite[Theorem 1.4]{DRZ17}, and so we omit the details.
\end{proof}

\section{Estimates for Green's functions on percolation clusters,
and proof of Theorem  \ref{t:percolation_cluster}}
\label{sec-4}
This section is devoted to the proof of Theorem  \ref{t:percolation_cluster}, based on quantitative homogenization results and a-priori estimates from percolation theory, and is structured as follows. Section \ref{sec-4.1} collects results from the quantitative homogenization literature, and in particular from \cite{AD18,DG21}, and introduces some useful modifications with respect to locality of various parameters and uniformity as function of the parameter $p$. The next three sections are devoted to the proof of Theorem \ref{t:percolation_cluster}, by checking the assumptions of Theorem \ref{t:mainthm}. In Section \ref{sec-4.2}, which deals with the ``easy to check'' assumptions of Theorem \ref{t:mainthm}, 
we slightly modify the results from Section \ref{sec-4.1} and verify all assumptions of Theorem \ref{t:mainthm} except for \ref{a:logupp} and \ref{a:sparseT}. These arguments are valid for all $p>p_c=1/2$. Section \ref{sec-4.3} 
is devoted to the statements and proofs of some large deviation results for
the percolation cluster $\CC_\infty$, for $p$ close to $1$.
The verification of the remaining Assumptions \ref{a:logupp} and 
\ref{a:sparseT} for $p$ sufficiently close to 1 is carried out in Section \ref{sec-4.4}, and is based on 
the results in Section \ref{sec-4.3}.
\subsection{Homogenization on percolation clusters}
\label{sec-4.1}
In this section we introduce various notation and recall various existing results on the structure of supercritical clusters, and on the large-scale behaviour of the graph Laplacian on the cluster. We take various results on the homogenization of the percolation cluster from \cite{AD18,DG21}. Most results have the character that there is a random scale $\Rscale_v$ indexed by $v\in\Z^2$ such that on length-scales $\ge\Rscale_v$ around $v$ some desirable property holds. Where possible we also state an estimate that the event $\{\Rscale\le R\}$ is asymptotically a local event. This is not done in the works \cite{AD18,DG21} from which we cite. However in most cases it is very easy do so, so we state and prove the (asymptotic) locality right away. A notable exception is Theorem \ref{t:green_asympt}, where there is no direct way to obtain locality of the relevant random scale, and so we postpone discussing the locality in that theorem to the next section.

In this section we only consider dimension $d=2$. We denote by $\nabla$ and $\Delta$ the lattice gradient and Laplacian, and by $\bar\nabla$ and $\bar\Delta$ their continuous counterparts. We also denote by $L^q$ discrete $q$-norms, and by $\bar L^q$ the (standard) continuous $q$-norms. Thereby we hopefully avoid any risk of confusion.

Recall that $V_N(w)$ denoted the cube of sidelength $N$ and lower left corner $w$, and that $\partial^+ V_N(w)$
denotes its outer boundary. In this section we will frequently encounter cubes with given center. Thus we define $Q_N(w)=w+[-N+1,N-1]^2\cap \Z^2$.
 
In the following we denote the sidelength of a cube $Q:=V_N(w)$ by $\ell(Q):=N$ (note that $\ell(Q_N(w))=2N-1$ with our definitions).

Let $0<\Lambda^-\le\Lambda^+$. We let $\F$ be the Borel $\sigma$-algebra on $(\{0\}\cup[\Lambda^-,\Lambda^+])^{E(\Z^2)}$, and we let $\PP$ be an i.i.d. probability measure on $(\{0\}\cup[\Lambda^-,\Lambda^+])^{E(\Z^2)},\F$. In other words, for each edge $e\in E(Z^2)$ we consider an i.i.d random variable $\A(e)$ supported in $\{0\}\cup[\Lambda^-,\Lambda^+]$, and denote by $\PP$ the joint law of the $\A(e)$.

For some $A\subset E(\Z^2)$, we let $\F_A$ be the $\sigma$-algebra generated by $\{\A(e)\colon e\in A\}$. When $A$ is equal to the edges of the subgraph induced by $Q_N(w)$, we write, slightly abusing notation, $\F_{Q_N(w)}$ for the corresponding $\sigma$-algebra.

Let $p=\PP(\A(e)>0)$. We always assume $p>1/2$, the critical threshold for bond percolation on $\Z^2$ \cite{K80}. It is well-known that $\PP$-almost surely there is a unique infinitely cluster $\CC_\infty$, i.e. a infinite connected subgraph of $\Z^2$ such that $\A(e)>0$ for each $e\in E(\CC_\infty)$. We denote by $\dist_{\CC_\infty}$ the graph distance on $\CC_\infty$ (as an unweighted graph).

Unless indicated otherwise, we regard $p,\Lambda^-,\Lambda^+$ as fixed and so do not make explicit how various constants depend on this quantities, except that in various
locations, with $\Lambda^-,\Lambda^+$ fixed  we specify uniformity with respect to $p$ in a neighborhood of $1$.

We let $\Delta_{\A}$ be the graph Laplacian on $\CC_\infty$, i.e.
\[\Delta_{\A}U(x)=\sum_{\{u,v\}\in E(\CC_\infty)}\A(\{u,v\})(U(u)-U(v)).\]

If $V_N(w)\subset\Z^2$ is a box, we can define a Green's function $G^\A_{V_N(w)}\colon\CC_\infty\times\CC_\infty\to\R$ as follows: if $v\notin\V_N(w)$, then $G^\A_{V_N(w)}(\cdot,v)=0$, while if $v\in\V_N(w)$ then $G^\A_{V_N(w)}(\cdot,v)$ is the unique function which is 0 on $\CC_\infty\setminus V_N(w)$ and such that $-\Delta_\A G^\A_{V_N(w)}(\cdot,v)=\delta_v$ on $V_N(w)\cap \CC_\infty$. Note that $G^\A_{V_N(w)}$ is also the Green's function for the (variable speed, continuous time) random walk on $\CC_\infty$ that is killed when exiting $V_N(w)$.

We also want to define a Green's function $G^\A$ on all of $\CC_\infty$. This is only possible when we fix some normalization. We let $G^\A\colon\CC_\infty\times\CC_\infty\to\R$ be a function such that $-\Delta_\A G^\A(\cdot,y)=\delta_y$, $G^\A(y,y)=0$ and $\lim_{|x|\to\infty}\frac{1}{|x|}G^\A(x,y)=0$ for all $x,y$. There is a unique such function $\PP$-almost surely. The function $G^\A$ is also the potential kernel for random walk on $\CC_\infty$.

From the theory of stochastic homogenization it is known that there is a deterministic constant $\oA$ such that the operator $-\Delta_\A$ homogenizes to the operator $-\oA\bar\Delta$, i.e. a scalar multiple of the standard continuous Laplacian. In the following we rely in particular on the quantitative homogenization results from \cite{AD18,DG21}. Let us quote the results that are the most important for us.

We begin with an estimate for $G^\A$ that follows easily from \cite[Theorem 1.2]{DG21}. Here and in the rest of the chapter, we use the phrase
``there is a uniform $s>0$'' to mean 
that $s>0$ is 
bounded below uniformly in $p$ in a neighborhood of $1$, for
fixed ${\Lambda^+}/{\Lambda^-}$.

\begin{theorem}\label{t:green_asympt}
For each $p>1/2$ and each ${\Lambda^+}/{\Lambda^-}\ge1$ there are  random variables $\mathsf{K}_v\in\R$ and $\Rscale^{\mathrm{Green}}_v\in\Z$ indexed by $v\in\CC_\infty$, such that
\[\left|G^\A(u,v)+\frac{1}{2\pi\oA}\log|u-v|-\mathsf{K}_v\right|\le\frac{1}{\oA|u-v|^{3/4}}\]
whenever $u,v\in\CC_\infty$ are such that $|u-v|\ge \Rscale^{\mathrm{Green}}_v$.
In addition $R^{\mathrm{Green}}_v$ satisfies the tail bound
\[\PP(\Rscale^{\mathrm{Green}}_v\le R)\ge 1- C\e^{-R^{s}/C}\]
with some constant $C>0$ and some uniform $s>0$ independent of $v\in \CC_\infty$.
\end{theorem}
Note that here (unlike the following statements) we do not make any claim about the locality of $\{\Rscale^{\mathrm{Green}}_v\le R\}$. The problem is that $G^\A$ is by definition a global object depending on all of $\CC_\infty$ and not just on $\CC_\infty$ intersected with some finite box.

\begin{proof}
In \cite[Theorem 1.2]{DG21}, the same statement is given, with $1/(\oA|u-v|^{3/4})$ 
replaced by ${C_\ep}/({|u-v}|^{1-\ep})$ for any $\ep>0$. Our version follows by 
taking some $\ep<1/4$ and increasing $\Rscale^{\mathrm{Green}}_v$ by a bounded factor.

The fact that $s$ is bounded below uniformly as claimed
%non-increasing as a function of $p$ 
is not made explicit in \cite{AD18,DG21}, however it follows from tracking their proofs (cf. also the discussion below \cite[Remark 1.1]{AD18}). 
In particular, the quantitative estimates in \cite{AD18} improve 
as $p-p_c$ increases, and thus it should even be possible to take
$s$ non-increasing in $p\in\left(\frac12,1\right]$.
\end{proof}

In order to state the other results we need, we need to introduce more notation to describe the large-scale structure of $\CC_\infty$.

In \cite{AD18,DG21} the lattice is partitioned into triadic cubes in such a way that all the cubes $Q$ are well-connected (this means in particular that they contain a unique large crossing cluster, denoted $\CC_*(Q)$). The details of the construction can be found in \cite[Section 2.1]{DG21}, but are not important for our purposes here. What is important is that almost surely there is a partition $\mathcal{P}$ of $\Z^2$ into well-connected cubes that satisfies
\[\bigcup_{Q\in\mathcal{P}}\CC_*(Q)\subset \CC_\infty\]
(note that in general we do not have equality here). We also need that this partition is local and comes with tail bounds on the size of the cubes. 
More precisely, we have that for each triadic cube $Q=Q_{3^\ell}(3^j)$ the event $Q\in\mathcal{P}$ is $F_{Q_{3^{\ell+1}}(3^j)}$-measurable.}
As a consequence of this, while $\CC_\infty$ is of course a global object, we can approximate it by local objects as follows:
\begin{lemma}\label{l:local_approx_cluster}
For each $p>1/2$ there is $C>0$ such that for any $R\in\N$ there are events $\Ev_v^{R,\mathrm{Clust}}\in\F_{Q_{9R}(v)}$  indexed by $v\in\Z^2$ such that
\[\PP(\Ev_v^{R,\mathrm{Clust}})\ge 1-C\e^{-R/C}\]
and such that on the event $\Ev_v^{R,\mathrm{Clust}}$ we have
\[\CC_*(Q_{3R}(v))\cap Q_R(v)=\CC_\infty\cap Q_R(v).\]
\end{lemma}
We emphasize that $\CC_*(Q_{3R}(v))$ only depends on the bonds in $Q_{3R}(v)$, and so is a genuinely local object.
\begin{proof}
This follows from the results in \cite[Section 2]{AD18}, in particular Equation (2.18) there.
\end{proof}

%Furthermore, if for $v\in\Z^2$ we let $Q(v)$ be the unique cube in $\mathcal{P}$ containing $v$, then
%\begin{equation}\label{e:diam_cube_in_P}
%\PP(\ell(Q(v))\ge L)\le C\e^{L/C}
%\end{equation}
%where the event here is $\F_{v+[-3L,3L]^2}$-measurable.

Already in the statement of Theorem \ref{t:percolation_cluster} we needed to consider points not in
$\CC_\infty$ and their projection to $\CC_\infty$. Let us recall this here: For $v\in\Z^2$ we denote by $v^*$ the point in $\CC_\infty$ closest to $v$ in Euclidean distance, where in case of a tie we take the lexicographically first point. The following lemma provides us with a tail bound on the distance between $v$ and $v^*$.
For $v\in\Z^2$, define the random variable
\begin{equation}
  \label{eq-Rvdist}
  \Rscale^{\mathrm{Dist}}_v=|v-v^*|_\infty\in\N.
\end{equation}
\begin{lemma}\label{l:dist_to_cluster}
  For each $p>1/2$ there is $C>0$ so that, for any $R\in\N$, we have
\begin{equation}\label{e:dist_to_cluster}
\PP(\Rscale^{\mathrm{Dist}}_v\le R,\Ev_v^{R,\mathrm{Clust}})\ge 1-C\e^{R/C}.
\end{equation}
Moreover, the event $\{\Rscale^{\mathrm{Dist}}_v\le R\}\cap \Ev_v^{R,\mathrm{Clust}}$ is $\F_{Q_{9R}(v)}$-measurable.
\end{lemma}
\begin{proof}
The tail bound \eqref{e:dist_to_cluster} follows from Lemma \ref{l:local_approx_cluster} and \cite[Lemma 2.7]{AD18}. The locality follows from the fact on the event $\Ev_v^{R,\mathrm{Clust}}$, the clusters $\CC_\infty$ and $\CC_*(Q_{3R}(v))$ agree on $Q_R(v)$, so that $v^*$ is also the point in $\CC_*(Q_{3R}(v))$ closest to $v$.
\end{proof}

We also have good upper bounds on the size of the largest cube in $\mathcal{P}$ intersecting a given cube. We denote by $Q^{\mathcal{P}}(v)$ the unique cube in $\mathcal{P}$ containing $v$.
\begin{lemma}\label{l:maximal_size_P}
For each $p>1/2$ and any $\ep>0$ there are random variables $\Rscale^{\mathrm{Part},\ep}_w\in\N$ indexed by $w\in\Z^2$ such that the following holds: Let $w\in\Z^2$, $R\in\N$ with $R\ge\Rscale^{\mathrm{Part},\ep}_w$. Then
\begin{equation}\label{e:maximal_size_P}
\max_{v\in V_R(w)}\ell(Q^{\mathcal{P}}(v))\le R^{\ep}.
\end{equation}
In addition, 
there are constants $C>0$ and uniform $s>0$   so that for any $w\in\Z^2$ and any $R,R'\in\N$ with $R'\ge R$ there is an
event $\Ev^{R,R',\mathrm{Part},\ep}_w\in\F_{Q_{4R'}(w)}$ 
such that
\begin{equation}\label{e:maximal_size_P1}
\PP(\Ev^{R,R',\mathrm{Part},\ep}_w)\ge 1-C\e^{-R^{s}/C}
\end{equation}
and 
\begin{equation}\label{e:maximal_size_P2}
\PP(\Rscale^{\mathrm{Part},\ep}_w> R,\Ev^{R,R',\mathrm{Part},\ep}_w)\le C\e^{-R'^{s}/C}.
\end{equation}
%and $s$ is non-decreasing as a function of $p$.
\end{lemma}
Note that the exponent in the right hand side of \eqref{e:maximal_size_P2}
is $R'$, in contrast with \eqref{e:maximal_size_P1}. The estimates \eqref{e:maximal_size_P1} and \eqref{e:maximal_size_P2} combined formalize the intuition that the event $\Rscale^{\mathrm{Part},\ep}_w> R$ is approximately local in the sense that if it occurs then either the local event $\Ev^{R,R',\mathrm{Part},\ep}_w$ or an event with very small probability occurs. Similar statements will occurs many times in the next lemmas.
\begin{proof}
The event that \eqref{e:maximal_size_P} holds for some fixed $R$ is 
local. We will use this fact to  prove both the existence of $\Rscale^{\mathrm{Part},\ep}_w$ and that of $\Ev^{R,R',\mathrm{Part},\ep}_w$. 
%The existence of $\Ev^{R,R',\mathrm{Part},\ep}_w$ follows from the fact that the event that \eqref{e:maximal_size_P} holds for some fixed $N$ is local.
We will use this argument many times in the following, so we spell out the details here.
Define the event
\[\Ev^{\rho,\mathrm{Part},\ep}_w:=\left\{\max_{v\in V_\rho(w)}\ell(Q^{\mathcal{P}}(v))\le \rho^{\ep}\right\}.\]
For a fixed $\rho$ the event $\Ev^{\rho,\mathrm{Part},\ep}_w$ only 
depends on the events $Q\in\mathcal{P}$ for the cubes $Q$ intersecting $V_\rho(w)$ and of diameter $\le \rho^\ep$. Thus the event that \eqref{e:maximal_size_P} holds is in $\F_{Q_{\rho+3\rho^\ep}(w)}\subset \F_{Q_{4\rho}(w)}$.
Moreover, from \cite[Proposition 2]{DG21} it follows that
\begin{equation}\label{e:maximal_size_P3}
\PP(\Ev^{\rho,\mathrm{Part},\ep}_w)\ge 1-C\e^{-\rho^{s}/C}.
\end{equation}

We set now
\[\Rscale^{\mathrm{Part},\ep}_w=\inf\left\{R\in\N\colon \Ev^{\rho,\mathrm{Part},\ep}_w \quad \forall \rho\ge R\right\}\]
so that \eqref{e:maximal_size_P} clearly holds, and further set
\begin{align*}
\Ev^{R,R',\mathrm{Part},\ep}_w&=\left\{\Rscale^{\mathrm{Part},\ep}_w\le R\text{ or }\Rscale^{\mathrm{Part},\ep}_w>R'\right\}
=\bigcap_{R<\rho\le R'}\Ev^{\rho,\mathrm{Part},\ep}_w.
\end{align*}
This event is obviously in $\F_{Q_{4R'}(w)}$, and the estimates \eqref{e:maximal_size_P1} and \eqref{e:maximal_size_P2} follow directly from \eqref{e:maximal_size_P3}.

Regarding the uniformity of $s$, 
see the comment in the proof of Theorem \ref{t:green_asympt}.
\end{proof}

Next, we need a lower bound on the density of the cluster in a large cube.
\begin{lemma}\label{l:density_cluster}
For each $p>1/2$ there are $C>0$ 
%$s>0$, 
and random variables 
$\Rscale^{\mathrm{Dens}}_v\in\N$ indexed by $v\in\Z^2$ such for any 
$v\in\Z^2$, $R\in\N$ with $R\ge\Rscale^{\mathrm{Dens}}_v$,
\[\left|\CC_\infty\cap Q_R(v)\right|\ge\frac{R^2}{C}.\]
In addition, for any $v\in\Z^2$ and any $R,R'\in\N$ with $R\le R'$ 
there is an event $\Ev^{R,R',\mathrm{Dens}}_v\in\F_{Q_{9R'}(v)}$ 
such that
\[
\PP(\Ev^{R,R',\mathrm{Dens}}_v)\ge 1-C\e^{-R/C}
\]
and 
\[
\PP(\Rscale^{\mathrm{Dens}}_v> R,\Ev^{R,R',\mathrm{Dens}}_v)\le C\e^{-{R'}/C}.
\]

\end{lemma}
\begin{proof}
From \cite[Theorems 2,3]{DS88} we have that
there is a constant $C>0$ such that
\begin{equation}
  \label{eq-DSlater}
  \PP\left(\left|\CC_\infty\cap Q_R(v)\right|<\frac{R^2}{C}\right)\le  C\e^{-R/C}.
  \end{equation}
The argument to deduce from this the lemma is similar, but not completely the same as in Lemma \ref{l:maximal_size_P}. Namely, the event 
\[\Ev^{R,\mathrm{Dens}}:=\left\{\left|\CC_\infty\cap Q_R(v)\right|<\frac{R^2}{C}\right\}\] only depends on $\CC_\infty\cap Q_R(v)$, and so $\Ev^{R,\mathrm{Dens}}_v\cap \Ev^{R,\mathrm{Clust}}_v$ is $\F_{Q_{9R}(v)}$-measurable. We can now define
\begin{align*}
  \Rscale^{\mathrm{Dens}}_v&=\inf\left\{R\in\N\colon(\Ev^{\rho,\mathrm{Dens}}\cap \Ev^{\rho,\mathrm{Clust}}_v)\;\forall \rho\ge R\right\},\\
  \Ev^{R,R',\mathrm{Dens}}_v&=\bigcap_{R<\rho\le R'}\Ev^{\rho,\mathrm{Dens}}_v\cap \Ev^{\rho,\mathrm{Clust}}_v,
\end{align*}
and directly obtain the claimed result from \eqref{eq-DSlater} and Lemma
\ref{l:local_approx_cluster}.
%Regarding the monotonicity of $s$, see the proof of Theorem \ref{t:green_asympt}.
\end{proof}
%\begin{remark}
%  \corO{The bound, and therefore the  \eqref{eq-DSlater} holds
%\end{remark}

Given a function $f\colon\CC_\infty\to\R$, we can make use of the partition $\mathcal{P}$ to define an interpolated version $[f]_{\mathcal{P}}\colon\R^2\to\R$ as follows. We first define $[f]_{\mathcal{P}}$ on $\Z^2$ by setting $[f]_{\mathcal{P}}(v)=f(z(Q^{\mathcal{P}}(v)))$ for $v\in\Z^2$, where $Q^{\mathcal{P}}(v)$ is the unique cube in $\mathcal{P}$ containing $v$ and $z$ is the lattice point closest to its center (where in case of ties we take the lexicographically first). Then extend to $\R^2$ by letting $[f]_{\mathcal{P}}$ be piecewise constant on each cube $v+\left[-\frac12,\frac12\right)^2$. We also fix a mollifier $\eta\in C_c^\infty(\R^2)$ such that $\int\eta=1$. Throughout, write $\hat V_N(w)=V_N(w)\cup \partial^+ V_N(w))$.

We are now ready to state an estimate for the homogenization error for the Dirichlet problem. 
\begin{theorem}\label{t:homog_dirich}
  For each $p>1/2$ and any ${\Lambda^+}/{\Lambda^-}\ge1$ there are $C>0$, 
  $\ep_{\mathrm{Homog}}\in(0,1)$,  and random variables $\Rscale^{\mathrm{Homog}}_w\in\N$ indexed by $w\in\Z^2$ such that the following holds.
Let $w\in\Z^2$, $N\ge \Rscale^{\mathrm{Homog}}_w$, let $F\colon\CC_\infty\cap \hat V_N(w)\to\R$, and let $U\colon\CC_\infty\cap \hat V_N(w) \to\R$ be the unique solution of the discrete elliptic equation
\begin{alignat*}{2}
	\begin{aligned}
		-\Delta_\A U&=0 && \text{in }\CC_\infty\cap V_N(w),\\
		U&=F && \text{on }\CC_\infty\cap\partial^+ V_N(w).
	\end{aligned}	
	\end{alignat*}
Furthermore let $\bar U\colon \hat V_N(w)
% w+[-1,N]^2
\to\R$ be the unique solution of the continuous elliptic equation
\begin{alignat*}{2}
	\begin{aligned}
		-\oA\bar\Delta \bar U&=0 && \text{in }w+(-1,N)^2,\\
		\bar U&=[F]_{\mathcal{P}}*\eta  && \text{on }\partial(w+(-1,N)^2).
	\end{aligned}	
	\end{alignat*}
Then we have the estimate
\begin{equation}\label{e:homog_dirich}
\|U-\bar U\|_{L^2(\CC_\infty\cap V_N(w))}\le CN^{2-\ep_{\mathrm{Homog}}}\|\nabla F\|_{ L^\infty(\hat V_N(w))}.
\end{equation}
In addition, there is a uniform $s>0$ so that for any $w\in\Z^2$ and any $R,R'$ with $R\le R'$ there is an event $\Ev^{R,R',\mathrm{Homog}}_w\in\F_{Q_{9R'}(w)}$ such that, 
\[
\PP(\Ev^{R,R',\mathrm{Homog}}_w)\ge 1-C\e^{-R^s/C}
\]
and 
\[
\PP(\Rscale^{\mathrm{Homog}}_w> R,\Ev^{R,R',\mathrm{Homog}}_w)\le C\e^{-R'^s/C}.
\]
%and $s$ is non-decreasing as a function of $p$ for each fixed $\frac{\Lambda^+}{\Lambda^-}$.
\end{theorem}
We expect that the $L^2$-norms of  $U$ and $\bar U$ are both at most of order $N^2\|\nabla F\|_{ L^\infty(\hat V_N(w))
%\cup\partial^+ V_N(w))
}$, so \eqref{e:homog_dirich} effectively means a gain of $N^{-\ep_{\mathrm{Homog}}}$ over the trivial estimate.

\begin{proof}
This is essentially \cite[Theorem 1]{AD18}. However, the result there is phrased differently. The result, as stated above, follows immediately from \cite[Theorem 3.2]{DG21}. There the parabolic version of the theorem is given, but the elliptic version follows by taking all functions constant in time.
The locality follows again as in Lemma \ref{l:density_cluster}. Regarding the 
uniformity of $s$, see the proof of Theorem \ref{t:green_asympt}.
\end{proof}

We also need various functional equalities and elliptic regularity estimates for $-\Delta_\A$-harmonic functions, most of which hold true beyond some random lengthscale. We collect them in the following theorem. We denote by $(U)_A$ the average of $U$ over a set $A$, and remark for future use that for $A\subset A'$ we have
\begin{equation}\label{e:monotonicity_L2}
\|U-(U)_A\|_{L^2(A)}\le \|U-(U)_{A'}\|_{L^2(A)}\le \|U-(U)_{A'}\|_{L^2(A')}.
\end{equation}

\begin{theorem}\label{t:regularity}
For each $p>1/2$ and each ${\Lambda^+}/{\Lambda^-}\ge1$ there are $C>0$  and random variables $\Rscale^{\mathrm{Regul}}_v\in\N$ indexed by $v\in\Z^2$, such that the following holds:
Let $v\in\Z^2$, $R\ge \rho\ge\Rscale^{\mathrm{Regul}}_v$. Let $U\colon\CC_\infty\cap Q_{R+1}(v)\to\R$ satisfy $-\Delta_A U=0$ in $\CC_\infty\cap Q_R(v)$. Then we have the elliptic regularity estimate
\begin{equation}\label{e:regularity}
\left\|U-(U)_{\CC_\infty\cap Q_{\rho}(v))}\right\|_{L^2(\CC_\infty\cap Q_{\rho}(v))}\le \frac{C\rho^2}{R^2}\left\|U-(U)_{\CC_\infty\cap Q_{R}(v))}\right\|_{L^2(\CC_\infty\cap Q_{R}(v))}.
\end{equation}

In addition, for any $v\in\Z^2$ and any $R,R'\in\N$ with $R\le R'$ there is an event $\Ev^{R,R',\mathrm{Regul}}_v\in\F_{Q_{9R'}(v)}$ such that, 
with $s>0$ uniform,
\[
\PP(\Ev^{R,R',\mathrm{Regul}}_v)\ge 1-C\e^{-R^s/C}
\]
and 
\[
\PP(\Rscale^{\mathrm{Regul}}_v> R,\Ev^{R,R',\mathrm{Regul}}_v)\le C\e^{-R'^s/C}.
\]
%and $s$ is non-decreasing as a function of $p$ for each fixed $\frac{\Lambda^+}{\Lambda^-}$.
\end{theorem}
\begin{proof}
This result be extracted from \cite{AD18} and \cite{DG21}. Indeed, for \eqref{e:regularity} we can first assume that $R\ge 2\rho$, as otherwise the result follows trivially from \eqref{e:monotonicity_L2}. We can also 
assume that $(U)_{\CC_\infty\cap(Q_R(v)))}=0$, as otherwise we replace $U$ with $U-(U)_{\CC_\infty\cap Q_R(v))}$. Then \eqref{e:regularity} follows from \cite[Theorem 2 (iii)]{AD18} by taking $k=0$. Namely the only bounded $-\Delta_\A$-harmonic functions are constants, and so the optimal $\phi$ on the left-hand side in \cite[Theorem 2 (iii)]{AD18} is indeed $(U)_{\CC_\infty\cap Q_{\rho}(v))}$.

Once more, the locality follows as in Lemma \ref{l:density_cluster}, and the 
uniformity of $s$  as in Theorem \ref{t:green_asympt}.
\end{proof}

\subsection{Proof of Theorem \ref{t:percolation_cluster}, first part}
\label{sec-4.2}

We can now begin with the proof of Theorem \ref{t:percolation_cluster}. As mentioned above,
Assumptions \ref{a:logupp} and \ref{a:sparseT} are much harder to establish than the other five assumptions. So in this subsection we only consider the five "easier" assumptions, and postpone the discussion of Assumptions \ref{a:logupp} and \ref{a:sparseT} to the following sections.

We will use the results collected in the previous section as a toolbox, but we cannot directly use them in our setting.
As a first step, we establish improved versions of Theorem \ref{t:green_asympt} and \ref{t:homog_dirich}. For Theorem \ref{t:green_asympt} the improvement consists in showing that the event $\{\Rscale^{\mathrm{Green}}_v>R\}$ can be approximated by a local event. For Theorem \ref{t:homog_dirich} this fact is rather obvious. There a different improvement is necessary. Namely we claim that (at least under sufficiently strong assumptions on $f$) we can get rid of the interpolation $[f]_{\mathcal{P}}*\eta$. We also use the regularity results from Theorem \ref{t:regularity} to replace the $L^2$-estimate in \eqref{e:homog_dirich} by a pointwise estimate, provided we stay far enough from the boundary.

We begin with a version of Theorem \ref{t:green_asympt}.
Recall that for $v\in \Z^2$, we denote by 
$v^*\in\Z^2$  the point in $\CC_\infty$ closest to $v$ 
(with ties broken  lexicographically).
\begin{lemma}\label{l:green_asympt_improved}
For each $p>1/2$ and each ${\Lambda^+}/{\Lambda^-}\ge1$ there are $C>0$ and random variables $\mathsf{K}'_v\in\R$ and $\Rscale^{\mathrm{Green}'}_v\in\N$ indexed by $v\in\Z^2$ with the following property: if $u,v\in\Z^2$ satisfy $|u-v|\ge \Rscale^{\mathrm{Green}'}_v$ and we also have $u\in\CC_\infty$ or $|u-v|\ge \Rscale^{\mathrm{Green}'}_u$, then
\begin{equation}\label{e:green_asympt_improved}
\left|G^\A(u^*,v^*)+\frac{1}{2\pi\oA}\log|u-v|-\mathsf{K}'_v\right|\le\frac{1}{\oA|u-v|^{1/2}}.
\end{equation}

In addition, for any $v\in\Z^2$ and any $R,R'\in\N$ with $R\le {R'}/{2}$ there is an event $\Ev^{R,R',\mathrm{Green}'}_v\in\F_{Q_{9R'}(v)}$ such that,
with $s>0$ uniform,
\begin{equation}\label{e:green_asympt_improved1}
\PP(\Ev^{R,R',\mathrm{Green}'}_v)\ge 1-C\e^{-R^{s}/C}
\end{equation}
and 
\begin{equation}\label{e:green_asympt_improved2}
\PP(\Rscale^{\mathrm{Green}'}_v> R,\Ev^{R,R',\mathrm{Green}'}_v)\le C\e^{-R'^{s}/C}.
\end{equation}
\end{lemma}
In other words, the bad event $\Rone_v>R$ implies that either the local bad event $ (\Ev^{R,R',\mathrm{Green}'}_v)^\complement$ or a very unlikely event occurs.

\begin{proof}
We take 
\begin{align*}
\Rscale^{\mathrm{Green'}}_v=\Rscale^{\mathrm{Green}}_{v^*}\vee(\Rscale^{\mathrm{Dist}}_v)^4\vee C', \qquad 
\mathsf{K}'_v=\mathsf{K}_{v^*},
\end{align*}
with the objects on the right-hand sides of the equality
as in Theorem \ref{t:green_asympt} and \eqref{eq-Rvdist},
%Lemma \ref{l:dist_to_cluster}, 
and $C'$ is a constant to be chosen shortly. We also set $s$ as the minimum of the exponents $s$ in the quoted results.

Let us first show that with this choice we have \eqref{e:green_asympt_improved}. From Theorem \ref{t:green_asympt} we know that
\begin{equation}\label{e:green_asympt_improved3}
\left|G^\A(u^*,v^*)+\frac{1}{2\pi\oA}\log|u^*-v^*|-\mathsf{K}_{v*}\right|\le\frac{1}{|u^*-v^*|^{3/4}}.
\end{equation}
The assumption $|u-v|\ge\Rscale^{\mathrm{Green'}}_v$ implies that
\[|v-v^*|_\infty=\Rscale^{\mathrm{Dist}}_v\le(\Rscale^{\mathrm{Green'}}_v)^{1/4}\le|u-v|^{1/4}.\]
Similarly, if $|u-v|\ge\Rscale^{\mathrm{Green'}}_v$ holds then
\[|u-u^*|_\infty\le|u-v|^{1/4},\]
and the same is trivially true if $u\in\CC_\infty$. So under the given assumptions on $u,v$ we have in any case that
\[\left|\frac{|u^*-v^*|}{|u-v|}-1\right|\le\frac{2\sqrt{2}}{|u-v|^{3/4}}\]
and hence \eqref{e:green_asympt_improved3} implies
\[\left|G^\A(u^*,v^*)+\frac{1}{2\pi\oA}\log|u-v|-\mathsf{K}'_v\right|\le\frac{C}{|u-v|^{3/4}}.\]
If we choose $C'$ large enough, this implies \eqref{e:green_asympt_improved}.

Regarding $\Ev^{R,R',\mathrm{Green'}}_v$, a first attempt might be to define
\[\Ev^{R,R',\mathrm{Green'}}_v=\left\{\Rscale^{\mathrm{Green'}}_v\le R\text{ or }\Rscale^{\mathrm{Green'}}_v> R'\right\}\cap \Ev_v^{R',\mathrm{Clust}}\cap \left\{\Rscale^{\mathrm{Dist}}\le R^{1/4}\right\}\cap \Ev_v^{R^{1/4},\mathrm{Clust}}.\]
With this definition \eqref{e:green_asympt_improved1} and \eqref{e:green_asympt_improved2} would easily follow. However, the event is not local. The problem is that $\left\{\Rscale^{\mathrm{Green}}_v\le R\right\}$ depends on all of $\CC_\infty$, not just on $\CC_\infty$ intersected with some large box.

So we need to do another approximation step. We let $B_{r}(v)=\{u\in\Z^2\colon |u-v|\le r\}$ be the intersection of $\Z^2$ with an (Euclidean) ball of radius $r$ centered at $v$, and define the event $\Ev^{R,R',\mathrm{Green'}}_v$ as follows: We let $\Ev^{R,R',\mathrm{Green'}}_v$ be the certain event if $R\le C''$ for some $C''$ to be determined later, and if $R>C''$ we set 
\begin{align}\label{e:green_asympt_improved4}
%\begin{split}
&\Ev^{R,R',\mathrm{Green'}}_v\nonumber\\
&=\left\{\exists K\text{ s.t. }\left|G^\A_{B_{R'}(v)}(u,v^*)+\frac{1}{2\pi\oA}\log|u-v|-K\right|\le\frac{1}{\oA|u-v|^{1/2}}\ \forall u\in B_v(R')\text{ with }|u-v|\ge R\right\}\nonumber\\
&\qquad\cap\Ev_v^{R',\mathrm{Clust}}\cap \left\{\Rscale^{\mathrm{Dist}}_{v}\le R^{1/4}\right\}\cap \Ev_v^{R^{1/4},\mathrm{Clust}}.
%\end{split}
\end{align}
By the same argument as in the proof of Lemma \ref{l:dist_to_cluster} we have $\Ev^{R,R',\mathrm{Green}}_v\in\F_{Q_{9R'}(v)}$.

To see \eqref{e:green_asympt_improved1}, observe that we know
\[\PP\left(\Ev_v^{R',\mathrm{Clust}}\cap \left\{\Rscale^{\mathrm{Dist}}_{v}\le R^{1/4}\right\}\cap \Ev_v^{R^{1/4},\mathrm{Clust}}\right)\ge 1-C\e^{-R^s/C}\]
from Lemmas \ref{l:local_approx_cluster} and  \ref{l:dist_to_cluster}. 
So it suffices to show that the first event on the right-hand side of \eqref{e:green_asympt_improved4} also occurs with probability at least $1-C\e^{-R^s/C}$. In fact, we will show that this event occurs whenever 
$\Rscale^{\mathrm{Green}}_{v^*}\le {R}/{2}$ and $(\Rscale^{\mathrm{Dist}}_v)^4\le R$ (which by Theorem \ref{t:green_asympt} and Lemma \ref{l:local_approx_cluster} is sufficient).
Indeed, if $\Rscale^{\mathrm{Green}}_{v^*}\le {R}/{2}$, then there is a constant $K=K_{v^*}$ such that
\[\left|G^\A(u,v^*)+\frac{1}{2\pi\oA}\log|u-v^*|-K\right|\le\frac{1}{|u-v^*|^{3/4}}\ \forall u\in \CC_\infty\text{ with }|u-v^*|\ge \frac{R}{2}.\]
By the same argument as earlier in the proof, we know that if $u\in\CC_\infty$ and $|u-v|\ge \vee\Rscale^{\mathrm{Green}'}_v$ then $\vee|v-v^*|\le|u-v|^{1/4}$, and so also
\[\left|G^\A(u,v^*)+\frac{1}{2\pi\oA}\log|u-v|-K\right|\le\frac{C}{|u-v|^{3/4}}\ \forall u\in \CC_\infty\text{ with }|u-v|\ge R.\]
In particular, on $B_{R'}(v)\setminus B_{R'-1}(v)$, the oscillation of 
$G^\A(\cdot,v^*)$ is at most ${C}/{R'^{3/4}}$. Now we can employ the maximum principle on the domain $B_{R'}(v)$ with comparison function $G^\A(\cdot,v)-c_\pm$ for suitable $c_\pm$ to conclude that
\[\left|G^\A_{B_v(R')}(u,v^*)+\frac{1}{2\pi\oA}\log|u-v|-K'\right|\le\frac{1}{|u-v|^{3/4}}+\frac{C}{R'^{3/4}}\le \frac{C}{|u-v|^{3/4}}\]
for some $K'$ (which can be taken as $K'=K_{v^*}+(\log R')/2\pi\A$)
whenever $u\in B_{R'}(v)$ with $|u-v|\ge R$. As soon as $C''$ is large bounded, the right-hand side here is bounded by ${1}/{|u-v|^{1/2}}$, and so indeed $\Ev^{R,R',\mathrm{Green}}_v$ occurs.

Regarding \eqref{e:green_asympt_improved2} it suffices to check that if $\Ev^{R,R',\mathrm{Green'}}_v$ occurs and $\Rscale^{\mathrm{Green'}}_v\ge R$, then already $\Rscale^{\mathrm{Green'}}_v\ge R'$. This claim follows again from the maximum principle on $B_v(R')$ by a very similar argument.
\end{proof}

In Theorem \ref{t:homog_dirich} the locality of the relevant event is much easier to show. Here the main challenge is to show  a pointwise estimate instead of just an $L^2$ error estimate. Recall that $\bar\nabla$ and $\bar \Delta$ denote,
respectively, the continuous gradient and Laplacian operators.

\begin{lemma}\label{l:homog_dirich_improv}
For each $p>1/2$ and each ${\Lambda^+}/{\Lambda^-}\ge1$ there are $C>0$, $\ep_{\mathrm{Homog'}}>0$ and random variables $\Rscale^{\mathrm{Homog'}}_w\in\N$ and $\Rscale^{\mathrm{Regul'}}_v\in\N$ indexed by $w,v\in\Z^2$ respectively, such that the following holds:
Let $w\in\Z^2$, $N\ge \Rtwo_w$, 
 let $F\in C^{1,1}(w+(-1,N)^2,\R)$, and let $U\colon\CC_\infty\cap \hat V_N(w)\to\R$ be the unique solution of the discrete elliptic equation
\begin{alignat*}{2}
	\begin{aligned}
		-\Delta_\A U&=0 && \text{in }\CC_\infty\cap V_N(w),\\
		U&=F && \text{on }\CC_\infty\cap\partial^+ V_N(w).
	\end{aligned}	
	\end{alignat*}
Furthermore let $\hat U\colon w+(-1,N)^2\to\R$ be the unique solution of the continuous elliptic equation
\begin{alignat*}{2}
	\begin{aligned}
		-\oA\bar\Delta \hat U&=0 && \text{in }w+(-1,N)^2,\\
		\hat U&=F && \text{on }\partial(w+(-1,N)^2).
	\end{aligned}	
	\end{alignat*}
Then for every $\delta>0$ there are constants $C_\delta$, $N_\delta$ such that if $N\ge N_\delta$ then for any $v\in V_N^\delta(w)$ with $\Rscale^{\mathrm{Regul'}}_v\le\delta N$ we have the estimate
\begin{equation}\label{e:homog_dirich_improv}
|U(v^*)-\hat U(v^*)|\le C_\delta\left( N^{1-\ep_{\mathrm{Homog'}}}(\Rscale^{\mathrm{Regul'}}_v)^{1/2}+(\Rscale^{\mathrm{Regul'}}_v)^2\right)\|\bar\nabla F\|_{\bar L^\infty(w+(-1,N)^2)}.
\end{equation}

In addition, for any $w,v\in\Z^2$ and any $R,R'\in\N$ with $R\le {R'}/{2}$ there are events $\Ev^{R,R',\mathrm{Homog'}}_w\in\F_{Q_{9R'}(w)}$ and $\Ev^{R,R',\mathrm{Regul}}_v\in\F_{Q_{9R'}(v)}$ such that, with uniform $s>0$, 
\begin{align}
\pp(\Ev^{R,R',\mathrm{Homog'}}_w)&\ge 1-C\e^{-R^{s}/C}\label{e:homog_dirich_improv1},\\
\pp(\Ev^{R,R',\mathrm{Regul'}}_v)&\ge 1-C\e^{-R^{s}/C}\label{e:homog_dirich_improv2},
\end{align}
and 
\begin{align}
\pp(\Rscale^{\mathrm{Homog'}}_w>R,\Ev^{R,R',\mathrm{Homog'}}_w)&\le C\e^{-R^{s}/C},\label{e:homog_dirich_improv3}\\
\pp(\Rscale^{\mathrm{Regul'}}_v>R,\Ev^{R,R',\mathrm{Regul'}}_v)&\le C\e^{-R^{s}/C}.\label{e:homog_dirich_improv4}
\end{align}
%and $s$ is non-decreasing as a function of $p$ for each fixed $\frac{\Lambda^+}{\Lambda^-}$.
\end{lemma}

\begin{proof}
In order to improve the $L^2$-estimate from Theorem \ref{t:homog_dirich} into a pointwise estimate, we will use the regularity of $\hat U$ (which follows from standard Schauder estimates) and of $U$ (which follows from Theorem \ref{t:regularity}). More precisely, we use these regularity results to show that if $U(v)$ and $\hat U(v)$ were far apart, then the averages of $U$ and of $\hat U$ over $Q_S(v)$ for some mesoscopic length-scale $S$ would still be be far apart, in contradiction to the $L^2$-estimate.

\emph{Step 1: Preliminaries and $L^2$-estimate}\\
We let $\ep_{\mathrm{Homog'}}={\ep_{\mathrm{Homog}}}/{2}$ with the $\ep_{\mathrm{Homog}}$ from Theorem \ref{t:homog_dirich}.
We take 
\begin{align*}
\Rscale^{\mathrm{Homog'}}_w&=\Rscale^{\mathrm{Homog}}_w\vee \Rscale^{\mathrm{Part},\ep_{\mathrm{Homog}}}_w, \qquad
\Rscale^{\mathrm{Regul'}}_v=\Rscale^{\mathrm{Regul}}_{v^*}\vee\Rscale^{\mathrm{Dens}}_{v^*}\vee\Rscale^{\mathrm{Dist}}_v,
\end{align*}
where the random scales on the right-hand sides are those from Lemma \ref{l:dist_to_cluster}, Lemma \ref{l:maximal_size_P}, Lemma \ref{l:density_cluster}, Theorem \ref{t:homog_dirich} and Theorem \ref{t:regularity}. We take $s$ as the minimum of the exponents $s$ in the quoted results.
We also define
\begin{align*}
\Ev^{R,R',\mathrm{Homog'}}_w&=\Ev^{R,R',\mathrm{Homog}}_w\cap \Ev^{R,R',\mathrm{Part},\ep_{\mathrm{Homog}}}_w,\\
\Ev^{R,R'\mathrm{Regul'}}_v&=\Ev^{R,R',\mathrm{Regul}}_v\cap \Ev^{R,R',\mathrm{Dens}}_{v}\cap\left\{(\Rscale^{\mathrm{Dist}})^4\le R\right\}.
\end{align*}
With these definitions, \eqref{e:homog_dirich_improv1} and \eqref{e:homog_dirich_improv3} and the fact that $\Ev^{R,R',\mathrm{Homog'}}_w\in\F_{w+[-9R',9R']^2}$ follow directly from Theorem \ref{t:homog_dirich} and Theorem \ref{t:regularity}. Similarly, \eqref{e:homog_dirich_improv2} and \eqref{e:homog_dirich_improv4} and $\Ev^{R,R',\mathrm{Regul}}_v\in\F_{Q_{9R'}(v)}$ follow from Lemma \ref{l:dist_to_cluster}, Lemma \ref{l:maximal_size_P} and Lemma \ref{l:density_cluster}.

%For $\Ev^{R,R',\mathrm{Regul'}}$ there is the additional subtlety that $\Rscale^{\mathrm{Regul}'}$ is defined in terms of random scales at $v^*$ and not just at $v$. Nonetheless, we obtain \eqref{e:homog_dirich_improv2} and \eqref{e:homog_dirich_improv4} from Lemma \ref{l:dist_to_cluster},Lemma \ref{l:maximal_size_P} and Lemma \ref{l:density_cluster}. Furthermore, if $\Rscale^{\mathrm{Regul'}}_v\le R'$ then $|v-v|_*\le R'$. So from the quoted Lemmas we also obtain that $\Ev^{R,R',\mathrm{Regul'}}_v$  only depends on the bonds in 
%\[Q_{9R'}(v)\cup Q_{9R'}(v^*)\subset Q_{10R'}(w)\]
%and so it is $\F_{Q_{10R'}(v)}$-measurable, as claimed.

The main challenge of the proof is to show \eqref{e:homog_dirich_improv}. For that purpose, we can assume that $v=v^*$.
In this first step, we control the difference between $U$ and $\hat U$ in $L^2$. We claim that
\begin{equation}\label{e:homog_dirich_improv5}
  \|U-\hat U\|_{L^2(\CC_\infty\cap V_N(w))}\le CN^{2-\ep_{\mathrm{Homog}}}\|\bar\nabla F\|_{\bar L^\infty(w+(-1,N)^2)}.
\end{equation}
With $\bar U$ as in the statement of Theorem \ref{t:homog_dirich}, we know that
\begin{equation}\label{e:homog_dirich_improv6}
  \|U-\bar U\|_{L^2(\CC_\infty\cap V_N(w))}\le CN^{2-\ep_{\mathrm{Homog}}}\|\nabla F\|_{ L^\infty(\hat V_N(w))}\le
CN^{2-\ep_{\mathrm{Homog}}}\|\bar\nabla F\|_{\bar L^\infty(w+(-1,N)^2)},
\end{equation}
and so it suffices to control $\|\bar U-\hat U\|_{L^2(\CC_\infty\cap V_N(w))}$. This can be done using the maximum principle for the Laplacian. Indeed, $\bar U-\hat U$ is a solution of the (discrete) elliptic equation
\begin{alignat*}{2}
	\begin{aligned}
		-\oA\Delta (\bar U-\hat U)&=0 && \text{in }w+(-1,N)^2\\
		\bar U-\hat U&=[F]_{\mathcal{P}}*\eta -F && \text{on }\partial(w+(-1,N)^2)
	\end{aligned}	
\end{alignat*}
and so the (discrete) maximum principle implies that
\begin{align*}
\|\bar U-\hat U\|_{L^2(\CC_\infty\cap V_N(w))}&\le N\|\bar U-\hat U\|_{L^\infty(\CC_\infty\cap V_N(w))}\\
&\le N\sup_{x\in\partial(w+(-1,N)^d)}\left|([F]_{\mathcal{P}}*\eta)(x) -F(x)\right|\\
&\le CN\max_{\substack{Q\in\mathcal{P}\\ Q\cap V_N(w)\neq\varnothing}}\ell(Q)\|\bar\nabla F\|_{\bar L^\infty(w+(-1,N)^2)}.
\end{align*}
From Lemma \ref{l:maximal_size_P} we know that the maximum on the right-hand side is at most $N^{\ep_{\mathrm{Homog'}}}$, which is much less than $N^{1-\ep_{\mathrm{Homog'}}}$. So, together with \eqref{e:homog_dirich_improv6}, we have shown \eqref{e:homog_dirich_improv5}.

\emph{Step 2: Regularity of $U$ and $\hat U$}\\
Our task now is to improve the $L^2$-estimate \eqref{e:homog_dirich_improv5} to a pointwise estimate. For that purpose we use the knowledge that both $U$ and $\hat U$ are solutions of elliptic equations, and so we expect them to be regular on small scales, meaning that if they were far apart at some point $u$, then the $L^2$-norm of their difference would also be large.

To make this rigorous, recall that $(U)_A$ denotes the average of $U$ on 
a set $A$. We claim that for any $S$ with $\Rscale^{\mathrm{Regul'}}_v\le S\le \delta N$ we have
\begin{align}
\left\|U-(U)_{\CC_\infty\cap Q_S(v)}\right\|_{L^2(\CC_\infty\cap Q_{\Rscale^{\mathrm{Regul'}}_v}(v))}&\le C_\delta \Rscale^{\mathrm{Regul'}}_v S\|\bar\nabla F\|_{\bar L^\infty(w+(-1,N)^2)},\label{e:regularity_U}\\
\left\|\hat U-(\hat U)_{\CC_\infty\cap Q_S(v)}\right\|_{L^2(\CC_\infty\cap Q_{\Rscale^{\mathrm{Regul'}}_v}(v))}&\le C_\delta \Rscale^{\mathrm{Regul'}}_v S\|\bar\nabla F\|_{\bar L^\infty(w+(-1,N)^2)}.\label{e:regularity_hatU}
\end{align}

We begin by establishing \eqref{e:regularity_hatU}. This follows easily from standard elliptic regularity theory. Indeed, by the maximum principle we have
\begin{align*}
\left\|\hat U-(\hat U)_{w+(-1,N)^2}\right\|_{\bar L^2(w+(-1,N)^2)}&\le N\left(\sup_{x\in(w+[-1,N]^2)}\hat U(x)-\inf_{x\in\CC_\infty\cap V_N(w)}\hat U(x)\right)\\
&\le N\left(\sup_{x\in\partial(w+[-1,N]^2)}F(x)-\inf_{x\in\partial(w+[-1,N]^2)}F(x)\right)\\
&\le CN^2\|\bar\nabla F\|_{\bar L^\infty(w+(-1,N)^2)}
\end{align*}
and by interior Schauder estimates for $-\bar\Delta$ we then have
\[
\left\|\bar\nabla\hat U\right\|_{\bar L^\infty(w+(\delta N,(1-\delta)N)^2)}\le\frac{C_\delta}{N^2}\left\|\hat U-(\hat U)_{w+[-1,N]^2}\right\|_{\bar L^2(w+[-1,N]^2)}\le C_\delta \|\bar\nabla F\|_{\bar L^\infty(w+(-1,N)^2)}.
\]
This means in particular that
\[\sup_{x\in Q_S(v)}U(x)-\inf_{x\in Q_S(v)}U(x)\le C_\delta S\|\bar\nabla F\|_{\bar L^\infty(w+(-1,N)^2)},\]
which easily implies \eqref{e:regularity_hatU}.

For \eqref{e:regularity_U} we cannot proceed like this, because we do not have Schauder estimates for $-\A\Delta$. Instead we will make use of Theorem \ref{t:regularity}, which can be thought of as a large-scale $C^{0,1}$-estimate. But as we do not actually control the gradient of $U$, but only $L^2$-norms on various scales, the proof of \eqref{e:regularity_U} will be quite technical. The  following argument is similar to the well-known proof that Campanato spaces embed into H\"{o}lder spaces. 

Note first that by the (discrete) maximum principle we have
\begin{align*}
\left\|U-(U)_{\CC_\infty\cap V_N(w))}\right\|_{L^2(\CC_\infty\cap V_N(w))}&\le N\left(\sup_{u\in\CC_\infty\cap V_N(w)}U(u)-\inf_{u\in\CC_\infty\cap V_N(w)}U(u)\right)\\
&\le N\left(\sup_{x\in\partial(w+[-1,N]^2)}F(x)-\inf_{x\in\partial(w+[-1,N]^2)}F(x)\right)\\
&\le CN^2\|\bar\nabla f\|_{\bar L^\infty(w+(-1,N)^2)}
\end{align*}
and so in particular, by Theorem \ref{t:regularity} and \eqref{e:monotonicity_L2},
\begin{equation}\label{e:homog_dirich_improv7}
\begin{split}
\left\|U-(U)_{\CC_\infty\cap Q_S(v)}\right\|_{L^2(\CC_\infty\cap Q_S(v))}&\le\frac{CS^2}{(\delta N)^2}\left\|U-(U)_{\CC_\infty\cap Q_{\delta N}(v)}\right\|_{L^2(\CC_\infty\cap Q_{\delta N}(v))}\\
&\le\frac{CS^2}{(\delta N)^2}\left\|U-(U)_{\CC_\infty\cap V_N(w))}\right\|_{L^2(\CC_\infty\cap V_N(w))}\\
&\le C_\delta S^2\|\bar\nabla F\|_{\bar L^\infty(w+(-1,N)^2)}.
\end{split}
\end{equation}
If now $\Rscale^{\mathrm{Regul'}}_v\ge {S}/{2}$, then \eqref{e:regularity_U} follows from another application of \eqref{e:monotonicity_L2}. So in the following we assume $\Rscale^{\mathrm{Regul'}}_v\le{S}/{2}$.

Let $k_0=\left\lfloor\log_2\left(\frac{S}{\Rscale^{\mathrm{Regul'}}_v}\right)\right\rfloor$ and note that 
$k_0\ge1$. Pick some $k\in\{0,1,\ldots,k_0\}$. Applying Theorem \ref{t:regularity} with length-scales $S$ and $2^{-k}S$ and using \eqref{e:homog_dirich_improv7} we find
\begin{equation}\label{e:homog_dirich_improv8}
\begin{split}
  \left\|U-(U)_{\CC_\infty\cap Q_{2^{-k}S}(v))}\right\|_{L^2(\CC_\infty\cap Q_{2^{-k}S}(v))}&\le C\frac{2^{-2k}S^2}{(\delta N)^2}N^2\|\bar\nabla F\|_{\bar L^\infty(w+(-1,N)^2)}\\
&=C_\delta2^{-2k}S^2\|\bar\nabla F\|_{\bar L^\infty(w+(-1,N)^2)}.
\end{split}
\end{equation}
Let now $k\in\{0,1,\ldots,k_0-1\}$. Using \eqref{e:homog_dirich_improv8} for $k$ and $k+1$ as well as the lower bound on the cluster density from Lemma \ref{l:density_cluster} and \eqref{e:monotonicity_L2}, we can now estimate that
\begin{align*}
&\left|(U)_{\CC_\infty\cap Q_{2^{-k}S}(v)}-(U)_{\CC_\infty\cap Q_{2^{-k-1}S}(v)}\right|\\
&\le\frac{1}{2^{-k}S}\left\|(U)_{\CC_\infty\cap Q_{2^{-k}S}(v)}-(U)_{\CC_\infty\cap Q_{2^{-k-1}S}(v)}\right\|_{L^2(\CC_\infty\cap Q_{2^{-k-1}S}(v))}\\
&\le\frac{1}{2^{-k}S}\Big(\left\|U-(U)_{\CC_\infty\cap Q_{2^{-k}S}(v)}\right\|_{L^2(\CC_\infty\cap Q_{2^{-k}S}(v))}
+\left\|U-(U)_{\CC_\infty\cap Q_{2^{-k-1}S}(v)}\right\|_{L^2(\CC_\infty\cap Q_{2^{-k-1}S}(v))}\Big)\\
&\le C_\delta 2^{-k}S\|\bar\nabla F\|_{\bar L^\infty(w+(-1,N)^2)}.
\end{align*}
We can sum this estimate over $k\in\{0,1,\ldots,k_0-1\}$ and obtain that
\begin{equation}\label{e:homog_dirich_improv9}
  \left|(U)_{\CC_\infty\cap Q_{2^{-k_0}S}(v)}-(U)_{\CC_\infty\cap Q_S(v)}\right|\le C_\delta S\|\bar\nabla F\|_{\bar L^\infty(w+(-1,N)^2)}.
\end{equation}
Next, we combine \eqref{e:homog_dirich_improv9} with \eqref{e:homog_dirich_improv8} (for $k=k_0$) and use that $\Rscale^{\mathrm{Regul'}}_v\le 2^{-k_0}S\le2\Rscale^{\mathrm{Regul'}}_v$ to see that
\begin{align*}
&\left\|U-(U)_{\CC_\infty\cap Q_S(v)}\right\|_{L^2(\CC_\infty\cap Q_{\Rscale^{\mathrm{Regul'}}_v}(v))}\\
&\le\left\|U-(U)_{\CC_\infty\cap Q_S(v)}\right\|_{L^2(\CC_\infty\cap Q_{2^{-k_0}S}(v))}\\
&\le\left\|U-(U)_{\CC_\infty\cap Q_{2^{-k_0}S}(v)}\right\|_{L^2(\CC_\infty\cap Q_{2^{-k_0}S}(v))}
+2^{-k_0}S\left|(U)_{\CC_\infty\cap Q_{2^{-k_0}S}(v)}-(U)_{\CC_\infty\cap Q_S(v)}\right|\\
&\le C_\delta2^{-2k_0}S^2\|\bar\nabla F\|_{\bar L^\infty(w+(-1,N)^2)}+C_\delta 2^{-k_0}S^2\|\bar\nabla F\|_{\bar L^\infty(w+(-1,N)^2)}\\
&\le C_\delta \Rscale^{\mathrm{Regul'}}_v S \|\bar\nabla F\|_{\bar L^\infty(w+(-1,N)^2)},
\end{align*}
which is \eqref{e:regularity_U}.

\emph{Step 3: Pointwise estimate}\\
Using the results from the previous steps, it is now easy to finish the proof of \eqref{e:homog_dirich_improv}.

Let $v\in V_N^\delta(w)$. Consider $S$ with $\Rscale^{\mathrm{Regul'}}_v\le S\le \delta N$ (which we will choose shortly). Then \eqref{e:homog_dirich_improv5} and  \eqref{e:regularity_U} imply that
\begin{align}
  \left|U(v)-(U)_{\CC_\infty\cap Q_S(v)}\right|&\le \left\|U-(U)_{\CC_\infty\cap Q_S(v)}\right\|_{L^2(\CC_\infty\cap Q_{\Rscale^{\mathrm{Regul}}_v}(v)}\le C_\delta \Rscale^{\mathrm{Regul'}}_v S\|\bar\nabla F\|_{L^\infty(w+(-1,N)^2)},\label{e:homog_dirich_improv10}\\
  \left|\hat U(v)-(\hat U)_{\CC_\infty\cap Q_S(v)}\right|&\le \left\|\hat U-(\hat U)_{\CC_\infty\cap Q_S(v)}\right\|_{L^2(\CC_\infty\cap Q_{\Rscale^{\mathrm{Regul}}_v}(v)}\le C_\delta \Rscale^{\mathrm{Regul'}}_v S\|\bar\nabla F\|_{L^\infty(w+(-1,N)^2)}.\label{e:homog_dirich_improv11}
\end{align}
Furthermore, Lemma \ref{l:density_cluster}, the Cauchy-Schwarz inequality and \eqref{e:homog_dirich_improv5} imply that 
\begin{equation}\label{e:homog_dirich_improv12}
\begin{split}
\left|(U)_{\CC_\infty\cap Q_S(v)}-(\hat U)_{\CC_\infty\cap Q_S(v)}\right|&\le \frac{C}{S^2}\|U-\bar U\|_{L^1(\CC_\infty\cap Q_S(v))}\\
&\le \frac{C}{S}\|U-\bar U\|_{L^2(\CC_\infty\cap Q_S(v))}\\
&\le \frac{CN^{2-\ep_{\mathrm{Homog}}}}{S}\|\bar\nabla F\|_{\bar L^\infty(w+(-1,N)^2)}.
\end{split}
\end{equation}
Combining \eqref{e:homog_dirich_improv10}, \eqref{e:homog_dirich_improv11} and \eqref{e:homog_dirich_improv12} we see that
\[\left|U(v)-\hat U(v)\right|\le C_\delta\left(\Rscale^{\mathrm{Regul'}}_v S+ \frac{N^{2-\ep_{\mathrm{Homog}}}}{S}\right)\|\bar\nabla F\|_{\bar L^\infty(w+(-1,N)^2)},\]
and all that remains is to optimize $S$. Choosing
\[S=\sqrt{\frac{N^{2-\ep_{\mathrm{Homog}}}}{\Rscale^{\mathrm{Regul'}}_v}}\vee \Rscale^{\mathrm{Regul'}}_v=\frac{N^{1-\ep_{\mathrm{Homog'}}}}{(\Rscale^{\mathrm{Regul'}}_v)^{1/2}}\vee \Rscale^{\mathrm{Regul'}}_v\]
we certainly have $S\ge \Rscale^{\mathrm{Regul'}}_v$, and for $N\ge N_\delta:=\delta^{-1/\ep_{\mathrm{Homog'}}}$ we also have $S\le \delta N$. Now with this choice of $S$ we finally obtain \eqref{e:homog_dirich_improv}.
\end{proof}

Lemma \ref{l:green_asympt_improved} and Lemma \ref{l:homog_dirich_improv} now 
allow us to begin with the proof of Theorem \ref{t:percolation_cluster}. Namely we prove that for $p>1/2$ the Gaussian free field on a percolation cluster satisfies all assumptions of Theorem \ref{t:mainthm} expect \ref{a:logupp} and \ref{a:sparseT}.

Before we begin with the actual proof, let us introduce notation for the relevant continuous Green's functions: We let
\[\bar G^\oA(x,y)=-\frac{1}{2\pi\oA}\log|x-y|\]
be the Green's function of $-\oA\bar\Delta$ on $\R^2$. Furthermore, for $Q\subset\R^2$ open and bounded, we let $\bar G^\oA_Q$ be the Green's function of $-\oA\bar\Delta$ on $Q$. That is, for each $y\in Q$ the function $\bar G^\oA_Q(\cdot,y)$ is the unique solution of
\begin{alignat*}{2}
	\begin{aligned}
		-\oA\bar\Delta \bar G^\oA_Q(\cdot,y)&=\bar\delta_y && \text{in }Q,\\
		\bar G^\oA_Q(\cdot,y)&=0 &&\text{on }\partial Q,
	\end{aligned}	
\end{alignat*}
where $\bar\delta$ denotes the Dirac delta distribution.

The idea of the proof of the first part of Theorem \ref{t:percolation_cluster} is now relatively straightforward: Lemma \ref{l:green_asympt_improved} gives us precise asymptotics for $G^\A$, and Lemma \ref{l:homog_dirich_improv} allows us to compare the difference of $G^\A$ and $G^\A_{V_N(w)}$ with the difference of $\bar G^\oA$ and $\bar G^\oA_{V_N(w)}$. So, combining the two lemmas, we can get very good estimates on $G^\A_{V_N(w)}$. 

\begin{proof}[Proof of Theorem \ref{t:percolation_cluster}, first part]$ $

\emph{Step 1: Preliminaries}\\
To prove Theorem \ref{t:percolation_cluster} we need to check to check whether $\sqrt{2\pi\oA}\varphi'^{\A,N,w}$ satisfies Assumption
\ref{as:main}. Let us begin by relating this field to the Green's functions of the previous lemmas.

Recall that the field $\varphi^{\A,N,w}$ is defined a priori only on $\CC_\infty\cap V_N(w)$, and that we extend it to $V_N(w)$ by setting $\varphi'^{\A,N,w}_v=\varphi^{\A,N,w}_{v^*}$. Now for $u,v\in V_N$ we have
\begin{equation}\label{e:estgreen}
\E\sqrt{2\pi\oA}\varphi'^{\A,N,w}_{u}\sqrt{2\pi\oA}\varphi'^{\A,N}_{v}=2\pi\oA G^\A_{V_N(w)}(u^*,v^*),
\end{equation}
and so all the items in Assumption \ref{as:main} correspond to various statements on $2\pi\oA G^\A_{V_N(w)}$.

We can assume that the $\ep_{\mathrm{Regul'}}$ from Lemma \ref{l:homog_dirich_improv} is at most $1/3$. We make the choices
\begin{align*}
\Rone_v&=\Rscale^{\mathrm{Green}'}_v\vee(\Rscale^{\mathrm{Regul}'}_v)^{1/\ep_{\mathrm{Regul'}}},
\qquad \Rtwo_w=\Rscale^{\mathrm{Homog}'}_w,
\end{align*}
where the random scales on the right-hand sides are as in Lemmas \ref{l:green_asympt_improved} and \ref{l:homog_dirich_improv}.

\emph{Step 2: Verification of \ref{a:micro} and \ref{a:macro}}\\
Let now $w\in\Z^2$. We define the function $H^\A_{V_N(w)}\colon (\CC_\infty\cap V_N(w))\times(\CC_\infty\cap V_N(w))\to\R$ by
\[H^\A_{V_N(w)}(u^*,v^*)=G^\A_{V_N(w)}(u^*,v^*)-\mathsf{K}'_v-G^\A(u^*,v^*)\]
This function can be thought of as the regular part of $G^\A$ on $\CC_\infty\cap V_N(w)$. Similarly, we define its continuous analogue
$\bar H^\oA_{w+(-1,N)^2}\colon (w+(-1,N)^2)\times(w+(-1,N)^2)\to\R$ by
\[\bar H^\oA_{w+(-1,N)^2}(x,y)=\bar G^\oA_{w+(-1,N)^2}(x,y)-\bar G^\oA(x,y)\]
Our main claim is that if $u,v\in V_N^\delta(w)$ and $\Rone_v\vee\Rone_v\le\delta N$, $\Rtwo_w\le N$ then
\begin{equation}\label{e:estgreen1}
\left|H^\A_{V_N(w)}(u^*,v^*)-\bar H^\oA_{w+(-1,N)^2}(u^*,v^*)\right|\le\frac{C_\delta}{N^{\ep_{\mathrm{Homog'}}/2}}.
\end{equation}
To see this, note that $\bar H^\oA_{w+(-1,N)^2}(\cdot,v^*)$ is harmonic on $w+(-1,N)^2$. Furthermore, $\bar H^\oA_{w+(-1,N)^2}(\cdot,v^*)$ is Lipschitz-continuous. Indeed, by interior Schauder estimates we have
\[\left\|\bar\nabla\bar H^\oA_{w+(-1,N)^2}(\cdot,v^*)\right\|_{\bar L^\infty(w+(\delta N,(1-\delta)N)^2)}\le \frac{C_\delta}{N}\] 
while near the boundary we can estimate the gradient of $\bar G^\oA(\cdot,v^*)$ and $\bar G^\oA_{w+(-1,N)^2}(\cdot,v^*)$ separately (for the latter there are no issues near the boundary because of the Schwarz reflection principle) and find that
\begin{align*}
 & \left\|\bar\nabla\bar G^\oA(\cdot,v^*)\right\|_{\bar L^\infty((w+(-1,N)\setminus(w+(\delta N,(1-\delta)N))}+\!\left\|\bar\nabla\bar G^\oA_{w+(-1,N)^2}(\cdot,v^*)\right\|_{\bar L^\infty((w+(-1,N)^2\setminus(w+(\delta N,(1-\delta)N)^2)}\\
 &\qquad \qquad \qquad \qquad \qquad \qquad \qquad \qquad \le \frac{C_\delta}{N}.\end{align*}
Combining the last two estimates, we indeed obtain the global Lipschitz estimate
\begin{equation}\label{e:estgreen2}
\left\|\bar\nabla\bar H^\oA_{w+(-1,N)^2}(\cdot,v^*)\right\|_{\bar L^\infty(w+(-1,N)^2)}\le \frac{C_\delta}{N}.
\end{equation}
Now let $\tilde H^\A_{V_N(w)}\colon (\CC_\infty\cap V_N(w))\times(\CC_\infty\cap V_N(w))\to\R$ be such that $\tilde H^\A_{V_N(w)}(\cdot,v^*)$ is the unique solution of 
\begin{alignat*}{2}
	\begin{aligned}
		-\Delta_\A \tilde H^\A_{V_N(w)}(\cdot,v^*)&=0 && \text{in }\CC_\infty\cap V_N(w),\\
		\tilde H^\A_{V_N(w)}(\cdot,v^*)
		&=\bar H^\oA_{w+(-1,N)^2}(\cdot,v^*) && \text{on }\CC_\infty\setminus V_N(w).
	\end{aligned}	
\end{alignat*}
From Lemma \ref{l:homog_dirich_improv} we obtain that
\begin{align}\label{e:estgreen3}
&\left|\tilde H^\A_{V_N(w)}(u^*,v^*)-\bar H^\oA_{w+(-1,N)^2}(u^*,v^*)\right|\\
&\nonumber
\qquad\le C_\delta\left( N^{1-\ep_{\mathrm{Homog'}}}(\Rscale^{\mathrm{Regul'}}_v)^{1/2}+(\Rscale^{\mathrm{Regul'}}_v)^2\right)\left\|\bar\nabla\bar H^\oA_{w+(-1,N)^2}(\cdot,v^*)\right\|_{\bar L^\infty(w+(-1,N)^2)}.
\end{align}
We can estimate the right-hand side here further. On the one hand, the Lipschitz norm is bounded by \eqref{e:estgreen3}, on the other hand our assumption $\delta N\ge\Rone_v\ge(\Rscale^{\mathrm{Regul'}}_v)^{1/\ep_{\mathrm{Regul'}}}$ implies that $\Rscale^{\mathrm{Regul'}}_v\le C_\delta N^{\ep_{\mathrm{Regul'}}}$. So from \eqref{e:estgreen3} we obtain
\begin{equation}\label{e:estgreen4}
\left|\tilde H^\A_{V_N(w)}(u^*,v^*)-\bar H^\oA_{w+(-1,N)^2}(u^*,v^*)\right|\le\frac{C_\delta}{N^{\ep_{\mathrm{Homog'}}/2}}.
\end{equation}
On the other hand, $H^\A_{V_N(w)}(\cdot,v^*)$ and $\tilde H^\A_{V_N(w)}(\cdot,v^*)$ are close. Indeed, both are $-\Delta^\A$-harmonic functions on $\CC_\infty\cap V_N(w)$, and by Lemma \ref{l:green_asympt_improved} and our assumption $\delta N\ge\Rscale^{\mathrm{Green}}_v$ we have
\[\sup_{u'^*\in\CC_\infty\cap\partial(w+[-1,N]^2)}\left|H^\A_{V_N(w)}(u'^*,v^*)-\tilde H^\A_{V_N(w)}(u'^*,v^*)\right|\le\frac{1}{(\delta N)^{1/2}}\le \frac{C_\delta}{N^{1/2}}\]
so that by the discrete maximum principle we also have
\begin{equation}\label{e:estgreen5}
\left|H^\A_{V_N(w)}(u^*,v^*)-\tilde H^\A_{V_N(w)}(u^*,v^*)\right|\le\frac{C_\delta}{N^{1/2}}.
\end{equation}
Combining \eqref{e:estgreen4} and \eqref{e:estgreen5}, we have established \eqref{e:estgreen1}.

Using this estimate, we can now verify \ref{a:micro} and \ref{a:macro}.
We begin with \ref{a:micro}. The functions $\bar G^\oA_{w+(-1,N)^2}$ and $\bar H^\oA_{w+(-1,N)^2}$ satisfy the scaling relations 
\begin{align}
\bar G^\oA_{w+(-1,N)^2}(x,y)&=\bar G^\oA_{(0,1)^2}\left(\frac{x-w+e}{N+1},\frac{y-w+e}{N+1}\right)\label{e:estgreen7G}\\
\bar H^\oA_{w+(-1,N)^2}(x,y)&=\bar H^\oA_{(0,1)^2}\left(\frac{x-w+e}{N+1},\frac{y-w+e}{N+1}\right)+\frac{1}{2\pi\oA}\log (N+1)\label{e:estgreen7}
\end{align}
where $e:=(-1,-1)\in\Z^2$.
Furthermore, by standard elliptic regularity theory, $\bar H^\oA_{(0,1)^2}$ is a smooth function on $(0,1)^2\times(0,1)^2\setminus\{(x,x)\colon x\in(0,1)^2\}$ that has a continuous extension onto the diagonal. So
\[f(x):=2\pi\oA\lim_{\substack{x'\to x\\x''\to x}}\bar H^\oA_{(0,1)^2}(x',x'')\]
is a well-defined continuous function on $(0,1)^2$. The function $f$ is also bounded above (as can be seen from the maximum principle and the fact that $\bar H^\oA_{(0,1)^2}(\cdot,x)=-\bar G^\oA_{(0,1)^2}(\cdot,x)$ is bounded by $\frac{1}{2\pi\oA}\log\sqrt{2}$ on $\partial(0,1)^2$). We also choose $g(u,v)=2\pi\oA G^\A_{V_N(w)}(u^*,v^*)+\mathsf{K}'_v$. With these definitions, \eqref{e:estgreen1} can be rewritten as
\[\Big|G^\A_{V_N(w)}(u^*,v^*)-\frac{1}{2\pi\oA}g(u,v)-\frac{1}{2\pi\oA}\log (N+1)-\bar H^\oA_{(0,1)^2}\left(\frac{u^*-w+e}{N+1},\frac{v^*-w+e}{N+1}\right)\Big|\le\!\frac{C_\delta}{N^{\ep_{\mathrm{Homog'}}/2}}\]
and hence also
\begin{equation}\label{e:estgreen12}
\begin{split}
&\left|G^\A_{V_N(w)}(u^*,v^*)-\frac{1}{2\pi\oA}g(u,v)-\frac{1}{2\pi\oA}\log N-\frac{1}{2\pi\oA}f\left(\frac{v-w}{N}\right)\right|\\
&\;\le\frac{C_\delta}{N^{\ep_{\mathrm{Homog'}}/2}}+\left|\bar H^\oA_{(0,1)^2}\left(\frac{u^*-w+e}{N+1},\frac{v^*-w+e}{N+1}\right)-\bar H^\oA_{(0,1)^2}\left(\frac{v-w}{N},\frac{v-w}{N}\right)\right|+\!\frac{1}{2\pi\oA}\log\frac{N+1}{N}.
\end{split}
\end{equation}
Under the assumption $|u-v|_\infty\le L$ and $|u-u^*|^4\vee|v-v^*|^4\le\Rone_u\vee\Rone_v\le\delta N$ we have that
$|u^*-v|\vee|v^*-v|\le L+N^{1/4}$, and all of $\frac{u^*-w+e}{N+1},\frac{v^*-w+e}{N+1},\frac{v-w}{N}$ are elements of $\left(\frac{\delta}{2},1-\frac{\delta}{2}\right)^2\times\left(\frac{\delta}{2},1-\frac{\delta}{2}\right)^2$. The function $\bar H^\oA_{(0,1)^2}$ is uniformly continuous on that set, and therefore the second summand on the right-hand side of \eqref{e:estgreen12} is bounded by a function of $\delta$ and ${L}/{N}$ that goes to 0 as $N\to\infty$. From this we immediately obtain the desired estimate for $2\pi\oA G^\A_{V_N(w)}$, i.e. \ref{a:micro}.

The argument for \ref{a:macro} is similar. We choose $h(x,y)=2\pi\oA\bar G^\oA_{(0,1)^2}$. The estimate \eqref{e:estgreen1} and the scaling relation \eqref{e:estgreen7G} imply that 
\begin{equation}\label{e:estgreen6}
\begin{split}
&\left|G^\A_{V_N(w)}(u^*,v^*)-\frac{1}{2\pi\oA}h\left(\frac{u-w}{N},\frac{u-w}{N}\right)\right|\\
&\quad\le\frac{C_\delta}{N^{\ep_{\mathrm{Homog'}}/2}}+\left|G^\A(u^*,v^*)+\mathsf{K}_v-\bar G^\oA(u^*,v^*)\right|\\
&\qquad+\frac{1}{2\pi\oA}\left|h\left(\frac{u^*-w+e}{N+1},\frac{v^*-w+e}{N+1}\right)-h\left(\frac{u-w}{N},\frac{v-w}{N}\right)\right|.
\end{split}
\end{equation}
Under the assumption $\Rone_u\vee\Rone_v\le\delta N$ we know that
$|u^*-u|\vee|v^*-v|\le N^{1/4}$, and so by Lemma \ref{l:green_asympt_improved} the second term on the right-hand side of \eqref{e:estgreen6} is bounded by $|u-v|^{-1/2}$ for $|u-v|\ge \Rone_v$. The third term on the right-hand side is small because we know $|u^*-v^*|\ge{N}/{2L}$ , and $h$ is uniformly continuous away from the diagonal and the boundary of $(0,1)^2$. Thus we obtain \ref{a:macro}.

\emph{Step 3: Verification of \ref{a:logbd}}\\
This step is different than the previous one in that we make no use of \eqref{e:estgreen1}. Instead we will argue directly using the maximum principle.

To do so, we begin with a rather crude estimate. Namely we claim that for any $N$ and $v,u,u'$ we have
\begin{equation}\label{e:estgreen8}
  0\le G^\A_{Q_R(v)}(u^*,u'^*)\le\frac{\dist_{\CC_\infty}(u^*,\partial Q_R(v))\wedge \dist_{\CC_\infty}(u'^*,\partial Q_R(v))}{\Lambda^-}
\end{equation}
where we recall that $\dist_{\CC_\infty}$ denotes the graph distance on $\CC_\infty$. 
Indeed, the lower bound is trivial. For the upper bound, it suffices (by symmetry) to prove the upper bound $\frac{\dist_{\CC_\infty}(u^*, \partial Q_R(v))}{\Lambda^-}$. Moreover, it suffices (by the maximum principle) to consider the case $u=u'=u^*$. It is helpful to consider $G^\A_{Q_R(v)}(\cdot,u)$ as the voltage distribution that arises when we let a unit current flow from $u$ through the electrical network given by the conductances $\A$. Then for each edge the current through it is at most 1, and so the voltage drop along that edge is at most its resistance. In other words, for each edge the difference of $G^\A_{Q_R(v)}(\cdot,u)$ at the two endpoints is at most ${1}/{\Lambda^-}$. If we use this estimate along each of the edges of path $\gamma$ consisting of $\dist_{\CC_\infty}(u,\partial Q_R(v))$ edges connecting $u$ to a vertex in $\CC_\infty\cap \partial Q_R(v)$, we directly obtain
\[G^\A_{Q_R(v)}(u,u)\le\frac{\dist_{\CC_\infty}(u,\partial Q_R(v))}{\Lambda^-}.\]
This establishes \eqref{e:estgreen8}.

%it suffices (by symmetry) to prove the upper bound $\frac{\dist_{\CC_\infty}(u^*, Q_R(v))}{\Lambda^-}$. Next, it suffices (by the maximum principle) to consider the case $u=u'=u^*$. We consider a path $\gamma$ consisting of $\dist_{\CC_\infty}(u,\partial Q_R(v))$ edges connecting $u$ to a vertex in $\CC_\infty\cap \partial Q_R(v)$. If we replace the conductances on all edges in $\gamma$ with $\Lambda^-$ and delete all other edges in $\CC_\infty\cap  Q_R(v)$, then by Rayleigh's monotonicity law the value of $G^\A_{Q_R(v)}(u,u)$ is non-decreasing. But after deletion the problem is one-dimensional, and we can explicitly compute that the resulting value of $G^\A_{ Q_R(w)}(u,u)$ then is precisely $\frac{\dist_{\CC_\infty}(u,\partial Q_R(v))}{\Lambda^-}$. This establishes \eqref{e:estgreen8}.

Note that if $v\in \CC_\infty\cap Q_R(w)$, then $\dist_{\CC_\infty}(v,\partial Q_R(w))\le 4R^2$. Thus, the bound in \eqref{e:estgreen8} is at worst quadratic. This motivates the choice $\alpha_\delta(R)=C'_\delta (R^2+1)$ (for a constant $C'_\delta$ that depends on $\delta$ and will be chosen later). To verify \ref{a:logbd} we have to check whether for $u,v\in V_N^\delta(w)$ we have
\begin{equation}\label{e:estgreen9}
\left|G^\A_{V_n(w)}(u^*,v^*)-\frac{1}{2\pi\A}\log N+\frac{1}{2\pi\A}\log_+|u-v|\right|\le C_\delta\left((\Rone_u)^2+(\Rone_v)^2+1\right).
\end{equation}
As the estimate \eqref{e:estgreen9} is symmetric in $u,v$, we can assume $\Rone_u\le\Rone_v$. If $\Rone_v>\delta N$ then both the upper and lower bound in \eqref{e:estgreen9} follow from \eqref{e:estgreen8} (provided that $C_\delta$ is large enough). So we can assume that $\Rone_v\le\delta N$. 
%By the same argument as in the proof of Lemma \ref{l:green_asympt_improved}, this means in particular $|u-u^*|\vee|v-v^*|\le|u-v|^{1/4}$.
Lemma \ref{l:green_asympt_improved} then implies that $G^\A(\cdot,v^*)$ has oscillation bounded by $C_\delta$ on $\CC_\infty\cap V_N(w)$. Hence the maximum principle on $\CC_\infty\cap V_N(w)$ with comparison functions $G^\A(\cdot,v^*)+\frac{1}{2\pi\oA}\log N-\mathsf{K}'_v\pm C_\delta$ implies that
\begin{equation}\label{e:estgreen10}
\left|G^\A_{V_N(w)}(\cdot,v^*)-\frac{1}{2\pi\oA}\log N-G^\A(\cdot,v^*)\right|\le C_\delta
\end{equation}
on $\CC_\infty\cap V_N(w)$. If now $|u^*-v|>\Rone_v=\Rone_v\vee\Rone_u$, then \eqref{e:estgreen9} directly follows from \eqref{e:estgreen10} and Lemma \ref{l:green_asympt_improved}. If on the other hand $|u^*-v|\le\Rone_v$, we use the maximum principle once again: the estimate \eqref{e:estgreen10} implies that
\[\left|G^\A_{V_N(w)}(\cdot,v^*)-\frac{1}{2\pi\oA}\log \frac{N}{\Rone_v}\right|\le C_\delta\]
on $\CC_\infty\cap Q_{\Rone_v}(v)$, and so the maximum principle on $\CC_\infty\cap Q_{\Rone_v}(v)$ with comparison function $G^\A_{Q_{\Rone_v}(v)}$ implies that
\[\left|G^\A_{V_N(w)}(u^*,v^*)-G^\A_{Q_{\Rone_v}(v)}-\frac{1}{2\pi\oA}\log \frac{N}{\Rone_v}\right|\le C_\delta.\]
Remembering \eqref{e:estgreen8}, this implies that
\[\left|G^\A_{V_N(w)}(u^*,v^*)-\frac{1}{2\pi\oA}\log N\right|\le C_\delta+\frac{1}{2\pi\oA}\log \Rone_v.\]
We also know that 
\[\log_+|u-v|\le\log_+\left((\Rone_u)^{1/4}+|u*-v|\right)\le C(\Rone_u+\Rone_v)\]
and so we obtain \eqref{e:estgreen9} also in this case.

\emph{Step 4: Verification of \ref{a:sparseR}}\\
In order to verify Assumption \ref{a:sparseR}, we need an upper bound on the number of points where $\Rone$ or $\Rtwo$ are large.   In view of Lemmas
\ref{l:green_asympt_improved} and \ref{l:homog_dirich_improv}, our strategy is to write the events $\{\Rscale^{(i)}_v\ge R\}$ as the union of a local event and an event with very small probability. The former events can be controlled using a second moment bound, while the occurrence of the latter events can be controlled with a first moment bound.

We give details, beginning with $\Rone$. For $R\le R'$ we define the event
\[\Ev^{R,R',(1)}=\Ev^{R,R',\mathrm{Green'}}_v\cap \Ev^{R,R',\mathrm{Regul'}}_v.\]
From Lemmas \ref{l:green_asympt_improved} and \ref{l:homog_dirich_improv} we conclude that there are $C$ and $s>0$ such that
\[
\PP(\Ev^{R,R',(1)}_v)\ge 1-C\e^{-R^{s}/C}
\]
and 
\[
\PP(\Rone_v> R,\Ev^{R,R',(1)}_v)\le C\e^{-R'^{s}/C}
\]
and that moreover $\Ev^{R,R',(1)}\in\F_{Q_{9R'}(v)}$. So we know that if $\Rone_v>R$ then either the local event $(\Ev^{R,R',(1)}_v)^\complement$ or a very unlikely event occurs.

Now, given $R$, we choose $R'=\sqrt{{N}/{L}}$. Then for $
w'\in \W_{N,\left\lfloor \frac{N}{L}\right\rfloor}(w_N)$ we have
\begin{align*}
\left|\left\{v\in V_{\left\lfloor \frac{N}{L}\right\rfloor}(w'): \Rone_{v}> R\right\}\right|&=\sum_{v\in V_{\left\lfloor \frac{N}{L}\right\rfloor}(w')}\I_{\Rone_v>R}\\
&\le\sum_{v\in V_{\left\lfloor \frac{N}{L}\right\rfloor}(w')}\I_{(\Ev^{R,R',(1)}_v)^\complement}+\sum_{v\in V_{\left\lfloor \frac{N}{L}\right\rfloor}(w')}\I_{\Rone_v>R,\Ev^{R,R',(1)}_v},
\end{align*}
and thus, for a constant $C'$ to be fixed later,
\begin{equation}\label{e:estgreen11}
\begin{split}
&\PP\left(\left|\left\{v\in V_{\left\lfloor \frac{N}{L}\right\rfloor}(w'): \Rone_{v}> R\right\}\right|\ge C'\left(\frac{N}{L}\right)^2\e^{-R^s/C'}\right)\\
&\le \PP\bigg(\sum_{v\in V_{\left\lfloor \frac{N}{L}\right\rfloor}(w')}
\I_{(\Ev^{R,R',(1)}_v)^\complement}\ge \frac{C'}{2}\left(\frac{N}{L}\right)^2\e^{-R^s/C'}\bigg)\\
&\qquad \qquad +\PP\bigg(\sum_{v\in V_{\left\lfloor \frac{N}{L}\right\rfloor}(w')}\I_{\Rone_v>R,\Ev^{R,R',(1)}_v}\ge \frac{C'}{2}\left(\frac{N}{L}\right)^2\e^{-R^s/C'}\bigg).
\end{split}
\end{equation}
We can bound the second summand here very crudely with a union bound and obtain
\begin{align*}
&\PP\bigg(\sum_{v\in V_{\left\lfloor \frac{N}{L}\right\rfloor}(w')}\I_{\Rone_v>R,\Ev^{R,R',(1)}_v}\ge\frac{C'}{2}\left(\frac{N}{L}\right)^2\e^{-R^s/C'}\bigg)\\
&\le \PP\left(\exists v\colon\Rone_v>R,\Ev^{R,R',(1)}_v\right)
\le C(\frac{N}{L})^2\e^{-(N/L)^{s/2}/C}\le C\e^{-(N/L)^{s/2}/C}.
\end{align*}
For the first summand, we use that the events $\Ev^{R,R',(1)}_v$ and $\Ev^{R,R',(1)}_{v'}$ are independent when $|v-v'|\ge18R'$. Thus we can partition $V_{\left\lfloor \frac{N}{L}\right\rfloor}(w')$ into $(18R')^2=324{N}/{L}$ classes $(A_i)_{i=1}^{(18R')^2}$ of at most $C\left(\frac{N}{18R'L}\right)^2\le C\frac{N}{L}$ elements such that for each $i$ the events corresponding to $v\in A_i$ are independent. Using now, for example, Hoeffding's inequality on each $A_i$, we find
\begin{align*}
&\PP\bigg(\sum_{v\in V_{\left\lfloor \frac{N}{L}\right\rfloor}(w')}\I_{(\Ev^{R,R',(1)}_v)^\complement}\ge\frac{C'}{2}\left(\frac{N}{L}\right)^2\e^{-R^s/C'}\bigg)\\
&\le\sum_{i=1}^{324N/L}\PP\left(\sum_{v\in A_i}\I_{(\Ev^{R,R',(1)}_v)^\complement}\ge\frac{C'}{648}\frac{N}{L}\e^{-R^s/C'}\right)\\
&\le C\frac{N}{L}\exp\left(-2\left(\frac{C'}{648}\frac{N}{L}\e^{-R^s/C'}-\E\left(\sum_{v\in A_i}\I_{(\Ev^{R,R',(1)}_v)^\complement}\right)\right)^2\right)\\
&\le C\frac{N}{L}\exp\left(-2\left(\frac{C'}{648}\frac{N}{L}\e^{-R^s/C'}-C\frac{N}{L}\e^{-R^s/C}\right)^2\right)
\le C\frac{N}{L}\exp\left(-C\frac{N}{L}\e^{-R^s/C}\right),
\end{align*}
where the last step holds provided we have chosen $C'$ large enough.

Inserting these results into \eqref{e:estgreen11}, we have shown that
\begin{align*}
  &\PP\left(\left|\left\{v\in V_{\left\lfloor \frac{N}{L}\right\rfloor}(w'): \Rone_{v}> R\right\}\right|\ge C'\left(\frac{N}{L}\right)^2\e^{-R^s/C'}\right)\\
  &\qquad \qquad \le C\left(\frac{N}{L}\exp\left(-C\frac{N}{L}\e^{-R^s/C}\right)+\e^{-(N/L)^{s/2}/C}\right).\end{align*}
As the right-hand side is exponentially small in $N$, it follows from the Borel-Cantelli lemma that for each fixed $R$ and $L$
\[\PP\left(\limsup_{N\to\infty}\left(\frac{L}{N}\right)^2\left|\left\{v\in V_{\left\lfloor \frac{N}{L}\right\rfloor}(w'): \Rone_{v}> R\right\}\right|\ge C'\e^{-R^s/C'}\right)=0\]
and this in turn easily implies the assertion on $\Rone$ in Assumption \ref{a:sparseR}.

The assertion on $\Rtwo$ is shown very similarly. We define 
\[\Ev^{R,R',(2)}_w=\Ev^{R,R',\mathrm{Homog'}}_v\]
and use that $\Rtwo_w\ge L'$ implies that either $\Ev^{L',R',(2)}_w$ does not hold or $\left\{\Rtwo_w\ge L'\right\}\cap\Ev^{L',R',(2)}_w$. The latter event is so unlikely that with the choice $R'=\sqrt{\frac{N}{L}}$ we can control its occurrence with a union bound, while the former event is local, so that upon partitioning into $CR'^2$ classes we gain independence and can argue using Hoeffding's inequality as before.

\emph{Step 5: Verification of \ref{a:lln}}\\
When we consider $\delta_{\pp^{L',w''}}$ as a random variable with values in $\mathcal{M}_1(V_{L'}(w''))$, then it is $\F_{Q_L'(w'')}$-measurable and its law is invariant under translations of $\A$. So Assumption \ref{a:lln} follows from an appropriate ergodic theorem for random variables with values in the Polish space $\mathcal{M}_1(\R^{V_{L'}(0)})$.
\end{proof}

\subsection{Large deviation results for highly supercritical percolation clusters}
\label{sec-4.3}

All previous results were valid for any $p>\frac12$, and we did not care about the dependence on $p$. For the quantitative estimate in \ref{a:sparseT} we cannot be so careless. As a preparation for the second part of the proof of Theorem \ref{t:percolation_cluster}, we collect here three large deviation results for $\CC_\infty$. We no longer care about whether the relevant events are (asymptotically) local, however we are now interested in the decay rates as $p$ gets close to 1.

The first result complements Lemma \ref{l:density_cluster}, however we want now a $p$-independent lower bound on the density of the cluster. Of course this is only possible if $p$ is close enough to 1.

\begin{lemma}\label{l:density_cluster_highlysc}
For any $\kappa>0$ there is $p^{\mathrm{Dens}}_\kappa<1$ with the following property. For $p^{\mathrm{Dens}}_\kappa\le p\le1$ there is $C>0$ such that for any $v\in\Z^2$, $R\in\N$ we have
\[\PP\left(\left|\CC_\infty\cap Q_R(v)\right|\ge\frac{R^2}{2}\right)\ge 1-C\e^{-\kappa R}.\]
\end{lemma}
\begin{proof}
This is essentially a contour argument. An even stronger result can be found in \cite{DP96}.
\end{proof}

The second result is an estimate for the isoperimetric constant of sufficiently large subsets of $\CC_\infty$. 
%\todo{Is there an older good reference?}
For $A\subset\CC_\infty$ we denote by 
\[\partial^+A=\{u\in\CC_\infty\setminus A\colon\exists v\in A \,\text{such that}\, (u,v)\in E(\CC_\infty)\}\]
the outer (vertex) boundary of $A$ in $\CC_\infty$. Our result then is.
\begin{lemma}\label{l:isoperimetry}
For any $\kappa>0$ there is $p^{\mathrm{Iso}}_\kappa<1$ with the following property. For $p^{\mathrm{Iso}}_\kappa\le p\le1$ there is $C>0$ such that for any $v\in\Z^2$, $R\in\N$ we have
\[\PP\left(\exists A\subset\CC_\infty\colon v\in A,|A|\ge R^2,|\partial^+A|\le |A|^{1/2}\right)\le C\e^{-\kappa R}.\]
\end{lemma}
\begin{proof}
This is once again a straightforward Peierls' argument. The same argument (with the $p$-dependence not made explicit) can be found for example in \cite[Proof of Theorem 2.4]{BM03}.
\end{proof}

The third result is a large-deviation result for the graph distance (or chemical distance), which essentially follows from \cite{AP96}. For this specific version
we use a result of \cite{GM07}.

\begin{lemma}\label{l:chemical_distance}
For any $\kappa>0$ there is $p^{\mathrm{Chem}}_\kappa<1$ with the following property. For $p^{\mathrm{Chem}}_\kappa\le p\le1$ there is $C>0$ such that for any $v\in\Z^2$, $R\in\N$ we have
\begin{equation}\label{e:chemical_distance1}
\PP\left(\max_{u,u'\in \CC_\infty\cap Q_R(v)}\dist_{\CC_\infty}(u,u')\le 24R\right)\ge 1-C\e^{-\kappa R}.
\end{equation}
\end{lemma}
\begin{proof}
Applying \cite[Theorem 1.4]{GM07}, 
  %Following the argument for \cite[Theorem 1.1]{AP96},
  we have for all $p$ sufficiently close to 1 the large deviation estimate
\begin{equation}\label{e:chemical_distance2}
\PP(\dist_{\CC_\infty}(u,u')\ge 2|u-v|)\le Ce^{-\kappa|u-u'|}
\end{equation}
for any $u,u'\in\CC_\infty$. 
This bound is useful for us when $|u-u'|\ge cR$. To get an estimate also for $u,u'$ which are closer to each other, we proceed as follows. We can assume that $\CC_\infty\cap Q_{3R}(v)$ is non-empty, as otherwise the left-hand side of \eqref{e:chemical_distance1} is trivially equal to 1. So take some $\bar u\in \CC_\infty\cap (Q_{3R}(v)\setminus Q_{3R}(v))$. If there are $u,u'\in \CC_\infty\cap Q_R(v)$ with $\dist_{\CC_\infty}(u,u')\ge24R$, then one of $\dist_{\CC_\infty}(u,\bar u)$ and $\dist_{\CC_\infty}(u',\bar u)$ must be at least $12R$, while $R\le|u-\bar u|\le6R $ and $\rho\le|u'-\bar u|\le6R $. Thus, using \eqref{e:chemical_distance2} and a union bound,
\begin{align*}
&\PP\left(\exists u,u'\in \CC_\infty\cap Q_R(v)\colon \dist_{\CC_\infty}(u,u')> 24R\right)\\
&\quad\le\PP\left(\exists u\in\CC_\infty\cap Q_R(v),\bar u\in \CC_\infty\cap (Q_{3R}(v)\setminus Q_{2R}(v))\colon \dist_{\CC_\infty}(u,\bar u)\ge 12R \right)\\
&\quad\le\sum_{u\in\CC_\infty\cap Q_{R}(v)}\sum_{\bar u\in \CC_\infty\cap (Q_{3R}(v)\setminus Q_{2R}(v))}\PP(\dist_{\CC_\infty}(u,\bar u)\ge 12R)
\le CR^4\e^{-\kappa R},
\end{align*}
which implies \eqref{e:chemical_distance1}.
\end{proof}

\subsection{Proof of Theorem \ref{t:percolation_cluster}, second part}
\label{sec-4.4}
In this section we will complete the proof of Theorem \ref{t:percolation_cluster} by showing that for $p$ sufficiently close to 1, also Assumptions \ref{a:logupp} and \ref{a:sparseT} are satisfied for a suitable choice of $\T_\cdot$.

In the previous section the dependence of various quantities on $p$ and ${\Lambda^+}/{\Lambda^-}$ was unimportant for as, as anyhow all relevant results held for all supercritical $p$. This will be different in this section, and we change our convention on notation slightly: the generic constant $C$ might still depend on $p$ or ${\Lambda^+}/{\Lambda^-}$, but for all other named constants we now indicate explicitly their dependence on $p$ and ${\Lambda^+}/{\Lambda^-}$.

The Assumptions \ref{a:logupp} and \ref{a:sparseT} are at first glance very similar to Assumptions \ref{a:logbd} and \ref{a:sparseR}. However, establishing them is substantially more challenging, mainly because of two new difficulties.
The first 
is that we now require estimates not just in $V_N^\delta$ for some $\delta>0$, but up to the boundary. Near the boundary, the full-space Green's function $G^\A$ is no longer a useful comparison function to be used for the maximum principle, as it was used e.g. in the proof of Lemma \ref{l:green_asympt_improved}.
 Instead we will need to use half-space Green's functions as comparison function. Therefore, as a first step we prove asymptotics for these, by using Lemma \ref{l:green_asympt_improved} together with the reflection principle.

 The other difficulty is even more serious. Namely \ref{a:sparseT} is a quantitative upper bound on the number of points $v\in V_N(w)$ where $\T_v$ is large. Our argument for \ref{a:sparseR} was also quantitative, but the resulting bounds are too weak to be useful here. In more detail, for \ref{a:logbd} we have used Lemma \ref{l:green_asympt_improved} together with the maximum principle to obtain estimates for $G^\A(\cdot,v)$ on lengthscales at least $\Rone_v$, and used the deterministic estimate \eqref{e:estgreen8} to control $G^\A(\cdot,v)$ in the box $Q_{\Rone_v}(v)$. This has lead to bounds on $G^\A(u,v)$ with error term of order $(\Rone_v)^2$.
In order to establish \ref{a:logupp}, we would thus need to choose $\T_v=C(\Rone_v)^2$. In view of the tail-bound 
\begin{equation}\label{e:estgreenquant}
\PP(\Rone_v\ge R)\le C\e^{-R^s/C}
\end{equation}
we would then obtain the tail-bound
\[\PP(\T_v\ge T)\le C\e^{-T^{s/2}/C}\]
where $s>0$ is some small exponent. This only gives us a chance to also establish \ref{a:sparseT} if we are able to take $s=2$ (or even better $s>2$).

So a natural question is whether \eqref{e:estgreenquant} holds for some $s\ge2$. Unfortunately, the answer to that question is no, and the best possible $s$ one can hope for is $s=1$. The problem (which is explained well in \cite[Section 1.4]{DG21}) is that one needs to change only $CR$ bonds to (almost) disconnect a box $Q_R(v)$ from its complement, which makes the environment very irregular on length-scale $R$ around $v$. Therefore, the probability that $\Rone_v>R$ should be at least the probability of this particular bad event, which is at least $c\e^{-R/C}$.

Thus, with the approach that worked well for Assumptions \ref{a:logbd} and \ref{a:sparseR}, we cannot hope to also show \ref{a:logupp} and \ref{a:sparseT}. However, we remark that if $p=1$ (i.e. we are in the setting of uniformly bounded conductances), then things are much easier. Namely, then instead of \eqref{e:estgreen8} we could use the upper bound
\[
0\le G^\A_{V_N(w)}(u^*,v^*)\le\frac{\log N}{2\pi\Lambda^-}
\]
(which follows from Rayleigh monotonicity), that would allow us to take $\T_v=c\Rone_v$. With this choice of $\T_v$, \eqref{e:estgreenquant} would easily be strong enough to obtain exponential (and not just stretched-exponential) tails for $\T_v$, and so we would have a good chance to establish \ref{a:sparseT}.

If $p<1$, though, then \eqref{e:estgreen8} is clearly best-possible as a deterministic estimate. Our idea to establish \ref{a:logupp} and \ref{a:sparseT} now is that we do not actually need a bound like \eqref{e:estgreen8} that holds for all environments, but only a bound that holds for most environments. In other words, we need to understand the upper tail of the random variable $G^\A_{Q_N}(v^*,v^*)$, and in fact we need to show that it is exponentially unlikely in $T$ that this random variable exceeds its typical value ${\log N}/{(2\pi\oA)}$ by more that $T$. Such a quantitative tail estimate appears to be new, and previously only qualitative estimates were known (see e.g. \cite{A15} and the 
references therein, although of course  
stronger quantitative bounds could be deduced from \cite{DG21}).

Turning now to the actual proof, we begin with the asymptotics for the Green's function in half-spaces. Let us first introduce some notation.
For $w\in\Z^2$ and $\vec e\in\{\vec e_1,-\vec e_1,\vec e_2,-\vec e_2\}$ let 
\[Q^{\vec e}(w)=\{v\in\Z^2\colon (v-w)\cdot \vec e\ge 0\}\]
be a half-space, and let $G^\A_{Q^{\vec e}(w)}\colon\CC_\infty\times\CC_\infty\to\R$ be the Green's function on $Q^{\vec e}(w)$. That is, if $v\notin Q^{\vec e}(w)$, then $G^\A_{Q^{\vec e}(w)}(\cdot,v)=0$, while if $v\in Q^{\vec e}(w)$ then $G^\A_{Q^{\vec e}(w)}(\cdot,v)$ is 0 on $\CC_\infty\setminus Q^{\vec e}(w)$ and satisfies $-\Delta_\A G^\A_{Q^{\vec e}(w)}(\cdot,v)=\delta_v$ on $Q^{\vec e}(w)\cap \CC_\infty$ and grows sublinearly at infinity. Equivalently, $G^\A_{Q^{\vec e}(w)}$ is the Green's function for random walk on $\CC_\infty$ killed when exiting $Q^{\vec e}(w)$.

It follows from the general homogenization results in \cite{DG21} (or already from a qualitative invariance principle for the random walk on $\CC_\infty$), that $\PP$-almost surely there is a unique such $G^\A_{Q^{\vec e}(w)}$, and it is non-negative on $\CC_\infty\times\CC_\infty$.

We also denote by
\[v_{w,\vec e}:=v-2((v-w)\cdot \vec e)\vec e\]
the mirror point of $v$ under reflection at $\partial Q^{\vec e}(w)$. Then we have the following asymptotics for $G^\A_{Q^{\vec e}(w)}$ beyond some random scale.
\begin{lemma}\label{l:green_asympt_halfspace}
For each $p>1/2$ and any ${\Lambda^+}/{\Lambda^-}\ge1$ there are $C>0$
and random variables $\Rscale^{\mathrm{Green}''}_v\in\N$ indexed by $v\in\Z^2$ 
with the following property: if $w\in\Z^2$, $\vec e'\in\{\vec e_1,-\vec e_1,\vec e_2,-\vec e_2\}$ and $u,v\in Q^{\vec e}(w)$ satisfy $|u-v|\ge\Rscale^{\mathrm{Green}''}_v$, $(v-w)\cdot \vec e\ge \Rscale^{\mathrm{Green}''}_v$ and we also have $u\in\CC_\infty$ or $(u-w)\cdot \vec e\ge \Rscale^{\mathrm{Green}''}_u$, then
\begin{equation}\label{e:green_asympt_halfspace}
\left|G^\A_{Q^{\vec e}(w)}(u^*,v^*)-\frac{1}{2\pi\oA}\log\left(\frac{|u^*-v_{w,\vec e}|}{|u^*-v|}\right)\right|\le\frac{4}{\oA}.
\end{equation}

In addition, for any $v\in\Z^2$ and any $R,R'\in\N$ with $R\le {R'}/{2}$ there is an event $\Ev^{R,R',\mathrm{Green}''}_v\in\F_{Q_{10R'}(v)}$ such that,
with $s>0$ uniform,
\begin{equation}\label{e:green_asympt_halfspace1}
\PP(\Ev^{R,R',\mathrm{Green}''}_v)\ge 1-C\e^{-R^{s}/C}
\end{equation}
and 
\begin{equation}\label{e:green_asympt_halfspace2}
\PP(\Rscale^{\mathrm{Green}''}_v> R,\Ev^{R,R',\mathrm{Green}''}_v)\le C\e^{-R'^{s}/C}.
\end{equation}
%and $s$ is non-decreasing as a function of $p$ for each fixed $\frac{\Lambda^+}{\Lambda^-}$.
\end{lemma}
The heat kernel asymptotics in \cite{DG21} actually imply even sharper asymptotics for $G^\A_{Q^{\vec e}(w)}$ (with error term $o(1)$ instead of $O(1)$. However the estimate stated in Lemma \ref{l:green_asympt_halfspace} is sufficient for our purposes, and it has the advantage that it follows relatively easily from Lemma \ref{l:green_asympt_improved}. 

\begin{proof}
Our strategy is to use Lemma \ref{l:green_asympt_improved} in combination with the reflection principle. For that purpose we need to know that $\Rscale^{\mathrm{Green}'}$ is small not just at $v$ but also at the relevant mirror point. So we define $\Rscale^{\mathrm{Green}''}$ in such a way that it control $\Rscale^{\mathrm{Green}'}$ at $v$ as well as at the mirror points. To be precise, define
\[\Rscale^{\mathrm{Green}''}_v=\Rscale^{\mathrm{Green}'}_v\vee\max\left\{R\colon \Rscale^{\mathrm{Green}'}_{v_{w',\vec e'}}\le\frac{|v-v_{w',\vec e'}|}{2}\ \forall w',\vec e'\text{ with }R\le \frac{|v-v_{w',\vec e'}|}{2}\right\}\]
and
\[\Ev^{R,R',\mathrm{Green}''}_v=\Ev^{R,R',\mathrm{Green}'}_v\cap\bigcap_{\substack{w',\vec e'\\(w'-v')\cdot \vec e'\in\N\\R\le|w'-v|\le R'/2}}\Ev^{|v-v_{w',\vec e'}|/2,R',\mathrm{Green}'}_{v_{w',\vec e'}}\]
and take $s$ as in Lemma \ref{l:green_asympt_improved}. 

With these definitions, we have
\begin{align*}
\PP(\Ev^{R,R',\mathrm{Green}''}_v)&\ge 1-\PP( (\Ev^{R,R',\mathrm{Green}'}_v)^\complement)-\sum_{\substack{w',\vec e'\\(w'-v')\cdot \vec e'\in\N\\R\le|w'-v|\le R'/2}}\PP( (\Ev^{|v-v_{w',\vec e'}|/2,R',\mathrm{Green}'}_{v_{w',\vec e'}})^\complement)\\
&\ge1-Ce^{-R^s/C}-\sum_{\substack{w',\vec e'\\(w'-v')\cdot \vec e'\in\N\\R\le|w'-v|\le R'/2}}Ce^{-(|v-v_{w',\vec e'}|)^s/C}
\ge1-Ce^{-R^s/C}
\end{align*}
and thus \eqref{e:green_asympt_halfspace1}. The argument for \eqref{e:green_asympt_halfspace2} is analogous. Moreover, since $\Ev^{R,R',\mathrm{Green}'}_v\in\F_{Q_{9R'}(v)}$ and $\Ev^{|v-v_{w',\vec e'}|/2,R',\mathrm{Green}'}_{v_{w',\vec e'}}\in\F_{Q_{9R'}(v_{w',\vec e'})}\subset \F_{Q_{10R'}(v)}$, we see that $\Ev^{R,R',\mathrm{Green}''}_v\in\F_{Q_{10R'}(v)}$.

It remains to show \eqref{e:green_asympt_halfspace}. For this purpose consider $H^\A_{Q^{\vec e}(w)}\colon (\CC_\infty\cap Q^{\vec e}(w))\times (\CC_\infty\cap Q^{\vec e}(w))\to\R$ defined by 
\[H^\A_{Q^{\vec e}(w)}(u'^*,v'^*)=G^\A(u'^*,v'^*)-G^\A_{Q^{\vec e}(w)}(u'^*,(v_{w',\vec e'})^*)-\mathsf{K}'_{v'}+\mathsf{K}'_{v_{w',\vec e'}}.\]
Consider some $u'\in \CC_\infty\cap \partial Q^{\vec e}(w)$. Then 
$|u'-v|\ge(v-w)\cdot\vec e\ge \Rscale^{\mathrm{Green}''}_v\ge\Rscale^{\mathrm{Green}'}_v$, and so by Lemma \ref{l:green_asympt_improved} we know that
\[\left|G^\A(u'^*,v^*)+\frac{1}{2\pi\oA}\log|u'^*-v|-\mathsf{K}'_v\right|\le\frac{1}{\oA|u'^*-v|^{1/2}}\le\frac{1}{\oA\Rscale^{\mathrm{Green}''}_v}.\]
Similarly, $|u'-v_{w,\vec e}|\ge\Rscale^{\mathrm{Green}'}_{v_{w,\vec e}}$, and so by another application of \eqref{e:green_asympt_improved} we have
\[\left|G^\A(u'^*,(v_{w,\vec e})^*)+\frac{1}{2\pi\oA}\log|u'^*-v_{w,\vec e}|-\mathsf{K}'_{v_{w,\vec e}}\right|\le\frac{1}{\oA|u'^*-v_{w,\vec e}|^{1/2}}\le\frac{1}{\oA\Rscale^{\mathrm{Green}''}_v}.\]
Combining the last two estimates and using that $|u'-v_{w,\vec e}|=|u'^*-v|$ by symmetry, we find that
\[\left|H^\A_{Q^{\vec e}(w)}(u'^*,v^*)\right|\le \frac{2}{\oA\Rscale^{\mathrm{Green}''}_v}\le\frac{2}{\oA}.\]
In other words, $H^\A_{Q^{\vec e}(w)}(\cdot,v^*)$ grows sublinearly at infinity and is bounded by $\frac{2}{\oA}$ on $\CC_\infty\cap\partial Q^{\vec e}(w)$.

This means that $G^\A_{Q^{\vec e}(w)}(\cdot,v^*)-H^\A_{Q^{\vec e}(w)}(\cdot,v^*)$ is a $-\Delta_\A$-harmonic function on $\CC_\infty\cap\partial Q^{\vec e}(w)$ that is bounded by $\frac{2}{\oA}$ on $\CC_\infty\cap\partial Q^{\vec e}(w)$ and grows sublinearly at infinity. Thus it must be bounded by $\frac{2}{\oA}$ everywhere. We conclude that
\[\left|G^\A_{Q^{\vec e}(w)}(u^*,v^*)-G^\A(u'^*,v'^*)-G^\A_{Q^{\vec e}(w)}(u'^*,(v_{w',\vec e'})^*)-\mathsf{K}'_{v'}+\mathsf{K}'_{v_{w',\vec e'}}\right|\le\frac{2}{\oA}\]
We can now insert the asymptotics \eqref{e:green_asympt_improved} from Lemma \ref{l:green_asympt_improved} and directly obtain \eqref{e:green_asympt_halfspace}.
\end{proof}

As mentioned in the beginning of the section, we need a replacement for the deterministic upper bound on the Green's function in a (small) box given by \eqref{e:estgreen8}. The following lemma provides such an estimate. The random variable $\T^{\mathrm{Tail},S,\ep}_v$ defined here is genuinely local, and so this time (as already for Lemma \ref{l:dist_to_cluster}) we do not need an estimate regarding its approximate locality.

\begin{lemma}\label{l:tail_bound_green}
For any $\kappa>0$ there is $p^{\mathrm{Tail}}_\kappa<1$ with the following property. For $p^{\mathrm{Tail}}_\kappa\le p\le1$, for any $0<\Lambda^-\le\Lambda^+$ and for any $0<\ep<1$ there are $C>0$ and random variables $\T^{\mathrm{Tail},\ep}_v\in\R$ indexed by $v\in\Z^2$ such that if $R\in\N$ is such that $R^\ep\le \T^{\mathrm{Tail},\ep}_v$ then
\begin{equation}\label{e:tail_bound_green}
  \max_{u,u'\in Q_R(v)}G^\A_{Q_R(v)}(u^*,u'^*)\le \frac{\T^{\mathrm{Tail},\ep}_v}{2\pi\Lambda^+}.
\end{equation}
Furthermore, for all $T\in\R$ we have the tail bound 
\begin{equation}\label{e:tail_bound_green1}
\PP(\T^{\mathrm{Tail},\ep}_v\le T,\Ev^{T^{1/\ep},\mathrm{Clust}}_v)\ge 1-C\e^{-\kappa T\Lambda^-/\Lambda^+},
\end{equation}
and the event $\{\T^{\mathrm{Tail},\ep}_v\le T\}\cap\Ev^{T^{1/\ep},\mathrm{Clust}}_v$ is $\F_{Q_{9T^{1/\ep}}(v)}$-measurable.
\end{lemma}

Note that (by Rayleigh monotonicity) we have $\oA\le\Lambda^+$. So this lemma will allow us to control $2\pi\oA G_{Q_R(v)}(u^*,u'^*)$ by $\T^{\mathrm{Tail},\ep}_v$.

\begin{proof}
For any $\tau\in\N_\ep:=\{n^{\ep}\colon n\in\N\}$ consider the event
\[\Ev^{\tau,\mathrm{Tail},\ep}_v=\left\{\max_{R\le \tau^{1/\ep}}\max_{u,u'\in Q_R(w)}G^\A_{Q_R(w)}(u^*,u'^*)\le \frac{\tau}{2\pi\Lambda^+}\right\}.\]
This event only depends on $\CC_\infty\cap Q_{\tau^{1/\ep}}(v)$, and so $\Ev^{\tau,\mathrm{Tail},\ep}_v\cap \Ev^{\tau^{1/\ep},\mathrm{Clust}}_v$ is $\F_{Q_{9\tau^{1/\ep}}(v)}$-measurable. We can now define
\[\T^{\mathrm{Tail},\ep}_v=\inf\left\{\tau\in\N_\ep\colon \Ev^{\tau,\mathrm{Tail},\ep}_v\right\}.\]
From Lemma \ref{l:local_approx_cluster} we know that
\[\PP(\Ev^{\tau^{1/\ep},\mathrm{Clust}}_v)\ge 1-C\e^{-\tau^{1/\ep}/C}.\]
Now, if we also knew that
\begin{equation}\label{e:tail_bound_green3}
  \PP(\Ev^{\tau,\mathrm{Tail},\ep}_v)\ge 1-C\e^{-\kappa\tau\Lambda^-/(2\Lambda^+)},
\end{equation}
then we would have
\[\PP\left(\Ev^{\tau,\mathrm{Tail},\ep}_v\cap \Ev^{\tau^{1/\ep},\mathrm{Clust}}_v\right)\ge 1-C\e^{-T^{1/\ep}/C}-C\e^{-3\kappa\tau\Lambda^-/(4\Lambda^+)}\ge1-C\e^{-\kappa\tau\Lambda^-/\Lambda^+}\]
for a constant $C$ depending on ${\Lambda^+}/{\Lambda^-}$, $\ep$, from which \eqref{e:tail_bound_green1} easily follows.

So it remains to prove \eqref{e:tail_bound_green3}. Note first that
\[\max_{u,u'\in Q_R(v)}G^\A_{Q_R(v)}(u^*,u'^*)\le \max_{u\in \CC\infty\cap Q_R(v)}G^\A_{Q_R(v)}(u,u),\]
and so \eqref{e:tail_bound_green3} follows from a union bound if we show
\begin{equation}\label{e:tail_bound_green4}
  \PP\left(G^\A_{Q_R(v)}(u,u)>\frac{\tau}{2\pi\Lambda^+}\right)\le C\e^{-\kappa\tau\Lambda^-/(4\Lambda^+)}\quad\forall R\le\tau^{1/\ep},u\in \CC_\infty \cap Q_R(v).
\end{equation}

Our proof strategy for \eqref{e:tail_bound_green4} follows closely the approach in \cite{BK05}, where the super-levelsets of $G^\A_{Q_R(v)}(\cdot,u)$ are studied and their size is related to their isoperimetry. We use this exact strategy on larger scales (once the level sets have grown to size $\ge cT^2$). On smaller scales we use another argument, namely that the difference of the Green $G^\A_{Q_R(v)}(\cdot,u)$ on two different points is bounded by their chemical distance. In this manner we can efficiently find a large super-level set at height close to $v$.

To make this precise, fix for the moment some $u\in\CC_\infty\cap Q_R(v)$. As in \cite{BK05} and as in the argument that showed \eqref{e:estgreen8}, it is convenient to use the viewpoint of electrical network theory. That is, we consider $G^\A_{Q_R(v)}(\cdot,u)$ equivalently as the voltage distribution that arises when a unit current flows from $u$ through the electrical network given by the conductances $\A$. Then the voltage drop along each edge is at most ${1}/{\Lambda^-}$. Moreover, for each $A\subset\CC_\infty$ with $u\in A\subset Q_R(v)$ the total current flowing out of $A$ is equal to 1.
Let $n=|\CC_\infty\cap Q_R(v)|$. For $k\in\N$ let
\[\theta(k):=\sup_{\substack{B\subset \CC_\infty\\B\text{ connected}\\|B|=k}}\min_{\bar u\in B}G^\A_{Q_R(v)}(\bar u,u)\] and let $A_k$ be the optimizer (with ties broken in some deterministic manner). By the maximum principle, $G^\A_{Q_R(v)}(\bar u,u)\le \theta(k)$ for all $\bar u\in\CC_\infty\setminus A_k$.

Trivially, $\theta(k)$ is non-increasing in $k$, and $\theta(k)=0$ for $k\ge n+1$. Our goal will be to bound the decrements of $\theta(k)$, assuming some good events, which we define next.

Let
\begin{align*}
\Ev^{\tau,(I)}_u&=\left\{|\CC_\infty\cap Q_{\tau\Lambda^-/(100\pi\Lambda^+)}(u)|\ge\frac{\tau^2(\Lambda^-)^2}{20000\pi^2(\Lambda^+)^2}\right\},\\
\Ev^{\tau,(II)}_u&=\left\{\max_{u'\in \CC_\infty\cap Q_{\tau\Lambda^-/(100\pi\Lambda^+)}(u)}\dist_{\CC_\infty}(u,u')\le \frac{\tau\Lambda^-}{4\pi\Lambda^+}\right\},\\
\Ev^{\tau,(III)}_u&=\left\{\exists A\subset\CC_\infty\colon u\in A,|A|\ge \frac{\tau^2(\Lambda^-)^2}{20000\pi^2(\Lambda^+)^2},|\partial^+A|\le |A|^{1/2}\right\}^\complement.
\end{align*}
According to Lemma \ref{l:density_cluster_highlysc}, Lemma \ref{l:isoperimetry} and Lemma \ref{l:chemical_distance} we can make the probability of each of these three events larger than $1-C\e^{-\kappa\tau\Lambda^-/(4\Lambda^+)}$, if we choose $p$ close enough to 1 (depending on $\kappa$ only).

On the event $\Ev^{\tau,(I)}_u\cap \Ev^{\tau,(II)}_u$ there are at least 
${\tau^2(\Lambda^-)^2}/{(20000\pi^2(\Lambda^+)^2})$ points in $\CC_\infty$
at graph distance at most ${\tau\Lambda^-}/{(4\pi\Lambda^+)}$ from $u$. Call the set of these points $A$. As the voltage drop along each edge in $\CC_\infty$ is at most ${1}/{\Lambda^-}$, the voltage difference between $u$ and any point in $A$ is at most $({1}/{\Lambda^-})\cdot({\tau\Lambda^-}/{4\pi\Lambda^+})=
{\tau}/{(4\pi\Lambda^+)}$. In other words,
\[\min_{\bar u\in A_k}G^\A_{Q_R(v)}(\bar u,u)\ge G^\A_{Q_R(v)}(u,u)-\frac{\tau}{4\pi\Lambda^+}.\]
As $A$ is also clearly connected, we have shown that on $\Ev^{\tau,(I)}_v\cap \Ev^{\tau,(II)}_v$ we have
\begin{equation}\label{e:tail_bound_green5}
\theta\left(\left\lceil\frac{\tau^2(\Lambda^-)^2}{10000\pi^2(\Lambda^+)^2}\right\rceil\right)\ge G^\A_{Q_R(v)}(u,u)-\frac{\tau}{4\pi\Lambda^+}.
\end{equation}

On the other hand, assume for the moment that $\Ev^{\tau,(III)}_u$ holds. We can now directly reuse the argument from \cite{BK05} on large scales. Namely, let $k\ge{\tau^2(\Lambda^-)^2}/{(20000\pi^2(\Lambda^+)^2)}$. If $\theta(k)>0$ then $A_k\subset Q_R(u)$, so the current flowing out of $A_k$ is exactly 1. On the other hand, on $\Ev^{\tau,(III)}_u$ there are at least $\sqrt{k}$ vertices in $\partial^+A_k$. So by the pigeon-hole principle for at least ${\sqrt{k}}/{2}$ of the vertices in $\partial^+A_k$ the total current flowing from $A$ into each of them is at most ${2}/{\sqrt{k}}$. Let $\tilde A_k$ be the set of these vertices. For each vertex in $\tilde A_k$ the voltage drop from $A_k$ can be at most ${2}/{(\sqrt{k}\Lambda^-)}$. This means that
\[\min_{\bar u\in A_\cup \tilde A_k}G^\A_{Q_R(v)}(\bar u,u)\ge \theta(k)-\frac{2}{\sqrt{k}\Lambda^-}\]
and as $A_k\cup\tilde A_k$ is clearly connected and has at least $k+{\sqrt{k}}/{2}$ elements, we have shown that on $\Ev^{\tau,(III)}_u$, if $k\ge{\tau^2(\Lambda^-)^2}/{(20000\pi^2(\Lambda^+)^2)}$ and $\theta(k)>0$ then
\begin{equation}\label{e:tail_bound_green6}
\theta\left(\left\lceil k+\frac{\sqrt{k}}{2}\right\rceil\right)\ge\theta(k)-\frac{2}{\sqrt{k}}.
\end{equation}
The same estimate also holds trivially if $\theta(k)=0$ so we can remove that condition.
We can now iterate \eqref{e:tail_bound_green6} $2\sqrt{k}$ times to obtain
that for such $k$,
$\theta(2k)\ge \theta(k)-{C}/{\Lambda^-}$. We also know that $\theta(n+1)=0$, and so we see that on $\Ev^{\tau,(III)}_u$ we have
\begin{equation}\label{e:tail_bound_green7}
\begin{split}
\theta\left(\left\lceil\frac{\tau^2(\Lambda^-)^2}{20000\pi^2(\Lambda^+)^2}\right\rceil\right)&\le C\frac{\log n}{\Lambda^-}
\le C\frac{\log R^2}{\Lambda^-}
\le C\frac{\log \tau}{\ep\Lambda^-},
\end{split} 
\end{equation}
where $C$ is an absolute constant.

Now we are almost done. Combining \eqref{e:tail_bound_green5} and \eqref{e:tail_bound_green7}, we showed that on the event $\Ev^{\tau,(I)}_u\cap \Ev^{\tau,(II)}_u\cap \Ev^{\tau,(III)}_u$ we have
\[G^\A_{Q_R(v)}(u,u)\le \frac{\tau}{4\pi\Lambda^+}+C\frac{\log \tau}{\ep\Lambda^-}.\]
For $\tau$ large enough the right-hand side here is bounded by ${\tau}/{2\pi\Lambda^+}$, and so we have shown \eqref{e:tail_bound_green4} when $\tau\ge C'$ for some constant $C'$ (which depends on $\Lambda^+/{(\ep\Lambda^-)}$, but not on $\kappa$ or $p$). But this is enough, as for $\tau< C'$ the estimate \eqref{e:tail_bound_green4} trivially holds if we choose the constant there large enough.
\end{proof}

We can now turn to showing Assumptions \ref{a:logupp} and\ref{a:sparseT}. In this proof, we will have to make the $p$-dependence of various quantities explicit. %Let (as in Theorem \ref{t:perc_unit_cond}) denote by $a_p$  the value of $\oA$ when we only consider unit conductances, i.e. $\A$ is sampled from true Bernoulli percolation with parameter $p$. Then in the general case we have the inequality
%\begin{equation}\label{e:homogenized_coeff}
%\Lambda^- a_p \le\oA\le\Lambda^+ a_p  
%\end{equation}
%Moreover, $ a_p $ is a non-decreasing function of $p$. This is not completely trivial, but it follows for example from the variational characterization of $\oA$ in \cite[Equation (179)]{DG21} together with the well-known that one couple Bernoulli percolation processes with different $p$ in such a way that the corresponding infinite clusters are nested.

\begin{proof}[Proof of Theorem \ref{t:percolation_cluster}, second part]$ $

\emph{Step 1: Preliminaries}\\
In principle we would like to take $\T_v=\T^{\mathrm{Tail},\ep}_v$. However there are several other rare events that we need to account for: On the one hand, if $\Rscale^{\mathrm{Green}''}_v\ge T^{1/\ep}$, then $\T^{\mathrm{Tail},\ep}_v$ no longer controls $G^\A_{Q_{\Rscale^{\mathrm{Green}''}_v}}(v,v)$. On the other hand, we also need control over $\Rscale^{\mathrm{Dist}}_v$, which $\T^{\mathrm{Tail},\ep}_v$ does not provide. So the actual definition of $\T_v$ is a little involved.

The exponent $s$ in Lemma \ref{l:green_asympt_halfspace} is 
uniformly in $p$ bounded away from $1/2$. So there is $s_*>0$ such that $s\ge s_*$ for any $p\ge3/4$, say. We fix such an $s_*$, and choose $\ep={s_*}/{2}$. We also take $\kappa={3\Lambda^+}/{\Lambda^-}$, and define $p_{\Lambda^+/\Lambda^-}=p^{\mathrm{Tail}}_\kappa\vee\frac34$. This means that for $p\ge p_{\Lambda^+/\Lambda^-}$ the decay rate in \eqref{e:tail_bound_green1} is $\e^{-3T}$.

We now define 
\begin{align*}
\tilde\T_v&=2\pi\Lambda^+\max_{u,u'\in\CC_\infty\cap Q_{\Rscale^{\mathrm{Green}''}_v}(v)}G_{Q_{\Rscale^{\mathrm{Green}''}_v}}(u^*,u'^*)\vee(\Rscale^{\mathrm{Dist}}_v)^{\ep},\\
\T_v&=\tilde\T_v+C',\\
\Ev^{T,T'}_v&=\left\{\T^{T,\mathrm{Tail},\ep}_v\le T\right\}\cap\Ev^{T^{1/\ep},\mathrm{Clust}}_v\cap\Ev^{T^{1/\ep},T'^{1/\ep},\mathrm{Green}''}_v\cap\left\{\Rscale^{\mathrm{Dist}}_v\le T^{1/\ep}\right\},%\cap \Ev^{T^{1/\ep},\mathrm{Clust}}_v
\end{align*}
where the objects here are as in Lemma \ref{l:local_approx_cluster}, Lemma \ref{l:dist_to_cluster}, Lemma \ref{l:green_asympt_halfspace} and Lemma \ref{l:tail_bound_green}, and $C'$ is a constant that depends on ${\Lambda^+}/{\Lambda^-}$ only and will be fixed later.

The point of these definitions is that we want $\T_v$ to control the Green's function on lengthscale $\Rscale^{\mathrm{Green}''}_v$. As on lengthscales larger than that we have the estimates from Lemma \ref{l:green_asympt_halfspace}, we will be able to control the Green's function on all scales.
Furthermore $\Ev^{T,T'}_v$ is supposed to provide a local event that approximates $\{\T_v\le T\}$. Of the two terms to approximate in the definition of $\tilde\T_v$, the tricky one is the first one. There we use that if $\Rscale^{\mathrm{Green}''}_v\le T^{1/\ep}$ then the term is controlled by $\T^{T,\mathrm{Tail},\ep}_v$ (which is local once we intersect with $\Ev^{T^{1/\ep},\mathrm{Clust}}_v$), while if $\Rscale^{\mathrm{Green}''}_v> T^{1/\ep}$ we use that $\Rscale^{\mathrm{Green}''}_v$ itself is asymptotically local.

\emph{Step 2: Verification of \ref{a:logupp}}\\
We want to show that with $\T_v$ as defined above we have \ref{a:logupp}. The argument will be based on the maximum principle (similar to the argument that lead to \ref{a:logbd}).

In view of \eqref{e:estgreen} we have to prove that
\begin{align}
2\pi\oA G^\A_{V_N(w)}(v^*,v^*)&\le \log N+\tilde\T_v+C\label{e:estgreenT1}\\
2\pi\oA \left(G^\A_{V_N(w)}(v^*,v^*)-G^\A_{V_N(w)}(u^*,v^*)\right)&\le \log_+|u-v|+\tilde\T_u+\tilde\T_v+C\label{e:estgreenT2}
\end{align}
hold true for all $u,v\in V_N(w)$ (then \ref{a:logupp} follows once we take the constant $C'$ sufficiently large). We will begin by showing a stronger version of \eqref{e:estgreenT1}, which we will also use in the proof of \eqref{e:estgreenT2}. Thus, we claim that
\begin{equation}\label{e:estgreenT3}
2\pi\oA G^\A_{V_N(w)}(v^*,v^*)\le \log \left(1+\frac{\dist(v,\partial^+ V_N(w))}{\Rscale^{\mathrm{Green}''}_v}\right)+\tilde\T_v+C.
\end{equation}
To see this, we use the maximum principle in a similar manner as for \ref{a:logbd}. We can assume without loss of generality that the side of $V_N(w)$ closest to $v$ is the bottom side. The maximum principle on $V_N(w)$ implies that 
\[G^\A_{V_N(w)}(\cdot,v^*)\le G^\A_{Q^{\vec e_1}(w)}(\cdot,v^*)\]
on $\CC_\infty\cap V_N(w)$. We know that for $u\in \CC_\infty\cap\partial Q_{\Rscale^{\mathrm{Green}''}_v}(v)$ we have $|u^*-v|=\Rscale^{\mathrm{Green}''}_v$ and $|u^*-v_{w,\vec e_1}|\le \Rscale^{\mathrm{Green}''}_v+2\dist(v,\partial^+ V_N(w))$, and so Lemma \ref{l:green_asympt_halfspace} implies that
\begin{equation}\label{e:estgreenT4}
G^\A_{V_N(w)}(u^*,v^*)\le \frac{4}{\oA}+\frac{1}{2\pi\oA}\log\left(1+\frac{2\dist(v,\partial^+ V_N(w))}{\Rscale^{\mathrm{Green}''}_v}\right)\le \frac{C}{\oA}+\frac{1}{2\pi\oA}\log \left(1+\frac{\dist(v,\partial^+ V_N(w))}{\Rscale^{\mathrm{Green}''}_v}\right)
\end{equation}
for $u^*\in \CC_\infty\cap Q^{\vec e_1}(w)\cap \partial Q_{\Rscale^{\mathrm{Green}''}_v}(v)$.

We can now apply the maximum principle once again, namely on $\CC_\infty\cap Q^{\vec e_1}(w)\cap \partial Q_{\Rscale^{\mathrm{Green}''}_v}(v)$ with comparison function $G^\A_{Q^{\vec e_1}(w)\cap \partial Q_R(v)}$ and conclude from \eqref{e:estgreenT4} that
\[G^\A_{V_N(w)}(v^*,v^*)\le \frac{C}{\oA}+\frac{1}{2\pi\oA}\log \left(1+\frac{\dist(v,\partial^+ V_N(w))}{\Rscale^{\mathrm{Green}''}_v}\right)+G_{Q_{\Rscale^{\mathrm{Green}''}_v}}(v^*,v^*).\]
Because $\oA\le\Lambda^+$, this implies \eqref{e:estgreenT3} (if $C'$ in the definition of $\T_v$ is taken sufficiently large). 

Having showed \eqref{e:estgreenT3} (and in particular also \eqref{e:estgreenT1}), we now turn to \eqref{e:estgreenT2}. The argument here is similar. We first note that if $\Rscale^{\mathrm{Green}''}_v>\dist(v,\partial^+ V_N(w))$, then \eqref{e:estgreenT2} follows from \eqref{e:estgreenT3} and the trivial lower bound $G^\A_{V_N(w)}(u^*,v^*)\ge0$. So we can assume that $\Rscale^{\mathrm{Green}''}_v\le\dist(v,\partial^+ V_N(w))$. In that case we need a nontrivial lower bound for $G^\A_{V_N(w)}(u^*,v^*)$.

We claim that
\begin{equation}\label{e:estgreenT5}
G^\A_{V_N(w)}(u^*,v^*)\ge\frac{1}{2\pi\oA}\log \left(\frac{\dist(v,\partial^+ V_N(w))}{|u^*-v|\vee\Rscale^{\mathrm{Green}''}_v}\right)-\frac{C}{\oA}
\end{equation}
To see this, we use once again the maximum principle. Note that by Lemma \ref{l:green_asympt_halfspace} we have $G^\A_{Q^{\vec e_1}(w)}(\cdot,v^*)\le \frac{C}{\oA}$ on $\CC_\infty\setminus Q_{\Rscale^{\mathrm{Green}''}_v}(v)\supset \CC_\infty\cap\partial Q_N(w)$, and so the maximum principle on $\CC_\infty\cap V_N(w)$ implies that
\begin{equation}\label{e:estgreenT6}
G^\A_{V_N(w)}(\cdot,v^*)\ge G^\A_{Q^{\vec e_1}(w)}(\cdot,v^*)-\frac{C}{\oA}
\end{equation}
on $\CC_\infty\cap V_N(w)$. If $|u^*-v|\ge\Rscale^{\mathrm{Green}''}_v$ this directly implies \eqref{e:estgreenT5}, while if $|u^*-v|\le\Rscale^{\mathrm{Green}''}_v$ we obtain from \eqref{e:estgreenT6} and Lemma \ref{l:green_asympt_halfspace} that 
\[G^\A_{V_N(w)}(\cdot,v^*)\ge \frac{1}{2\pi\oA}\log \left(\frac{\dist(v,\partial^+ V_N(w))}{\Rscale^{\mathrm{Green}''}_v}\right)-\frac{C}{\oA}\]
 on $\CC_\infty\cap\partial Q_{\Rscale^{\mathrm{Green}''}_v}(w)$, from which \eqref{e:estgreenT5} follows by an application of the maximum principle on $\CC_\infty\cap Q_{\Rscale^{\mathrm{Green}''}_v}(w)$.
 
We can now combine \eqref{e:estgreenT3} and \eqref{e:estgreenT5} into
\begin{equation}\label{e:estgreenT7}
G^\A_{V_N(w)}(v^*,v^*)-G^\A_{V_N(w)}(u^*,v^*)\le \frac{1}{2\pi\oA}\log\left(\frac{|u^*-v|\vee \Rscale^{\mathrm{Green}''}_v}{\Rscale^{\mathrm{Green}''}_v}\right)+\tilde\T_v+\frac{C}{\oA}.
\end{equation}
This already looks similar to \eqref{e:estgreenT2}. In fact, if $|u^*-v|\le\Rscale^{\mathrm{Green}''}_v$ it directly implies \eqref{e:estgreenT2}, so we can assume the opposite. Then we still to make sure that $|u-v|$ is not much smaller that $|u^*-v|$. For that purpose we estimate that $|u^*-u|=\Rscale^{\mathrm{Dist}}_u\le(\tilde\T_u)^{1/\ep}$
and therefore 
\[\log|u^*-v|\le \log\left(|u-v|+(\tilde\T_u)^{1/\ep}\right)\le\log|u-v|+\frac{1}{\ep}\log \tilde\T_u.\]
Inserting this into \eqref{e:estgreenT7}, we finally obtain \eqref{e:estgreenT2}.

\emph{Step 3: Verification of \ref{a:sparseT}}\\
It remains to show that our choice of $\T_v$ also satisfies \ref{a:sparseT}. As $\T$ and $\tilde T$ only differ by a constant, it suffices to check that $\tilde T$ satisfies \ref{a:sparseT}. This is similar to the proof of \ref{a:sparseR}, so we will be brief here.

We claim that
\begin{align}
\PP(\Ev^{T,T'}_v)&\ge 1-C\e^{-3T},\label{e:estgreenT8}\\
\PP(\tilde\T_v> T,\Ev^{T,T'}_v)&\le C\e^{-3T'}.\label{e:estgreenT9}
\end{align}
Indeed, \ref{l:tail_bound_green} and our choice of $\kappa$ imply that
\[\PP(\T^{\mathrm{Tail},\ep}_v\le T,\Ev^{T^{1/\ep},\mathrm{Clust}}_v)\ge 1-C\e^{-\kappa T\Lambda^-/\Lambda^+}.\]
Moreover, Lemma \ref{l:green_asympt_halfspace} implies that
\[\PP(\Ev^{T^{1/\ep},T'^{1/\ep},\mathrm{Green}''}_v)\ge 1-C\e^{-T^2/C},\]
and Lemma \ref{l:dist_to_cluster} implies a similar estimate for $\left\{\Rscale^{\mathrm{Dist}}_v\le T^{1/\ep}\right\}\cap \Ev^{T^{1/\ep},\mathrm{Clust}}_v$. Combining these results, we find
\[\PP(\Ev^{T,T'}_v)\ge 1-C\e^{-3T}-C\e^{-T^2/C},\]
which implies \eqref{e:estgreenT8} if we choose the constant there sufficiently large (depending on $\frac{\Lambda^+}{\Lambda^-}$ and $p$).

Regarding \eqref{e:estgreenT9}, observe that if $\tilde\T_v>T$ and $\Ev^{T,T'}_v$ occurs, then by definition of $\T^{\mathrm{Tail},\ep}_v$ we must have $\Rscale^{\mathrm{Green}''}_v\ge \T^{1/\ep}$. But $\Ev^{T,T'}_v$ also implies that $\Ev^{T^{1/\ep},T'^{1/\ep},\mathrm{Green}''}_v$ occurs, and according to Lemma \ref{l:green_asympt_halfspace} we have
\[\PP(\Rscale^{\mathrm{Green}''}_v\ge \T^{1/\ep},\Ev^{T^{1/\ep},T'^{1/\ep},\mathrm{Green}''}_v)\le C\e^{-T^2/C}.\]
From this we obtain \eqref{e:estgreenT9}. Moreover, $\Ev^{T,T'}_v$ is $\F_{Q_{10T'^{1/\ep}}(v)}$-measurable.

Now the proof of \ref{a:sparseT} proceeds just like the proof of \ref{a:sparseR}. We choose $T'=\left(\frac{N}{L}\right)^{1/(2\ep)}$. If $\tilde T_v\ge T$ then the local event $Ev^{T,T'}_v$ or a very rare event occurs. The latter events can be controlled by a union bound, and the former events using Hoeffding's inequality and the locality.
\end{proof}

\begin{remark}\label{r:nearcritical}
Lemma \ref{l:tail_bound_green} is the main step in the verification of \ref{a:logupp} and \ref{a:sparseT}. In view of Question \ref{q:nearcritical}, let us explain why this proof cannot work for $p$ close to $1/2$. For convenience, we only consider the (easiest) case where $\Lambda^+=\Lambda^-$, and recall that we denoted the corresponding $\oA$ by $a_p$ in that case.

Roughly speaking, we would need to prove that for a generic $v\in V_N(w)$ we have
\[\PP\left(\frac{1}{2\pi a_p}G^\A_{V_N(w)}(v^*,v^*)\ge\frac{1}{2\pi a_p}(\log N+T)\right)\le C_p\e^{-(2+\ep)T}\]
for some $\ep>0$ or equivalently,
\begin{equation}\label{e:near_crit}
\PP\left(\frac{1}{2\pi a_p}G^\A_{V_N(w)}(v^*,v^*)\ge\frac{\log N}{2\pi a_p}+T\right)\le C_p\e^{-ca_pT}
\end{equation}
for a constant $c>4\pi$. There is a competition here: As $p\searrow1/2$, bad environments such as long pipes (which contribute variance of order of the length of the pipe)
become more likely, so, for fixed $T$, the left-hand side should be increasing. 
On the other hand, $a_p$ goes to 0 as $p\searrow1/2$, so the right-hand side is also increasing as $p\searrow1/2$.

In fact, critical scaling theory and numerical estimates suggest that $a_p\sim\left(p-1/2\right)^t$ as $p\searrow1/2$, where $t\approx\frac43\cdot0.9826\approx1.310$, see \cite{G99}. This means that as $p\searrow 1/2$,
we have a good (albeit non-rigorous) prediction of the behavior of the 
right-hand side of \eqref{e:near_crit}. Unfortunately, it is unclear to us 
even heuristically how the left-hand side of \eqref{e:near_crit} 
behaves as $p\searrow1/2$, and whether a variant of our proof of Lemma \ref{l:tail_bound_green} could be made to work for $p$ close to $1/2$. If one wanted 
to extend our proof, one would need to understand the 
large-deviation behavior of the isoperimetric profile and of 
the volume of balls with respect to the 
chemical distance for $p$ close to $1/2$, and both tasks seem difficult.
In fact, it might well be that our estimate of the effective resistance
by the chemical distance is already too crude, and so new 
ideas might be necessary. 
%The main difficulty seems to be in estimating the probability of the effective resistance being really of the order of the chemical distance, and in obtaining precise scaling behavior of the latter as $p\searrow1/2$.

\end{remark}

\subsubsection*{Acknowledgements} This project has received funding from the European Research Council (ERC) under the European Union's Horizon 2020 research and innovation programme (grant agreement No. 692452). Florian Schweiger is also supported by the Foreign Postdoctoral Fellowship Program of the Israel Academy of
Sciences and Humanities.

\small
%\bibliographystyle{alpha_edited2}
%\bibliography{MaximumGFFRandomConductances}

\end{document}